\documentclass[preprint,3p]{elsarticle}
\makeatletter
\def\ps@pprintTitle{%
 \let\@oddhead\@empty
 \let\@evenhead\@empty
 \def\@oddfoot{}%
 \let\@evenfoot\@oddfoot}
\makeatother
\usepackage{amsfonts}
\usepackage{amsmath}
\usepackage{amssymb}
\usepackage{amsthm}
\usepackage{array}
\usepackage{booktabs}
\usepackage{color}
\usepackage{enumerate}
\usepackage{esint}
\usepackage{lineno,hyperref}
\usepackage{graphicx}
\usepackage{relsize}
\usepackage{subfigure}
\usepackage{times}
\usepackage{color}

\newcommand{\pp}{\mathfrak{p}}

\newcommand{\tang}{ \boldsymbol{\tau} }
\newcommand{\sang}{ \boldsymbol{\sigma} }
\newcommand{\dang}{\boldsymbol{\eta}}

\newcommand{\pprime}{\prime \prime}
\newcommand{\ppprime}{\prime \prime \prime}

\newcommand{\Sd}{\mathcal{S}^{2}}
\newcommand{\tangc}{ \tang_{ {\rm c}} }
\newcommand{\normc}{ \dang_{ {\rm c}} }
\newcommand{\binc}{ \boldsymbol{ \mathfrak{b} }_{ {\rm c}} }

\newcommand{\pn}{\psi^{ {\rm n}}}
\newcommand{\pt}{\psi^{ {\rm t}} }
\newcommand{\pb}{\psi^{ {\rm b}} }

\newcommand{\veps}{\varepsilon}

\newcommand{\R}{\mathbb{R}}
\newcommand{\Z}{\mathbb{Z}}

\newcommand{\x}{\mathbf{x}}
\newcommand{\y}{\mathbf{y}}
\newcommand{\z}{\mathbf{z}}
\newcommand{\vv}{\mathbf{v}}

\newcommand{\e}{\mathbf{e}}
\newcommand{\rd}{\mathrm{d}}
\newcommand{\dd}{\vartheta}
\newcommand{\re}{\mathrm{e}}

\newcommand{\K}{\mathfrak{K}}

\newcommand{\T}{\mathbb{T}^1}
\newcommand{\F}{\mathbf{F}}
\newcommand{\pv}{\mathrm{p.v.}}
\newcommand{\gammat}{ \gamma_{\tang}}
\newcommand{\N}{\mathbb{N}}
\newcommand{\ac}{\alpha_{ {\rm c}} }

\newcommand{\cH}{\mathcal{H}^{1}}

\newcommand{\lenG}{\mathrm{len}(\gamma)}
\newtheorem{theorem}{Theorem}[section]
\newtheorem{lemma}[theorem]{Lemma}

\newtheorem{proposition}[theorem]{Proposition}
\newtheorem{corollary}[theorem]{Corollary}

\begin{document}

\begin{frontmatter}

\title{Dynamics of Embedded Curves by Doubly-Nonlocal Reaction-Diffusion Systems}


\author{James H. von Brecht\footnote{This work was supported by National Science Foundation grant DMS-1312344/DMS-1521138.} and Ryan Blair}

\begin{abstract}
{We study a class of nonlocal, energy-driven dynamical models that govern the motion of closed, embedded curves from both an energetic and dynamical perspective. Our energetic results provide a variety of ways to understand physically motivated energetic models in terms of more classical, combinatorial measures of complexity for embedded curves. This line of investigation culminates in a family of complexity bounds that relate a rather broad class of models to a generalized, or weighted, variant of the crossing number. Our dynamic results include global well-posedness of the associated partial differential equations, regularity of equilibria for these flows as well as a more detailed investigation of dynamics near such equilibria. Finally, we explore a few global dynamical properties of these models numerically.}

\end{abstract}

\end{frontmatter}


\section{Introduction}

In a variety of physical and mathematical contexts, a curvature-regularized nonlocal interaction energy
\begin{equation}\label{eq:introenergy}
E(\gamma) := \int_{\gamma} \kappa^2_{\gamma}(\x) \, \rd \cH(\x) + \int_{\gamma \times \gamma} K\big( \mathrm{dist}^2_{ \R^3}(\x,\y),\mathrm{dist}^2_{\gamma}(\x,\y)\big) \, \rd \cH(\x) \otimes \rd \cH(\y)
\end{equation}
governs the motion of a closed, embedded curve. For example, worm-like chain (WLC) models for DNA dynamics incorporate the bending energy (i.e. the squared curvature $\kappa^2_\gamma$) and an energetic barrier to self-intersection (i.e. a repulsive kernel $K(u,v)$) that prevents topological changes. Dynamical models of {protein folding \cite{CC} and vortex filaments \cite{CI}} also take a similar form. At the purely mathematical end of the spectrum, a relatively recent trend in geometric knot theory has witnessed a growth in the energetic study of \eqref{eq:introenergy} and its variants \cite{LS10,GM99,SV13,BarXiv1,BarXiv2}. A significant portion of this general line of work focuses on extremal properties of these various knot energies such as the existence, complexity and regularity of extremal embeddings within a given class of curves. In the context of \eqref{eq:introenergy}, the bi-variate kernel $K(u,v)$ encodes the interaction between pairs of points $(\x,\y)$ on a curve $\gamma \in \R^3$ according to the principles of physical potential theory --- an attraction or repulsion between points $\x,\y \in \gamma$ generically occurs in regions where $K(u,v)$ increases or decreases in its first argument, respectively. A typical kernel $K(u,v)$ will exhibit a singular, short-range repulsion and possibly an additional regular, far-field attraction, yet even this reasonable level of generality in the choice of kernel $K(u,v)$ leads to a class of models that exhibit a striking variety of complex and intricate dynamics.

We study the energetics and dynamics of \eqref{eq:introenergy} from both an analytical and a numerical perspective. At an energetic level, we establish a set of inequalities that relate the nonlocal energy \eqref{eq:introenergy} to classical, combinatorial measures of complexity for embedded curves. These results prove similar, in spirit, to a variety of work on ``knot energies'' that emanates from the knot theory community \cite{HS11,BR14,M,FHW,CKS,BS}. Our arguments reveal the sense in which some combinatorial measures of complexity, such as the average crossing number, have fractional Sobolev spaces, lurking just underneath the surface. This insight allows us to establish complexity bounds for a rather broad class of kernels $K(u,v),$ and it also provides the impetus to introduce generalizations of such complexity measures. {\color{black} In this sense, our work complements that strand of geometric knot theory that emphasizes the importance of taking a more analytic approach to \eqref{eq:introenergy} and its variants \cite{BR15,BR13,BR152}. For instance, the relative importance of harmonic analysis vis-\`{a}-vis M\"{o}bius invariance has been observed when studying {\color{black} regularity of extremal embeddings \cite{BRS}}.} At a dynamic level, we study a corresponding ``constrained gradient flow'' of the energy \eqref{eq:introenergy} that restricts the dynamics to lie in the class of unit speed parameterizations. This is a slightly non-standard approach since we do not simply functionally differentiate \eqref{eq:introenergy} with respect to $\gamma$, but it leads to an analytically and numerically well-behaved dynamical model. In particular, enforcing a uniform density along $\gamma$ at each time yields semi-linear rather than quasi-linear dynamics and also gives uniform numerical samples along $\gamma$ at each numerical time-step without having to resort to re-parameterizations. At a physical level, we may view our choice of dynamics as a hard limit of WLC models that include energetic barriers to bond stretching, in the sense that an infinite barrier to bond stretching results from constraining the flow to lie in the class of unit speed curves. This approach to dynamics leads to a class of nonlocal reaction-diffusion systems driven by \emph{doubly-nonlocal} forcings. The pairwise interaction energy leads to a forcing $\mathbf{F}_{\gamma}$ whose derivative $f_{\gamma}$ along $\gamma$ satisfies
\begin{equation*}
\mathbf{F}^{\prime}_{\gamma} = f_{\gamma}, \qquad f_{\gamma} := \pv \int_{\gamma} K_{u}\left( \mathrm{dist}_{\R^3}(\x,\y), \mathrm{dist}_{\gamma}(\x,\y) \right)(\x-\y) \, \rd \cH(\y),
\end{equation*}
and so recovering $\mathbf{F}_{\gamma}$ from $\gamma$ requires two consecutive nonlocal operations. We refer to the resulting class of dynamics as a \emph{doubly-nonlocal} reaction-diffusion system to emphasize this structure. Overall, this approach to dynamics allows us to obtain a global well-posedness result under quite general hypotheses on $K(u,v),$ to demonstrate regularity of critical embeddings, to perform a detailed study of equilibria and to devise reliable numerical procedures.

Most of our arguments rely, at least in part, on some basic harmonic analysis and a few notions from knot theory; we begin by reviewing this material in the next section. We then proceed to study the energy \eqref{eq:introenergy} from a complexity point-of-view in the third section. As an example in this direction, a corollary of our analysis shows that the distortion $\delta_{\infty}(\gamma)$ and bending energy $E_{ {\rm b} }(\gamma)$ bound the average crossing number $\bar{c}(\gamma)$
\begin{equation}\label{eq:introbound}
\bar{c}(\gamma) \lesssim \delta^{3}_{\infty}(\gamma) E_{ {\rm b} }(\gamma),
\end{equation}
of an embedding. An analogous result holds for \eqref{eq:introenergy} provided $K(u,v)$ satisfies a ``homogeneity'' assumption. This is perhaps surprising in light of the following {observation}: There exists an infinite sequence of {ambient} isotopy classes of knots with uniformly bounded distortion as well as an infinite sequence of {ambient} isotopy classes of knots with uniformly bounded bending energy. These sequences therefore exhibit ``conflicting'' behavior, since \eqref{eq:introbound} rules out the existence of such an infinite sequence when both distortion and bending energy remain uniformly bounded. We also show that an energy of the form \eqref{eq:introenergy} induces, in a quantified sense, a stronger invariant than the average crossing number. These results are motivated by analogous results for the M\"{o}bius energy \cite{FHW} (i.e. \eqref{eq:introenergy} with $K(u,v) = u^{-1}-v^{-1}$ and without curvature) and the {ropelength \cite{BS,cantarella2014ropelength,DS,LSDR,DDS}}; however, our results are not entirely comparable with these earlier works and we generally prove them using different means. We then proceed to address the dynamics of \eqref{eq:introenergy} in the fourth section. We initiate this process by proving global well-posedness and regularity under certain ``homogeneity'' and ``degeneracy'' assumptions on the kernel $K(u,v)$. These assumptions cover many of the kernels $K(u,v)$ that have generated prior interest in the literature {\cite{OH1,OH2,OH3,OH4}}, and the main contribution here lies in finding a relatively general and broadly applicable hypothesis under which such a global result holds. This result complements prior work that proves global existence for \eqref{eq:introenergy} under the M\"{o}bius kernel {\cite{LS10}}, although we assign dynamics in a slightly different fashion. We conclude the fourth section with an analysis of dynamics near equilibrium. {\color{black} For instance, it is known that unit circles globally minimize \eqref{eq:introenergy} under proper monotonicity and convexity assumptions {\cite{ACFGH}}, and in addition, that desirable regularity properties (e.g. smoothness) hold for critical points of the O'Hara and M\"{o}bius energies \cite{BR13,BRS}.}  We supplement this by adapting techniques for nonlocal dynamical systems to show that circles are also properly ``isolated'' modulo to rotation and scaling invariance. This yields a local asymptotic stability result for our dynamics, {\color{black} which should be compared to {\cite{B12}} and the recent contribution \cite{BarXiv2} for the special case of pure M\"{o}bius flow}. We use a similar technique to show that the natural ``global'' version of this result fails for a pure bending energy flow: There exist unknotted initial configurations that remain unknotted for all time under a bending energy flow, yet the dynamics do not asymptotically converge to the globally minimal, unknotted circle. {\color{black} This dynamic result complements the energetic {\color{black} result, proved in \cite{GRvdM}, that the double-covered circle the \emph{unique} bending energy minimizer for the trefoil knot type}.} Finally, we explore global dynamical properties numerically in the final section.

{\color{black} To place these efforts in context, note that a large body of work has studied the energetics and dynamics of various knot energies. {\color{black} Much of the work in the field has focused on a few families of knot energies such as those suggested in \cite{GM99}, including tangent-point energies \cite{MR2902275,BR15}, Menger curvature energies \cite{SV13} and O'Hara's energies {\cite{OH1,OH2,OH3,OH4}} (of which the M\"{o}bius energy \cite{FHW} is a special case). The most closely related examples include work on the energetics and minimizers of integral Menger curvature \cite{BR152} , as well as work on the energetics and criticality theory for the M\"{o}bius energy \cite{BarXiv2,FHW,BRS,B12} and the extended O'Hara family \cite{BarXiv1,BR13} }. In contrast to these energies, the physical model \eqref{eq:introenergy} crucially depends on the bending energy which, as previously noted \cite{LS10}, obviously improves the analytic properties of the total energy. When taken together our results show that, in addition to its physical significance, the admission of the bending energy and a (quite weak) barrier to self-intersection in \eqref{eq:introenergy} allows for a wide class of models that exhibit all of the analytic features that have previously motivated knot energies.}

\section{Preliminary Material}

We begin our study by introducing the notation we shall use, then providing a few key definitions and finally recording a series of preliminary lemmas on which we will rely throughout the remainder of the paper. Lightface roman letters such as $x,y$ and $z$ will denote real numbers; their boldface counterparts $\x,\y$ and $\z$ will denote three-dimensional vectors. We reserve $\e_1,\e_2,\e_3$ for the canonical standard basis. Given $\x,\y \in \R^3$ we use $\langle \x,\y \rangle$ for the standard Euclidean inner-product, $\x \times \y$ for the cross product and $|\x|$ for the Euclidean norm. The notation $\langle \x,\y,\z \rangle := \langle \x, \y \times \z \rangle$ denotes the scalar triple product. We use $\T$ to denote the standard one-dimensional torus, which we view as the interval $ [-\pi,\pi ] \subset \mathbb{R}$ with endpoints identified. For $z \in [-\pi,\pi]$ we define
$$
\dd(z) := \min\{ |z| , 2\pi - |z| \},
$$
and then extend $\dd(z)$ to all of $\R$ via $2\pi$-periodicity. Thus for any pair of points $x,y \in \T$ the quantity $\dd(x-y)$ simply gives the geodesic distance between them. Given a square-integrable function $\gamma: \T \to \R^{d}$ we shall always use the definition and notation
$$
\|\gamma\|^{2}_{L^{2}(\T)} := \frac1{2\pi} \int_{\T} |\gamma(x)|^2 \, \rd x := \fint_{\T} |\gamma(x)|^2 \, \rd x
$$
to denote the $L^{2}(\T)$-norm, with the dimensionality $d$ omitted but always clear from context. We shall use
\begin{align*}
\hat{\gamma}_{k} = \fint_{\T} \gamma(x) \re^{- i k x } \, \rd x \qquad \text{and} \qquad \gamma(x) = \sum_{k \in \Z} \hat{\gamma}_k \, \re^{ ikx }
\end{align*}
to denote the forward and inverse Fourier transforms, so that
\begin{align*}
\|\gamma\|^{2}_{L^{2}(\T)} = \sum_{k \in \Z} |\hat{\gamma}_k|^2, \quad \|\gamma\|^{2}_{\dot{H}^{s}(\T)} = \sum_{k \in \Z} |k|^{2s} |\hat{\gamma}_k|^2 \quad \text{and} \quad \|\gamma\|^{2}_{H^{s}(\T)} = |\hat{\gamma}_0|^2 + \|{\color{black} \gamma}\|^{2}_{\dot{H}^{s}(\T)}
\end{align*}
provide the $L^{2}(\T)$ norm as above and an equivalent definition of the $H^{s}(\T)$ Sobolev norm. For any sufficiently regular mapping $\gamma: \T \to \R^3$ we use $\dot{\gamma}(x) = \gamma^{\prime}(x), \ddot{\gamma}(x) = \gamma^{\pprime}(x)$ or $\gamma^{(3)}(x) = \gamma^{\ppprime}(x)$ to denote ordinary differentiation in both the strong and weak sense. We say a rectifiable $\gamma(x)$ has \emph{unit speed} if $|\dot{\gamma}(x)| \equiv 1$ on $\T,$ and we then set $\kappa_{\gamma}(x) := |\ddot{\gamma}(x)|$ or simply $\kappa(x)$ as the pointwise curvature of such an embedding. Similarly, we say a rectifiable $\gamma(x)$ has constant speed if $|\dot{\gamma}(x)| \equiv \lenG/2\pi$, where
$$
\lenG := \int_{\T} |\dot{\gamma}(x)| \, \rd x
$$
denotes the length of the embedded curve. Conversely, given any tangent field $\tang : \T \to \R^{3}$ with $|\tang(x)| \equiv A$ and $\hat{\tang}_0 = 0$ we may define
$$
\gamma_{\tang}(x) := \fint_{\T} (z - \pi) \tang(z) \, \rd z + \int^{x}_{-\pi} \tang(z) \, \rd z
$$
and thereby recover a $2\pi$-periodic, constant speed curve that has $\gamma^{\prime}_{\tang} = \tang$ and center of mass at the origin. We refer to $\gamma_{\tang}$ as the knot induced by such a vector-field. Given mean zero function $f \in L^{1}(\T)$ we analogously use
$$
F(x) := \fint_{ \T } (z-\pi) f(z) \, \rd z + \int^{x}_{-\pi} f(z) \, \rd z
$$
to denote its mean zero primitive.

In addition to these notations, we also need to recall a few elementary facts from harmonic analysis. First, for a given $\gamma \in L^{2}(\T)$ we use the notation
\begin{equation}\label{eq:maxfunc}
\gamma^{*}(x) := \sup_{  0 < \ell < \pi } \, \frac1{2\ell} \int^{x + \ell}_{x-\ell}  |\gamma(z)| \, \rd z
\end{equation}
to denote its maximal function, {\color{black}see, for example, \cite[p.216]{SS09l}.} We then recall the standard fact that the map $\gamma \mapsto \gamma^*$ defines a bounded operator on $L^{2}(\T)$, so that
\begin{equation}\label{eq:xxx2}
\| \gamma^{*} \|_{L^{2}(\T)} \leq C_{*} \| \gamma \|_{L^{2}(\T)}
\end{equation}
for $C_{*} > 0$ some absolute constant. If $\gamma$ has unit speed then $\kappa_{\gamma}(x) = |\ddot{\gamma}(x)|$ furnishes its pointwise curvature, in which case we define $\kappa^{*}_{\gamma}(x) := (\ddot{\gamma})^{*}(x) $ and refer to $\kappa^{*}_{\gamma}$ as the curvature maximal function. The following elementary lemma will also prove useful ---
\begin{lemma}\label{lem:sima}
Assume that $u,v \in L^{2}(\T)$ and that there exist finite constants $\frak{u},\frak{v}$ so that the decay estimates
\begin{align*}
\sup_{ k \in \Z , k\neq 0} \, |k| |\hat{u}_k| \leq \mathfrak{u}, \;\; |\hat{u}_{0}| \leq \mathfrak{u} \qquad \text{and} \qquad \sup_{ k \in \Z , k\neq 0} \, |k| |\hat{v}_k| \leq \frak{v}, \;\; |\hat{v}_{0}| \leq \frak{v}
\end{align*}
hold. Then the product $w := uv \in L^{1}(\T)$ obeys the decay estimate
$$
\sup_{ k \in \Z , k\neq 0} \, \frac{|k||\hat{w}_k| }{\log(1+|k|)} \leq C_{1} \frak{u}\frak{v}
$$
for $C_1>0$ a universal constant. If $u,v \in L^{2}(\T)$ obey the decay estimates
\begin{align*}
\sup_{ k \in \Z , k\neq 0} \, |k|^{p} |\hat{u}_k| \leq \frak{u}, \;\; |\hat{u}_{0}| \leq \frak{u} \qquad \text{and} \qquad \sup_{ k \in \Z , k\neq 0} \, |k|^{q} |\hat{v}_k| \leq \frak{v}, \;\; |\hat{v}_{0}| \leq \frak{v},
\end{align*}
where $\frak{u},\frak{v}$ are finite constants and $1 \leq q < p$ are arbitrary exponents, then the product $w := uv$ obeys the decay estimate
$$
\sup_{k \in \Z, k \neq 0} \, |k|^{q}|\hat{w}_k| \leq C_{p,q} \frak{u}\frak{v}
$$
with $C_{p,q}$ a positive constant depending only upon the exponents.
\end{lemma}
\noindent See \cite{von2014nonlinear} {\color{black} lemma 6.1 and lemma 6.2,} for instance, for an indication of the proof.

Recall that a unit speed embedding $\gamma: \T \to \R^3$ is \emph{bi-Lipschitz} if there exists a constant $\ell > 0$ so that the inequality
$$
\ell^{-1} \, \dd(x-y) \leq |\gamma(x) - \gamma(y)| \leq \dd(x-y)
$$
holds for any possible pair $x,y \in \T$ of points. Thus $\ell > 0$ exactly when Gromov's \emph{distortion}
$$
\delta_{\infty}(\gamma) := \sup_{ \{ (x,y) \in \T \times \T : \dd(x-y) > 0\} } \, \frac{ \dd(x-y) }{|\gamma(x) - \gamma(y)|}
$$
is finite, with $\delta_{\infty}(\gamma)$ furnishing the smallest possible bi-Lipschitz constant. Moreover, the lower bound $\delta_{\infty}(\gamma) \geq \pi/2$ holds for $\gamma$ any closed, rectifiable curve with unit speed, {\color{black} see pp.6-9 \cite{G81}}. We shall refer to any pair of points $x_0,y_0 \in \T$ for which
$$
\frac{\dd(x_0 - y_0)}{|\gamma(x_0) - \gamma(y_0)|} = \delta_{\infty}(\gamma)
$$
as a \emph{distortion realizing pair}. A simple invocation of Taylor's theorem shows that
$$
|\gamma(x+z) - \gamma(z)|^2 = |z|^{2}( 1 + o(1) ) \quad \text{as} \quad |z| \to 0
$$
for any $\gamma \in H^{2}(\T)$ with unit speed, and thus any such $\gamma$ actually admits a distortion realizing pair. Given a bi-Lipschitz embedding $\gamma$ with unit speed, we define
\begin{equation}
\dd_{\infty}(\gamma) := \sup\left\{ \dd(x_0-y_0) :  \frac{\dd(x_0 - y_0)}{|\gamma(x_0) - \gamma(y_0)|} = \delta_{\infty}(\gamma) \right\}
\end{equation}
as the maximal distance on $\T$ between any distortion realizing pair. {\color{black} By applying the first derivative test to the expression $\dd(x_0 - y_0)/|\gamma(x_0) - \gamma(y_0)|$ we obtain
$$\frac{1}{\delta_{\infty}(\gamma)}= \left| \left<\frac{\gamma(x_0) - \gamma(y_0)}{|\gamma(x_0) - \gamma(y_0)|}, \dot{\gamma}(x_0) \right> \right|,$$
and so by applying Taylor's theorem and the triangle inequality to bound the right hand side of this expression from below we obtain a lower bound}
\begin{equation}\label{eq:z0lower}
\dd_{\infty}(\gamma) \geq \dd(x_0 - y_0) \geq \frac1{4}\left( \frac{ \delta_{\infty}(\gamma) - 1 }{\delta_{\infty}(\gamma) + 1 } \right)^{2}\| \ddot{\gamma}\|^{-2}_{L^{2}(\T)}
\end{equation}
for the distance on $\T$ between any such pair. Finally, if a bi-Lipschitz embedding $\gamma \in H^{2}(\T)$ has unit speed and if we define $\delta(x,y)$ as
$$
|\gamma(x)-\gamma(y)|^{2} = ( 1 - \delta(x,y) )\dd^{2} (y-x) \quad \text{where} \quad 0 \leq \delta(x,y) \leq 1 - 1/\delta^{2}_{\infty}(\gamma),
$$
we may appeal to Taylor's theorem one final time to conclude that the inequalities
\begin{equation}\label{eq:deltabound}
\delta(x,y) \leq \dd(y-x)\kappa^{*}_{\gamma}(x) \qquad \text{and} \qquad \delta(x,y) \leq \sqrt{2\pi}\|\ddot{\gamma}\|_{L^{2}(\T)} \dd^{\frac1{2}}(y-x)
\end{equation}
hold. Despite their elementary proofs, these inequalities prove quite useful.

Finally, for given a bi-Lipschitz embedding $\gamma : \T \mapsto \R^3$ we shall use $\K$ to denote a generic ambient isotopy class. Similarly, we shall employ while the notation $\K_{\gamma}$ if we wish to emphasize the class {induced} by some underlying curve. The following lemma from \cite{R05} and the embedding $H^{2}(\T) \subset C^{1}(\T)$ shows that ambient isotopy classes $\K$ are well-behaved with respect to convergence in both the $C^{1}(\T)$ and $H^{2}(\T)$ topologies ---
\begin{lemma}\label{lem:ambclass}
Let $\gamma \in C^{1}(\T)$ denote a {\color{black} positive velocity}, simple closed curve. Then there exists $\epsilon = \epsilon(\gamma)$ such that all $\tilde{\gamma} \in C^{1}(\T)$ satisfying $\|\gamma - \tilde{\gamma}\|_{C^{1}(\T)}\leq \epsilon$ are ambient isotopic. In particular, $\tilde{\gamma} \in \K_{\gamma}$ if $\|\gamma - \tilde{\gamma}\|_{C^{1}(\T)}$ is sufficiently small.
\end{lemma}

\section{Basic Energetics}
This section provides a brief analysis of nonlocal energies that, when defined over unit speed embeddings, take the form
\begin{equation}\label{eq:EKdef}
E_{K}(\gamma) := \frac1{4}\int_{\T\times\T} K\left( |\gamma(x)-\gamma(y)|^2, \dd^2(x-y) \right) \, \rd x \rd y
\end{equation}
for $K(u,v)$ some bi-variate kernel. We shall also consider the normalized or scale-invariant bending energy
$$
E_{ {\rm b} }(\gamma) := \frac{\lenG}{2\pi} \fint_{\T} \kappa^2_{\gamma}(x) |\dot{\gamma}(x)| \, \rd x
$$
as well as the superposition $E_{ {\rm b}, K }(\gamma) := E_{ {\rm b}}(\gamma) + E_{K}(\gamma)$ of such nonlocal energies \eqref{eq:EKdef} with the bending energy. Well-known examples of kernels in \eqref{eq:EKdef} include
\begin{equation}\label{eq:choices}
K(u,v) = \frac1{u} - \frac1{v} \quad (\text{M\"{o}bius}) \qquad \text{and} \qquad K(u,v) = \left( \frac1{u^{j}} - \frac1{v^{j}} \right)^{q} \quad (\text{O'Hara}).
\end{equation}
The motivation for these choices arises, at least in part, from the fact $E_{K}$ then defines a differentiable approximation of the distortion. More specifically, prior work \cite{FHW,OH3} demonstrates that $E_{K}(\gamma) < +\infty$ necessarily implies that $\gamma$ has finite distortion. Moreover, the O'Hara family converges as $j \to 0$ and $q \to \infty$ to the log-distortion $\log \delta_{\infty}(\gamma)$ of $\gamma$ after a suitable normalization. This observation suggests the somewhat more obvious family
$$
K(u,v) = \frac{v^{q}}{u^{q}} \quad (2q-\text{Distortion}),
$$
since $E_{K}$ then corresponds to the classical $L^{2q}$-norm approximation of the $L^{\infty}$-norm. Further motivations for using \eqref{eq:choices} include the fact that the M\"{o}bius energy and a large class of the O'Hara energies attain their global minimum at the standard embedding of the unit circle, as well as the fact that the M\"{o}bius energy exhibits a relationship with classical combinatorial quantities such as the crossing number.

We now show that a very large class of energies based on $E_{K}(\gamma)$ have these three motivating properties. In a certain sense these properties are best understood from a pure analytical point-of-view, rather than from an appeal to geometric or topological considerations (e.g. M\"{o}bius invariance). {\color{black} This point of emphasis echoes an observation made in earlier work on the M\"{o}bius energy --- the smoothness of its critical points follows without explicitly appealing to M\"{o}bius invariance itself \cite{BRS}.} We begin by showing that, provided $K(u,v)$ exhibits a sufficient degree of singularity, finiteness of the integral \eqref{eq:EKdef} necessarily implies that an $H^{2}(\T)$ curve $\gamma$ is bi-Lipschitz. We shall also obtain a concrete bound for the distortion $\delta_{\infty}(\gamma)$ in terms of $E_{K}$ and the $\dot{H}^{2}(\T)$ semi-norm in the process, although the bound itself is far from optimal for general kernels.

\begin{lemma}\label{lem:distbound}
Assume that $K(u,v)$ satisfies the $h/p$ homogeneity property
$$
K\left( \frac{v}{\alpha} , v \right) = h(\alpha) v^{-p} \quad \text{for all} \quad \alpha \geq 1
$$
for some exponent $p \geq 0$ and some function $h \in C^{1}((1,\infty) )$. Assume that the lower bound $h^{\prime}(\alpha) \geq \lambda_{h} > 0$ holds on $[2,\infty)$ as well. If $\gamma \in C^{0,1}(\T)$ has unit speed then
$$
\log\left( 1 + c_{0} \delta_{\infty}(\gamma) \right) \leq C_{0} \left( 1 + E_{K}(\gamma) \right) \dd^{-2}_{\infty}(\gamma)
$$
for $C_0(p,h),c_0(p,h) > 0$ finite constants. Moreover, if $\gamma \in H^{2}(\T)$ then
\begin{equation}\label{eq:log}
\log\left( 1 + c_{0} \delta_{\infty}(\gamma) \right) \leq C_0 \left( 1 + E_{K}(\gamma) \right) \|\ddot{\gamma}\|^{4}_{L^{2}(\T)},
\end{equation}
and so $\gamma$ is bi-Lipschitz if $E_{K}(\gamma)$ is finite.
\end{lemma}
\begin{proof} See appendix, lemma \ref{lem:distboundA} \end{proof}
\noindent The $h/p$ homogeneity hypothesis provides a means to unify and simplify our analysis, yet it proves general enough to cover all families $K(u,v)$ introduced so far. For instance, we have $h(\alpha) = \alpha - 1$ and {\color{black} $p=1$} for the M\"{o}bius energy, the relations $h(\alpha) = (\alpha^{j}-1)^{q}$ and $p=jq$ for the O'Hara family and $h(\alpha) = \alpha^{q}$ and $p=0$ for the $2q$-distortion. The lemma therefore applies for all three families, provided $jq \geq 1$ in the second case and $q\geq1$ in the final case. It is also clear that the conclusion holds for kernels of the form $K(u,v) = K_{0}(u,v) + K_{1}(u,v)$ with $K_0$ an $h/p$ homogeneous kernel and $K_1$ bounded from below, although we have no impetus to pursue this level of generality since the requisite modifications to the argument and its conclusion are straightforward. Finally, we cannot remove the dependence on $\dd_{\infty}(\gamma)$ or dispense with the hypothesis that $\| \ddot{\gamma}\|_{L^{2}(\T)} < +\infty$ and still have the conclusion of the lemma hold at this level of generality. Indeed, for the $2q$-distortion family it is easy to construct a smooth sequence $\gamma^{n}$ for which $E_{K}(\gamma^n)$ remains uniformly bounded but $\delta_{\infty}(\gamma^n)$ and $\| (\gamma^n)^{\pprime}\|_{L^{2}(\T)}$ diverge.

The previous lemma illustrates the fact that a wide class of energies
$$
E_{K}(\gamma) := \frac1{4}\int_{\T \times \T} K\left( |\gamma(x) - \gamma(y)|^2 , \dd^2(x-y) \right) \, \rd y \rd x
$$
approximate the distortion. Under similar hypotheses on the kernel $K(u,v),$ both the energy $E_{K}(\gamma)$ and the bending energy $E_{ {\rm b}}(\gamma)$ attain their global minima at the standard embedding $\gamma_{ {\rm circ}}$ of the unit circle {\cite{ACFGH,LS85}}. We shall quickly review the arguments underlying these known facts, as this discussion will allow us to emphasize an analytical point --- both arguments appeal to the same classical technique in the calculus of variations, i.e. the use of Poincar\`{e}-type inequalities with optimal constants. For instance, the statement $E_{K}(\gamma) \geq E_{K}(\gamma_{ {\rm circ}})$ follows  directly from the Poincar\`{e}-type inequality
\begin{equation}\label{eq:poincare1}
\fint_{\T} |\gamma(x+z) - \gamma(x)|^2 \, \rd x \leq 2(1-\cos(z)) \fint_{\T} |\dot{\gamma}(x)|^2 \, \rd x
\end{equation}
with optimal constant. Following \cite{ACFGH}, for unit speed curves the inequality \eqref{eq:poincare1} yields
$$
K\left( \fint_{\T} |\gamma(x+z) - \gamma(x)|^{2} \, \rd x, z^2 \right) \geq K\left( 2(1-\cos(z)) , z^2 \right),
$$
whenever $K(u,v)$ decreases in its first argument over the non-negative reals. By Jensen's inequality this in turn shows that
$$
\frac{ 4 E_{K}(\gamma) }{4\pi^2} = \fint_{\T} \fint_{\T} K\left( |\gamma(x+z) - \gamma(x)|^{2} , z^2 \right)\, \rd x \rd z \geq \fint_{\T} K\left( 2(1-\cos(z)) , z^2 \right) \, \rd z = \frac{4 E_{K}(\gamma_{\mathrm{circ}})}{4\pi^2}
$$
whenever $K(\cdot,z^2)$ is also convex. Demonstrating global optimality of $\gamma_{ {\rm circ}}$ for the bending energy proceeds in a similar fashion. In this case the classical Poincar\`{e} inequality with optimal constant
\begin{equation}\label{eq:poincare2}
\fint_{\T} |\dot{\gamma}(x)|^2 \, \rd x \leq \fint_{\T} |\ddot{\gamma}(x)|^2 \, \rd x
\end{equation}
provides the starting point. Letting $\gamma_1$ denote the unit speed embedding of $2\pi\gamma/\lenG,$ a simple appeal to parametrization and scale invariance of $E_{ {\rm b}}(\gamma)$ shows
$$E_{ {\rm b}}(\gamma) = \fint_{\T} |\ddot{\gamma_{1}}(x)|^2 \, \rd x \geq \fint_{\T} |\dot{\gamma}_{1}(x)|^2 \, \rd x = 1 = E_{ {\rm b}}(\gamma_{ {\rm circ}} ).$$
As the class of standard circles $\gamma_{ {\rm circ}}$ furnish the only unit speed curves that achieve equality in either \eqref{eq:poincare1} or \eqref{eq:poincare2}, we may recall
\begin{theorem}
Suppose that for each $z \in \T$ the function $f_{z}(u) := K(u,z^2)$ is non-decreasing and convex on $\mathbb{R}^{+}$. Then for any $\nu \geq 0$ the energy
$$
E_{\nu,K}(\gamma) := \nu E_{ {\rm b}}(\gamma) + E_{K}(\gamma)
$$
attains its minimum over unit speed curves at the unit circle. Moreover, if either $\nu > 0$ or $f_{z}(u)$ is strictly decreasing then unit circles are the unique global minimizers.
\end{theorem}
\noindent In particular, the M\"{o}bius energy as well as the O'Hara family and the $2q$-Distortion family all have $\gamma_{ {\rm circ}}$ as their global minimizer. Once again, the arguments above leading to this theorem are both standard and known {\cite{FHW,OH4,ACFGH}}; we recall them simply to emphasize the connection between them as well as the broader connection to classical PDE and variational arguments. Thus there are analytical properties, such as convexity and Sobolev inequalities, rather than geometric properties, such as M\"{o}bius invariance, lying at the heart of the matter.

Finally, we turn our attention toward relating the energies $E_{K}(\gamma)$ and $E_{ {\rm b}}(\gamma)$ to more classical measures of complexity. Once again, analytical considerations and Sobolev spaces shall come to the fore. A corollary of this analysis will also allow us to illustrate that an energy of the form $E_{ {\rm b},K}(\gamma)$ induces, in a certain sense, a much stronger invariant than the crossing number. Following {\cite{FHW,FH,A}}, we shall begin by considering the mapping $f(x,z) : \T \times \T \to \Sd$ defined by
$$
f_{\gamma}(x,z) := \left( \frac{ \gamma(x) - \gamma(x+z) }{|\gamma(x) - \gamma(x+z)|} \right)\mathbf{1}_{ \{|z| > 0\} }(z)
$$
with $\gamma:\T \to \R^3$ a Lipschitz embedding with finite distortion. For any such Lipschitz curve $\gamma(x)$ the function $f_{\gamma}(x,z)$ is Lipschitz in both variables whenever $|z| > \epsilon$ and $\gamma(x)$ has finite distortion. For such functions, the co-area formula for Lipschitz maps {\cite{federer2014geometric}} therefore implies that
$$
\int_{\T \times \T}  \frac{ w(x,z)| \langle \dot{\gamma}(x) , \dot{\gamma}(x+z) , \gamma(x+z) - \gamma(x) \rangle | }{|\gamma(x+z) - \gamma(x)|^{3}} \, \rd z \rd x = \int_{ \Sd } \left( \sum_{ (x,z) \in f^{-1}_{\gamma}(\sigma) } w(x,z) \right) \, \rd \Sd_{\sigma},
$$
provided $w : \T \times \T \mapsto [0,\infty]$ is a positive, Lebesgue measurable function. For a given $\sigma \in \Sd$ let us define the crossing set $C_{\gamma}(\sigma) \subset \T \times \T$ as
$$
C_{\gamma}(\sigma) = \left\{ (x,z) : (P_{\sigma} \gamma)(x) = (P_{\sigma} \gamma)(x+z) , z \neq 0 \right\},
$$
where $P_{\sigma} \gamma := \gamma - \langle \gamma , \sigma\rangle \sigma \subset \sigma^{\perp}$ represents the corresponding planar curve induced by orthogonal projection. Thus $(x,z) \in C_{\gamma}(\sigma)$ precisely when distinct points of $P_{\sigma} \gamma$ self-intersect. As the equality
$$
C_{\gamma}(\sigma) = f^{-1}_{\gamma}(\sigma) \sqcup f^{-1}_{\gamma}(-\sigma)
$$
holds by a trivial calculation, we may therefore conclude
\begin{equation}\label{eq:gcn}
\int_{\T \times \T} \frac{ w(x,z) | \langle \dot{\gamma}(x) , \dot{\gamma}(x+z) , \gamma(x+z) - \gamma(x) \rangle | }{|\gamma(x+z) - \gamma(x)|^{3}} \, \rd z \rd x = \frac1{2} \int_{ \Sd } \left( \sum_{ (x,z) \in C_{\gamma}(\sigma) } w(x,z) \right) \, \rd \Sd_{\sigma}.
\end{equation}
When $w \equiv 1$ the integrand
$$
\sum_{ (x,z) \in C_{\gamma}(\sigma) } w(x,z)
$$
in \eqref{eq:gcn} simply gives the cardinality of $C_{\gamma}(\sigma),$ and so \eqref{eq:gcn} reduces to a constant multiple of the average crossing number. We may utilize the non-negative function $w: \T \times \T \mapsto [0,\infty]$ to apply a positive weight applied to each point of self-intersection in $C_{\gamma}(\sigma),$ and in this way arrive at a weighted generalization or ``weighted'' crossing number. By analogy with electrostatics, we shall consider the simple power-law family
\begin{equation}\label{eq:gpdef}
w_{p}(x,z) := \left( \frac{ \lenG }{ D_{\gamma}(\gamma(x),\gamma(x+z)) } \right)^{p} \qquad (p \geq 0)
\end{equation}
of weighting functions, where $D_{\gamma}(\gamma(x),\gamma(x+z))$ denotes the arc-length between points. We may then view \eqref{eq:gcn} as inducing a ``repulsion'' between points of self-intersection, in the sense that small values occur when weighted crossings are, on average, equally spaced along $\gamma$ as measured by relative length. We primarily intend this family as an analytical device or gauge for measuring and drawing analytical comparsions --- a bound on \eqref{eq:gcn} for $p>0$ represents a stronger conclusion, in general, than a bound on the crossing number or average crossing number. Moreover, for smooth curves $\gamma$ Taylor's theorem immediately yields
$$
|\langle \dot{\gamma}(x) , \dot{\gamma}(x+z) , \gamma(x+z) - \gamma(x) \rangle| = \frac{z^4}{12}|\langle \dot{\gamma}(x) , \ddot{\gamma}(x) , \dddot{\gamma}(x) \rangle| + O(z^5),
$$
while the denominator in \eqref{eq:gcn} scales like $|z|^{3+p}$ near the origin. If we define $c_{p}(\gamma)$ as \eqref{eq:gcn} with the weight \eqref{eq:gpdef} then a bound of the form
\begin{equation}\label{eq:bestpossible}
\sup_{ 0 \leq p < 2} \, (2-p)c_{p}(\gamma) \leq F\left( E_{ {\rm b},K}(\gamma) \right) \quad \text{with} \quad F(z) \text{ continuous, increasing}
\end{equation}
represents the strongest possible conclusion, and therefore the upper limit of our analytical scale.

We shall need with the following lemma in order to perform our analytical comparison.
\begin{lemma}\label{lem:mainenergy}
Assume that $\gamma \in C^{0,1}(\T)$ and that $\gamma$ has finite distortion. Assume further that $\gamma$ has constant speed. Then
\begin{equation}\label{eq:mainenergy}
c_{p}(\gamma) \leq  C_{p}\left( \frac{2\pi}{\lenG} \right)^2  \|\gamma\|^{2}_{ \dot{H}^{\frac{3+p}{2}}(\T)} \delta^3_{\infty}(\gamma),
\end{equation}
and if $-1 < p < 1$ then the constant
$$
C_{p} := 2\sqrt{2}(2\pi)^{p+1}\left( \int^{\infty}_{0} \frac{(1-\cos(u)) }{u^{2+p} } \, \rd u \right)^{\frac1{2}} \left( \int^{\infty}_{0}\frac{(1-\cos(u))^{2} + (\sin u - u)^2 }{u^{4+p} }  \, \rd u\right)^{\frac1{2}}
$$
is finite.
\end{lemma}
\begin{proof} See appendix, lemma \ref{lem:mainenergyA} \end{proof}
\noindent This lemma allows us to relate combinatorial measures of complexity for a curve $\gamma$ to more standard measures of regularity of the embedding. By taking $p=0$ and evaluating the constant $C_{p} = 4\pi^2/\sqrt{3}$ explicitly we may observe the most straightforward consequence, i.e. the inequality
\begin{equation}\label{eq:acnbound1}
\bar{c}(\gamma) := \frac1{8\pi} \int_{\Sd} |C_{\gamma}(\sigma)|\, \rd \Sd_{\sigma} \leq \frac{\pi}{\sqrt{3}}\left( \frac{2\pi}{\lenG} \right)^{2} \| \gamma \|^{2}_{\dot{H}^{\frac{3}{2}}(\T)} \delta^{3}_{\infty}(\gamma)
\end{equation}
for the average crossing number. To help place this inequality in a more familiar context, given an exponent $q \geq 1$ and a rectifiable $\gamma : \T \mapsto \mathbb{R}^{3}$ let us define the \emph{scale-invariant total $q$-curvature} of $\gamma$ as
$$
\bar{\kappa}_{q}(\gamma) := \left( \frac{\lenG}{2\pi}\right)^{1-\frac1{q}} \left( \fint_{\T} \kappa^{q}_{\gamma}(x) |\dot{\gamma}(x)| \, \rd x \right)^{\frac1{q}}.
$$
The case $q=2$ yields the (square root of) the bending energy, while the case $q=1$ reduces to Milnor's notion of total curvature {\cite{M}}. For constant speed curves and $1 < q \leq 2$ the inequality
$$
\left( \frac{2\pi}{\lenG} \right) \|\gamma\|_{ \dot{H}^{\frac{3}{2}}(\T)} \leq c_{q} \bar{\kappa}_{q}(\gamma)
$$
holds, for $c_{q} > 0$ some positive constant, due to the Hausdorff-Young inequality. If $1 < q \leq 2$ we may therefore conclude the following total $q$-curvature bounds
\begin{equation}\label{eq:qcurv}
\bar{c}(\gamma) \leq C_{q} \bar{\kappa}^{2}_{q}(\gamma) \delta^{3}_{\infty}(\gamma)
\end{equation}
for the average crossing number. From this observation we conclude that essentially \emph{any assumption whatsoever} that is stronger than an assumption of finite total curvature yields a bound for the average crossing number. Indeed, instead of assuming $\bar{\kappa}_{q}(\gamma) < \infty$ for some $q >1$ we could also invoke the slightly weaker assumption that $\kappa_{\gamma}(x)$ lies in the real Hardy space ${\rm Ha}^{1}(\T)$ and still deduce an average crossing number bound. As $q \to 1$ the constants $C_{q}$ do not remain bounded, however, and so the natural ``limiting'' conclusion $\bar{c}(\gamma) \leq C\bar{\kappa}^{2}_{1}(\gamma) \delta^{3}_{\infty}(\gamma)$ of \eqref{eq:qcurv} need not hold in general. The Hardy space ${\rm Ha}^{1}(\T)$ usually serves as the substitute for $L^{1}(\T)$ in such a circumstance, but as \eqref{eq:acnbound1} reveals, the fractional Sobolev space $\dot{H}^{\frac{3}{2}}(\T)$ is actually the natural choice: The inequality \eqref{eq:acnbound1} recovers the ``proper'' limiting inequality as $q \to 1$, for while neither embedding
$$
\| \ddot{\gamma} \|_{L^{1}(\T)} \leq C \| \gamma \|_{ \dot{H}^{\frac{3}{2}}(\T) } \qquad \text{nor} \qquad \| \gamma \|_{\dot{H}^{\frac{3}{2}}(\T)} \leq C \| \ddot{\gamma}\|_{ L^{1}(\T) }
$$
holds, the finiteness of $\bar{\kappa}_{q}(\gamma)$ implies the finiteness of both. As $q \to 1$ we recover the $\dot{H}^{3/2}(\T)$ norm rather than finite total curvature in \eqref{eq:acnbound1}, and it is therefore the natural measure of regularity to use. It prefers oscillatory components in the curvature measure $\ddot{\gamma}$ rather than the piecewise-constant curvature measures that characterize embeddings with finite total curvature, and so embeddings with oscillatory curvature rather than piecewise-linear embeddings must necessarily have finite average crossing number.

Another corollary of the bound \eqref{eq:acnbound1} shows that finite bending energy $\kappa_{2}(\gamma)$ and finite distortion imply a finite average crossing number. In fact we obtain a stronger conclusion from this analysis, in that an assumption of finite bending energy yields a significantly stronger conclusion than a simple crossing number bound. Specifically, for any $0 \leq p < 1$ the weighted crossing number
\begin{equation}\label{eq:wcnboundh}
\bar{c}_{p}(\gamma) := \frac1{4\pi}\int_{\Sd} \sum_{ (x,z) \in C_{\gamma}(\sigma) }  \left( \frac{ \mathrm{len}(\gamma) }{ D_{\gamma}(\gamma(x),\gamma(x+z)) }\right)^{p} \, \rd \Sd_{\sigma} \leq C_{p} \bar{\kappa}^{2}_{2}(\gamma) \delta^{3}_{\infty}(\gamma)
\end{equation}
remains bounded whenever $\gamma$ has finite distortion and finite bending energy. {\color{black} To appreciate the significance of this bound, recall that Gromov has provided an infinite sequence of ambient isotopy classes that have uniformly bounded distortion \cite[p.308]{G78}.} In addition, every $(2,q)$-torus knot has a smooth unit-length embedding $\gamma$ in $\mathbb{R}^3$ such that $E_b(\gamma)\leq (4\pi)^2+\epsilon$ for every $\epsilon>0$ \cite{GRvdM}, and so there also exists an infinite sequence of ambient isotopy classes with uniformly bounded bending energy. The crossing number bounds the weighted crossing number from below and defines a {\color{black} finite-to-one invariant (i.e. for every fixed integer $n$ there are at most finitely many knots with crossing number equal to $n$)}, so \eqref{eq:wcnboundh} cannot hold if we neglect either the bending energy or the distortion. However, by combining them we obtain not only a crossing number bound, but in fact a stronger weighted crossing number bound. Moreover, \eqref{eq:wcnboundh} shows that no infinite sequence of embeddings of distinct {ambient} isotopy classes of knots has both uniformly bounded distortion and uniformly bounded bending energy. A similar bound applies for $E_{ {\rm b},K}(\gamma)$ whenever the kernel $K(u,v)$ satisfies lemma \ref{lem:distbound}, and so the energy $E_{ {\rm b},K}(\gamma)$ is stronger than the weighted crossing number in an analogous sense.

Finally, it is worth briefly mentioning the consequences of \eqref{eq:wcnboundh} at the level of invariants. We may use $E_{ {\rm b},K}(\gamma)$ or the weighted crossing number to define invariants via minimization in the standard way, i.e. by defining
$$
\bar{c}_{2}(\K) := \inf_{ \gamma \in \K } \, \sup_{0 \leq p < 2} \, (2-p) \bar{c}_{p}(\gamma) \qquad E_{ {\rm q,\infty} }(\K) := \min_{ \gamma \in \K} \, \bar{\kappa}_{q}(\gamma)\delta_{\infty}(\gamma)
$$
for $\K$ an arbitrary ambient isotopy class. That  $E_{ {\rm b},K}(\gamma)$ (under appropriate hypotheses on the kernel) and $\bar{\kappa}_{q}(\gamma)\delta_{\infty}(\gamma)$ actually attain their minima within an {ambient} isotopy class $\K$ follows from lemma \ref{lem:ambclass} and a standard argument based on the direct method. For these invariants a bound of the form
$$
\bar{c}_{2}(\K) \leq \mathrm{ \textbf{poly} }(E_{ {\rm q,\infty} }(\K))
$$
holds for $q>1$ arbitrary. By appealing to results relating the {curvature and distortion to ropelength \cite{cantarella2014ropelength,DS,LSDR,DDS} and ropelength to crossing number \cite{CKS}} we may also conclude the corresponding upper bound
$$
E_{ {\rm q,\infty} }(\K) \leq \mathrm{ \textbf{poly} }( \bar{c}_{2}(\K) ),
$$
and so these invariants are polynomially equivalent.

\section{Dynamics of Embedded Curves}
With a basic understanding of the energies $E_{ {\rm b,K}}(\gamma)$ established, we now turn our attention toward the dynamics they induce via the constrained gradient flow
\begin{equation}\label{eq:gflow}
\tang_{t} = \tang_{xx} + \F_{\gamma_{\tang} } + \lambda_{\tang} + \mu_{\tang} \tang
\end{equation}
of such an energy. Given a bi-variate kernel $K(u,v)$ we shall use
\begin{align}\label{eq:force}
f_{\gamma_{\tang}}(x) &:= \pv \int_{\T} K_{u}\left( |\gammat(x) - \gammat(y)|^2, \dd^{2}(x-y) \right)(\gammat(x) - \gammat(y) ) \, \rd y,  \nonumber \\
\mathbf{F}_{\gamma_{\tang}}(x) &:= \fint_{\T} (z-\pi) f_{\gammat}(z) \, \rd z +  \int^{x}_{-\pi} f_{\gammat}(z) \, \rd z \qquad
\end{align}
to denote the self-repulsive nonlocal forcing. For any tangent field $\tang \in H^{1}(\T),$ defining the Lagrange multipliers $\lambda_{\tang} \in \R^{3}$ and $\mu_{\tang} \in L^{1}(\T)$ as
\begin{align*}
A_{\tang}\lambda_{\tang} &= \fint_{\T} \left(\frac{  \langle \F_{\gamma_{\tang} }, \tang \rangle-  |\tang_x|^2 }{|\tang|^2}\right)\tang \, \rd x, \quad A_{\tang} := \mathrm{Id} - \fint_{\T} \frac{ \tang \otimes \tang}{|\tang|^2} \, \rd x, \quad
\mu_{\tang} := \frac{ |\tang_x|^2 - \langle \F_{\gamma_{\tang} }, \tang \rangle - \langle \lambda_{\tang} , \tang \rangle}{ |\tang|^2 }
\end{align*}
completes the description of the flow. The presence of these multipliers guarantees that $\hat{\tang}_0 = 0$ and $|\tang(x,t)| \equiv C$ for all $(x,t) \in \T \times (0,\infty)$ provided these properties hold initially. The corresponding induced knot
$$\gamma_{ \tang(t) }(x) := \fint_{\T} (z - \pi) \tang(z,t) \, \rd z + \int^{x}_{-\pi} \tang(z,t) \, \rd z$$
then defines a closed, constant speed curve for as long as the solution to \eqref{eq:gflow} exists.

\subsection{Local Existence}
We first provide a local-in-time existence and uniqueness result for \eqref{eq:gflow}, which follows from a standard fixed-point argument. This argument requires a few minor but essential modifications to handle the technicality that, in general, our estimates for the nonlocal forcing \eqref{eq:force} only apply when given a-priori a unit speed curve. Moreover, some classical existence results (e.g. \cite{Henry:1981}) cannot account for the presence of Lagrange multipliers as lower-order terms while still yielding the full strength of an existence result for generic $H^{1}(\T)$ initial data. Even so, the majority of this sub-section still qualifies as grunt work; the disinterested reader may comfortably skip the proofs and simply take the results for granted.

Our argument proceeds by establishing bounds and Lipschitz estimates on the lower-order terms in \eqref{eq:gflow}, and then proceeds to the existence proof itself. We begin with bounds for the nonlocal forcing \eqref{eq:force}, where for a given tangent field $\tang \in H^{1}_{0}(\T)$ we shall always use $\gammat$ to denote the induced $H^{2}(\T)$ knot. The process of obtaining these estimates will reveal those conditions on the kernel $K(u,v)$ that our local existence approach requires. We shall make these requirements precise during the course of our arguments. To begin the task at hand, we recall the following basic result ---
\begin{lemma}\label{lem:lapest}
If $\gamma \in H^{2}(\T)$ and $0 \leq p < \frac1{2}$ then the linear operator
\begin{equation}\label{eq:linpv}
L_{p}[\gamma] := \pv \, \int_{\T} \frac{\gamma(x)-\gamma(y)}{\dd^{2(p+1)}(x-y)} \, \rd y
\end{equation}
defines an $H^{1-2p}(\T)$ function. As a Fourier multiplier it acts as a (negative) fractional Laplacian
$$
\hat{\gamma}_k \to \lambda_{k} \hat{\gamma}_k \qquad \lambda_{k} = C_{p,k} |k|^{1+2p} \qquad 0 < C_{k,p} = O(1) \quad {\color{black}\text{as} \quad |k|\rightarrow \infty,}
$$
and in particular the operator-norm estimate
$$
\|L_{p}[\gamma]\|_{ \dot{H}^{1-2p}(\T)} \leq C_{p} \|\gamma\|_{\dot{H}^{2}(\T)}
$$
holds.
\end{lemma}
\begin{proof} See appendix, lemma \ref{lem:lapestA} \end{proof}
\noindent With this lemma in hand, we may establish the requisite estimates for $\mathbf{F}_{\gammat}$ by appealing to the following pair of lemmas. The first provides uniform bounds for the nonlocal forcing; the second provides a simple Lipschitz estimate. We shall only require that the interaction kernel $K(u,v)$ satisfies some form of smoothness, homogeneity and {cancelation} or degeneracy condition in the course of obtaining these estimates. More specifically, we consider kernels obeying
\begin{enumerate}[]
\item {\rm ($h$/$p$ Homogeneity)}: There exists a function $h \in C^{1}( [1,\infty) )$ and an exponent $p \geq 0$ so that
$$
K \left( \frac{v}{\alpha},v\right) = h(\alpha)v^{-p}  \qquad \text{for all} \qquad v \neq 0, \; \alpha \geq 1.
$$
\item {\rm (0-Degeneracy)}: The function $g(\alpha) := -\alpha^{2} h^{\prime}(\alpha)$ is $C^{1}([1,\infty))$ and $p < \frac1{2}$.
\item {\rm ($m$-Degeneracy)}: For some non-negative integer $m > 4p-2,$ the function $g(\alpha) := -\alpha^{2} h^{\prime}(\alpha)$ is $C^{ (m+1) }( [1,\infty) )$ and has a root of order $(m+1)$ at $\alpha = 1,$
$$
g(1) = g^{\prime}(1) = \ldots = g^{(m)}(1) = 0.
$$
\end{enumerate}
\noindent {\color{black} To motivate these definitions, we may observe that kernel
$$
K(u,v) = \frac{v^{q}}{u^{q}} \quad (2q-\text{Distortion})
$$
for the $2q-$Distortion satisfies the $h/p$ homogeneity hypothesis with $h(\alpha) = \alpha^{q}$ and $p = 0;$ thus $0$-degeneracy applies to this family of examples. Similarly, the kernels
$$
K(u,v) = \left( \frac1{u^j} - \frac1{v^j} \right)^{q}
$$
defining the O'Hara family satisfy $h/p$ homogeneity with $h(\alpha) = (\alpha^{j}-1)^{q}$ and $p = jq;$ If $jq \geq 1$ and $q$ is sufficiently large, say $j = 1/q$ and $q \geq 5$, then the kernel is also $\geq3-$degenerate. These hypotheses therefore codify the sense in which both families provide distorion approximations (i.e. homogeneity) while still yielding convergent integrals for the corresponding energy (i.e. degeneracy).} Under these hypotheses, we have
\begin{lemma}\label{lem:boundF}
Assume that $K(u,v)$ is $h/p$ homogeneous and either $0$-degenerate or $m$-degenerate. Assume that $\gamma \in H^{2}(\T)$ is a unit speed, bi-Lipschitz embedding. Then for $0 < \epsilon < 1$ the function
$$
f_{ \gamma ,\epsilon}(x) :=  \int_{ \T \cap \{ \dd(x-y) \geq \epsilon\} } K_{u}\left( |\gamma(x) - \gamma(y)|^2, \dd^{2}(x-y) \right)(\gamma(x) - \gamma(y) ) \, \rd y
$$
has mean zero. Moreover, $f_{ \gamma ,\epsilon}$ obeys a uniform $L^{2}(\T)$ bound
$$
\| f_{ \gamma ,\epsilon} \|_{L^{2}(\T)} \leq C(\delta_{\infty}(\gamma),\|\ddot{\gamma}\|_{L^{2}(\T)})
$$
for $C(\delta_{\infty}(\gamma),\|\ddot{\gamma}\|_{L^{2}(\T)})$ an $\epsilon$-independent constant.
\end{lemma}
\begin{proof} See appendix, lemma \ref{lem:boundFA} \end{proof}
\noindent A simple reformulation of the argument underlying the lemma shows that
$$
\| f_{\gamma,\epsilon} - f_{\gamma,\delta} \|_{L^{2}(\T)} \to 0 \qquad \text{as} \qquad \epsilon,\delta \to 0,
$$
and so the principal value integral
$$
f_{ \gamma}(x) :=  \lim_{\epsilon \downarrow 0} \; \int_{ \T \cap \{ \dd(x-y) \geq \epsilon\} } K_{u}\left( |\gamma(x) - \gamma(y)|^2, \dd^{2}(x-y) \right)(\gamma(x) - \gamma(y) ) \, \rd y
$$
exists in $L^{2}(\T)$ and obeys the same bound established in the lemma. Note also that the mean-zero property
$$
\fint_{\T} f_{ \gamma}(x) \, \rd x = 0
$$
continues to hold for the limit as well. With this issue under control, we may now proceed to establish a simple Lipschitz estimate on the principal value that will prove useful throughout our analysis. By imposing the hypotheses of lemma \ref{lem:boundF}, we may also conclude
\begin{lemma}\label{lem:lipF}
Assume that $K(u,v)$ is $h/p$ homogeneous and either $0$-degenerate or $m$-degenerate. Let $\gamma,\phi \in H^{2}(\T)$ denote unit speed, bi-Lipschitz embeddings. Then for $0 < \epsilon < 1$ the uniform Lipschitz estimates
$$
\| f_{\gamma,\epsilon} - f_{\phi,\epsilon} \|_{L^{2}(\T)} \leq C(\delta_{\infty}(\gamma),\delta_{\infty}(\phi),\|\gamma\|_{H^{2}(\T)},\|\phi\|_{H^{2}(\T)} )\|\gamma - \phi\|_{H^{2}(\T)}
$$
hold, where $C(\delta_{\infty}(\gamma),\delta_{\infty}(\phi),\|\gamma\|_{H^{2}(\T)},\|\phi\|_{H^{2}(\T)} )$ denotes an $\epsilon$-independent constant.
\end{lemma}
\begin{proof} See appendix, lemma \ref{lem:lipFA} \end{proof}
\noindent As before, this lemma allows us to conclude that the principal value integrals $f_{\gamma},f_{\phi}$ obey the same Lipschitz estimate
$$
\| f_{\gamma}- f_{\phi}\|_{L^{2}(\T)} \leq C(\delta_{\infty}(\gamma),\delta_{\infty}(\phi),\|\gamma\|_{H^{2}(\T)},\|\phi\|_{H^{2}(\T)} )\|\gamma - \phi\|_{H^{2}(\T)}
$$
as well. We then simply note that the nonlocal forcing
$$
\mathbf{F}_{\gamma}(x) := \fint_{\T} (z-\pi)f_{\gamma}(z) \, \rd z + \int^{x}_{-\pi} f_{\gamma}(z) \, \rd z
$$
defines a continuous, periodic, $H^{1}(\T)$ function with mean zero. We may therefore infer that the estimates
\begin{align}
\|\mathbf{F}_{\gamma}\|_{H^{1}(\T)} &\leq C( \delta_{\infty}(\gamma) , \|\gamma\|_{H^{2}(\T)} ) \nonumber \\
\|\mathbf{F}_{\gamma} - \mathbf{F}_{\phi}\|_{H^{1}(\T)} &\leq C( \delta_{\infty}(\gamma) , \delta_{\infty}(\phi),\|\gamma\|_{H^{2}(\T)},\|\phi\|_{H^{2}(\T)} )\| \gamma - \phi \|_{H^{2}(\T)}
\end{align}
hold for the nonlocal forcing in the gradient flow. We may finally complete the task at hand by establishing analogous bounds on the Lagrange multipliers
$$
A_{\tang} \lambda_{\tang} := \fint_{\T} \left( \frac{ \langle \mathbf{F}_{\gammat} , \tang \rangle -  |\tang_x|^2 }{|\tang|^2}\right)\tang \, \rd x \qquad \mu_{\tang} = \frac{ |\tang_x|^2 - \langle  \mathbf{F}_{\gammat} +\lambda_{\tang} , \tang \rangle}{|\tang|^2}
$$
in the constrained flow. While these estimates follow from completely routine arguments, we provide proofs for the sake of completeness and concreteness. We begin by establishing that $A_{\tang}$ is non-singular and depends in a continuous fashion upon its input. Specifically, we have
\begin{lemma}\label{lem:boundA}
Suppose that $\tang,\sang \in H^{1}(\T)$ and that
$$
\min_{x \in \T} \, |\tang(x)| \geq \frac1{C},\quad \min_{x \in \T} \, |\sang(x)| \geq \frac1{C}
$$
for $0 < C < \infty$ a positive constant. If $\tang$ is non-constant then
$$
A_{\tang} := \mathrm{Id} - \fint_{\T} \frac{ \tang \otimes \tang }{|\tang|^2} \, \rd x
$$
is non-singular. The operator norm estimate
$$
\|A_{\tang} - A_{\sang}\|_{ {\rm op} } \leq 2 C^2\|\tang + \sang\|_{L^{2}(\T)}\| \tang - \sang\|_{L^{2}(\T)}
$$
also holds, and so $A^{-1}_{\sang}$ exists and obeys
\begin{align*}
\|A^{-1}_{\sang}\|_{ {\rm op} } &\leq \frac{ \| A^{-1}_{\tang} \|_{\rm op} }{1- \| A^{-1}_{\tang} \|_{\rm op} \|A_{\tang} - A_{\sang}\|_{ {\rm op} } }\\
\| A^{-1}_{\sang} - A^{-1}_{\tang} \|_{ {\rm op} } &\leq 2 C^{2}\| A^{-1}_{{\color{black}\sang}} \|_{ {\rm op} }\| A^{-1}_{\tang} \|_{ {\rm op} } \|\tang + \sang\|_{L^{2}(\T)}\| \tang - \sang\|_{L^{2}(\T)}
\end{align*}
whenever $\|\tang - \sang\|_{L^2(\T)}$ is sufficiently small.
\end{lemma}
\begin{proof}
See appendix, lemma \ref{lem:boundAA}
\end{proof}
\noindent In our next lemma we establish the needed estimates on the Lagrange multipliers themselves.
\begin{lemma}\label{lem:boundmult}
Let $f,g \in L^{1}(\T)$ denote arbitrary functions. Suppose $\tang,\sang \in H^{1}(\T)$ and that
$$
\min_{x \in \T} \, |\tang(x)| \geq \frac1{C},\quad \min_{x \in \T} \, |\sang(x)| \geq \frac1{C}
$$
for $1 \leq C < \infty$ a positive constant. Define
$$
\mu^{1}_{\tang} := \frac{ |\tang_x|^2 }{|\tang|^2}, \;\; \mu^{2}_{\tang} := \frac{\langle f + \lambda_{\tang} , \tang \rangle }{|\tang|^2}\qquad \mu^{1}_{\sang} :=  \frac{ |\sang_x|^2 }{|\sang|^2} , \;\; \mu^{2}_{\sang} := \frac{\langle g + \lambda_{\sang} , \sang \rangle }{|\sang|^2}
$$
and then let
\begin{align*}
 {\color{black}\lambda_{\tang} := A^{-1}_{\tang}} \fint_{\T} \left( \frac{ \langle f , \tang \rangle - |\tang_x|^2 }{|\tang|^2}\right)\tang \, \rd x \qquad \mu_{\tang} := \mu^{1}_{\tang} - \mu^{2}_{\tang} \\
 {\color{black}\lambda_{\sang} :=A^{-1}_{\sang}}\fint_{\T} \left( \frac{ \langle g , \sang \rangle - |\sang_x|^2 }{|\sang|^2}\right)\sang \, \rd x \qquad \mu_{\sang} := \mu^{1}_{\sang} - \mu^{2}_{\sang}
\end{align*}
denote the corresponding Lagrange multipliers. Then the Lagrange multiplier bounds
\begin{align}
&\| \mu^{1}_{\tang} \tang \|_{ L^{1}(\T) } \leq  C \|\tang\|^{2}_{H^{1}(\T)}, \qquad \|\mu^{2}_{\tang}\tang\|_{L^{1}(\T)} \leq 2\pi|\lambda_{\tang}| + \|f\|_{L^{1}(\T)}, \nonumber \\
&|\lambda_{\tang}| \leq \| A^{-1}_{\tang} \|_{ {\rm op} }\Big( \|f\|_{L^{1}(\T)} +  C  \|\tang\|^{2}_{H^{1}(\T)} \Big)
\end{align}
and the Lipschitz estimates
\begin{align}
&\| \mu^{1}_{\sang} - \mu^{1}_{\tang} \|_{L^{1}(\T)} \leq  C^{2} \Big( 1 + C^{2}\|\sang\|^{2}_{H^{1}(\T)} \Big) \| \tang + \sang\|_{H^{1}(\T)}\|\tang - \sang\|_{H^{1}(\T)} \nonumber \\
&\|\mu^{2}_{\tang} - \mu^{2}_{\sang}\|_{L^{1}(\T)} \leq C^{3} K_{1} \big( \|f-g\|_{L^{1}(\T)} + |\lambda_{\tang} - \lambda_{\sang}| + \|\tang - \sang\|_{H^{1}(\T)} \big)
\nonumber \\
&|\lambda_{\tang} - \lambda_{\sang}| \leq C^{3}\|A^{-1}_{\tang}\|_{ {\rm op} }K_{2}\Big( \|f-g\|_{L^{1}(\T)} + \|\tang - \sang\|_{H^{1}(\T)} + \|\mu^{1}_{\tang} - \mu^{1}_{\sang}\|_{L^{1}(\T)} + \| A^{-1}_{\tang} - A^{-1}_{\sang} \|_{ {\rm op} } \Big)
\end{align}
hold whenever both $A_{\tang}$ and $A_{\sang}$ are non-singular. The constant $K_1$ is a multi-linear function of the quantities
$$\|g\|_{L^{1}(\T)},|\lambda_{\sang}|,\|\tang\|_{H^{1}(\T)},\|\sang\|_{H^{1}(\T)}$$
alone, while the constant $K_{2}$ is a multi-linear function of the quantities
$$\|g\|_{L^{1}(\T)},\|\mu^{1}_{\sang}\|_{L^{1}(\T)},\|\tang\|_{H^{1}(\T)},\|\sang\|_{H^{1}(\T)}$$
alone.
\end{lemma}
\begin{proof}
Each inequality follows easily from the embedding $L^{\infty}(\T) \subset H^{1}(\T),$ and simple estimates of commutators via the triangle inequality.
\end{proof}
\noindent A straightforward combination of lemmas \ref{lem:boundF},\ref{lem:lipF},\ref{lem:boundA} and \ref{lem:boundmult} is thankfully all we need, so our estimation task is complete.

With these estimates in hand, we may now proceed with a demonstration of local existence and uniqueness. We begin by working with the class of \emph{mild solutions} $\tang \in C([0,T];H^{1}(\T))$ to the evolution. We first set
$$
G(\tang,\tang_x) := \F_{\gamma_{\tang} } + \lambda_{\tang} + \mu_{\tang} \tang, \qquad G:H^{1}(\T)\times L^{2}(\T) \mapsto L^{1}(\T)
$$
as the lower-order contribution to \eqref{eq:gflow}, and use the more compact notation
$$
\tang_{t} =  \tang_{xx} + G(\tang,\tang_x)
$$
to refer to the evolution equation. {\color{black}Following chapter 9 of \cite{McO03},} given an initial datum $\tang_0 \in H^{1}(\T),$ we then say that a continuous $H^{1}(\T)$-valued function $\tang \in C([0,T];H^{1}(\T))$ defines a mild solution to \eqref{eq:gflow} on $[0,T]$ if
$$
\tang(t) = S(t)[\tang_0] + \int^{t}_{0} S(t-s)[ G(\tang(s),\tang_x(s)) ]\, \rd s \quad \text{for all} \quad t \in [0,T].
$$
For a given $u \in L^{1}(\T)$ the notation $S(t)[u]$ refers to the action of the linear semi-group generated by the heat equation on $\T;$ as usual it acts according to
$$
S(0)[\phi] = \phi, \quad S(t)[\phi] = \int_{\T} H(x-y,t) \phi(y) \, \rd y \quad (t>0),
$$
where $H(z,t)$ denotes the periodic heat kernel.  We shall need the following proposition to get off the ground.
\begin{proposition}\label{prop:linear}
Suppose $f \in C([0,T];L^{1}(\T))$ for $T>0$ arbitrary. Then for any initial datum $u_0 \in H^{1}(\T)$ the linear initial value problem
\begin{equation}\label{eq:ez_lin}
u_{t} =  u_{xx} + f, \qquad u(x,0) = u_0(x)
\end{equation}
has a unique mild solution $u \in C([0,T]; H^{1}(\T) )$. For any $3/2 < \alpha < 2,$ the solution obeys the estimate
$$
\|u\|_{C([0,T];\dot{H}^{1}(\T))} \leq C(\alpha)\left( \|u_0\|_{ \dot{H}^1(\T) } + T^{1 - \frac{\alpha}{2}}\|f\|_{C([0,T];L^1(\T))}  \right)
$$
for $C(\alpha)$ a universal constant.
\end{proposition}
\begin{proof} See appendix, proposition \ref{prop:linearA} \end{proof}
\noindent By linearity, this proposition obviously affords us a simple stability estimate
\begin{equation}\label{eq:stab}
\|u-v\|_{C([0,T]; H^{1}(\T))} \leq C(\alpha)\left( \|u_0-v_0\|_{ \dot{H}^1(\T) } + T^{1 - \frac{\alpha}{2}}\|f_1 - f_2\|_{C([0,T];L^1(\T))}  \right)
\end{equation}
between distinct solutions $u,v \in C([0,T];H^{1}(\T) )$ as well. With this proposition in hand, we may proceed to establish a second intermediate result. It will allow us to apply a fixed-point argument in the class of mean zero, unit speed embeddings. This proposition is the main ingredient needed to prove local well-posedness. Throughout the remainder of this sub-section we use $B_{\delta}(\tang_0)$ denote the $\delta$-ball around $\tang_0$ in the $H^{1}(\T)$ norm.
\begin{proposition}\label{lem:firstexistence}
Let $f \in C([0,T];L^{1}(\T))$ for $T>0$ arbitrary. Suppose that $\fint_{\T} f \, \rd x = 0$ on $[0,T],$ and define Lagrange multipliers according to
$$
A_{\tang} \lambda_{\tang} = \fint_{\T} \left( \frac{ \langle f , \tang \rangle -  |\tang_x|^2 }{|\tang|^2}\right)\tang \, \rd x \qquad \mu_{\tang} = \frac{ |\tang_x|^2 - \langle f +\lambda_{\tang} , \tang \rangle}{|\tang|^2}.
$$
Let $\tang_0 \in H^{1}(\T)$ denote any initial datum with $\fint \tang_0 \, \rd x = 0$ and unit speed. Then for any $\delta>0$ there exists a time $0 < T_0 \leq T,$ depending only on $\tang_0, \delta$ and $\|f\|_{C([0,T];L^{1}(\T))},$ so that the initial value problem
$$
\tang_{t} =  \tang_{xx} + f + \lambda_{\tang} + \mu_{\tang} \tang, \;\; \tang(0) = \tang_0 \qquad \text{and} \qquad \tang(t) \in B_{\delta}(\tang_0) \quad \text{for all} \quad t \in [0,T_0]
$$
has a unique mild solution $\tang \in C([0,T_0],H^{1}(\T))$ with mean zero and unit speed.
\end{proposition}
\begin{proof}
As $\tang_0$ has unit speed and mean zero, it is necessarily non-constant. Thus $A^{-1}_{\tang_0}$ exists by lemma \ref{lem:boundA}, and moreover the properties
\begin{align*}
\min_{x \in \T} \, |\tang(x)| \geq \frac1{2}, \qquad A^{-1}_{\tang} \;\; \text{exists}, \qquad  \|A^{-1}_{\tang}\|_{ {\rm op} } \leq 2 \|A^{-1}_{\tang_0}\|_{ {\rm op} }
\end{align*}
hold for all $\tang \in B_{\delta}(\tang_0)$ provided $\delta>0$ is sufficiently small. By lemma \ref{lem:boundA}, the choice of $\delta$ depends only on $\|\tang_0\|_{L^{2}(\T)}$ and $\|A^{-1}_{\tang_0}\|_{ {\rm op}}$. Fix such a $\delta > 0$, and given a time parameter $0 < T_{0} \leq T$ let
$$
X_{T_0,\delta} := \left\{ \tang \in C([0,T_0];H^{1}(\T)) : \tang(t) \in B_{\delta}(\tang_0) \;\; \forall \; t \in [0,T_0] \right\}.
$$
For $\tang \in X_{T_0,\delta}$ consider the mapping $\mathcal{A}[\tang]$ defined by
$$
\tang \in X_{T_0,\delta} \qquad \overset{\mathcal{A}[\tang]}{\longmapsto} \qquad S(t)[\tang_0] + \int^{t}_{0} S(t-s)[ G(f(s),\tang(s),\tang_x(s)) ]\, \rd s,
$$
with $G(f,\tang,\tang_x) := f + \lambda_{\tang} + \mu_{\tang} \tang$ denoting the lower-order terms.
Lemma \ref{lem:boundmult} implies that the estimates
\begin{align*}
&\| G(f,\tang,\tang_x) \|_{C([0,T_0];L^{1}(\T))} \leq K \qquad \text{and} \\
&\| G(f,\tang,\tang_x) - G(f,\sang,\sang_x)\|_{ C([0,T_0];L^{1}(\T))} \leq K \| \tang - \sang \|_{C([0,T_0];H^{1}(\T))}
\end{align*}
hold whenever $\tang,\sang \in X_{\delta,T_0}$ and $0 < T_0 \leq T$ are arbitrary. The constant $K$ depends only on $\|A^{-1}_{\tang_0}\|_{ {\rm op}},\|\tang_0\|_{H^{1}(\T)}$ and $\|f\|_{C([0,T];L^{1}(\T))}$. By proposition \ref{prop:linear} {\color{black}(more specifically \eqref{eq:stab})}, if $\tang \in X_{T_0,\delta}$ and $3/2 < \alpha < 2$ then
$$
\|{\color{black}\mathcal{A}[\tang]} - S[\tang_0]\|_{C([0,T_0];H^{1}(\T))} \leq C(\alpha)T^{1 - \frac{\alpha}{2}}_{0}\|G(f,\tang,\tang_x)\|_{C([0,T];L^{1}(\T))} \leq C(\alpha)T^{ 1 - \frac{\alpha}{2}}_0 K.
$$
The semigroup $S(t)[\tang_0]$ also satisfies
$$
\|S[\tang_0] - \tang_0\|_{C([0,T_0];H^{1}(\T))} \leq \frac{\delta}{2}
$$
provided $T_0$ is sufficiently small, with the choice of $T_0$ depending only upon $\delta$ and $\tang_0$ itself. Thus $A[\tang] \in X_{T_0,\delta}$ provided $T_0$ is sufficiently small. For any such $T_0$, if $\tang,\sang \in X_{T_0,\delta}$ then proposition \ref{prop:linear} {\color{black} (more specifically \eqref{eq:stab})} also yields
\begin{align*}
\|A[\tang] - A[\sang]\|_{C([0,T_0];H^{1}(\T))} &\leq C(\alpha)T^{1 - \frac{\alpha}{2}}_{0}\|G(f,\tang,\tang_x) - G(f,\sang,\sang_x)\|_{C([0,T];L^{1}(\T))} \\
&\leq C(\alpha)T^{ 1 - \frac{\alpha}{2}}_0 K  \| \tang - \sang\|_{C([0,T_0];H^{1}(\T))}.
\end{align*}
The mapping $A[\tang]$ therefore defines a contraction on $X_{T_0,\delta},$ by taking $T_0$ smaller if necessary. For some $0 < T_0 \leq T$ sufficiently small, the equation
$$
\tang_t = \tang_{xx} + f + \lambda_{\tang} + \mu_{\tang} \tang
$$
therefore has a unique mild solution $\tang \in C([0,T_0];H^{1}(\T))$. Moreover, $\delta,K$ and therefore the choice of $T_0$ depends only upon $\tang_0$ and $\|f\|_{C([0,T];L^{1}(\T))}$ as claimed.

The remainder of the proof is devoted to showing that $\tang(t)$ has mean zero and unit speed. Fix $[0,T_0]$ as the existence interval and for  $0 < t \leq T_0$ let
$$\hat{\tang}_{k}(t) := \frac1{2\pi} \int_{\T} \tang(t) \mathrm{e}^{-ikx} \, \rd x$$
denote the $k^{ {\rm th}}$ Fourier coefficient of the solution. Similarly, let
$$
\hat{G}_{k}(s) := \frac1{2\pi} \int_{\T} G(f(s),\tang(s),\tang_x(s)) \mathrm{e}^{-ikx} \, \rd x
$$
denote the $k^{ {\rm th}}$ Fourier coefficient of the lower-order terms in the evolution. Then
\begin{equation}\label{eq:hateqn}
\hat{\tang}_{k}(t) = \mathrm{e}^{- k^2 t} \left( \hat{\tang}_k(0) + \int^{t}_{0} \mathrm{e}^{ k^{2}s} \hat{G}_{k}(s) \, \rd s \right)
\end{equation}
since $\tang$ defines a mild solution. In the case that $k = 0,$ this implies
$$
\hat{\tang}_{0}(t) = \int^{t}_{0} \hat{G}_{0}(s) \, \rd s
$$
since $\tang_0$ has mean zero. By definition of the Lagrange multipliers,
$$
\fint_{\T} \mu_{\tang}\tang \, \rd x = \fint_{\T} \left( \frac{ |\tang_x|^2 - \langle f + \lambda_{\tang} , \tang\rangle }{|\tang|^2} \right) \tang \, \rd x = -A_{\tang} \lambda_{\tang} - \left( \fint_{\T} \frac{\tang \otimes \tang}{|\tang|^2} \, \rd x \right)\lambda_{\tang} = -\lambda_{\tang}.
$$
The fact that $f$ has mean zero on $[0,T]$ then gives
$$
\hat{G}_{0}(s) = \fint_{\T} G(f(s),\tang(s),\tang_x(s)) \, \rd x = \fint_{\T} \Big(f + \lambda_{\tang} + \mu_{\tang} \tang\Big) \, \rd x = \lambda_{\tang} + \fint_{\T} \mu_{\tang} \tang \, \rd x = 0,
$$
and so $\tang(t)$ has mean zero on $[0,T_0]$ as desired.

To show that $\tang(t)$ has unit speed, first note that for each $s \in [0,T_0]$ the equality
\begin{equation}\label{eq:frr}
\langle G(f(s),\tang(s),\tang_x(s)) , \tang(s) \rangle =  |\tang_x(s)|^2
\end{equation}
holds in the $L^{1}(\T)$ sense. This identity follows once again by definition of the Lagrange multipliers. Now $\tang_{x}(s) \in L^{2}(\T)$ implies
$$
 \fint_{\T} |\tang_x(s)|^2 \mathrm{e}^{-ikx} \, \rd x = \sum_{ \ell \in \Z} - \ell(k-\ell)\langle \hat{\tang}_{\ell}(s),\hat{\tang}_{k-\ell}(s) \rangle
$$
by the convolution theorem. That $G(f(s),\tang(s),\tang_x(s)) \in L^{1}(\T)$ and $\tang(s) \in H^{1}(\T)$ also provides enough regularity to conclude
\begin{align*}
\fint_{\T} \langle G(f(s),\tang(s),\tang_x(s)) , \tang(s) \rangle \mathrm{e}^{-ikx} \, \rd x = \sum_{\ell \in \Z} \left< \hat{\tang}_{\ell}(s) , \hat{G}_{k - \ell}(s) \right>,
\end{align*}
which by the observation \eqref{eq:frr} yields

\begin{equation}\label{eq:sumsum}
\sum_{\ell \in \Z} \left< \hat{\tang}_{\ell}(s) , \hat{G}_{k - \ell}(s) \right> = - \sum_{\ell \in \Z} \ell(k-\ell) \langle \hat{\tang}_{\ell}(s) , \hat{\tang}_{k-\ell}(s) \rangle
\end{equation}
as an immediate consequence. If $t > 0$ and $k \neq 0$ the decay estimate
\begin{align*}
|\hat{\tang}_k(t)| &\leq \re^{- k^2 t}|\hat{\tang}_k(0)| +  k^{-2}\|G(f,\tang,\tang_x)\|_{C([0,T_0];L^{1}(\T))}\\
& \leq C\left( t^{-\frac1{2}} \|\tang_0\|_{H^{1}(\T)} + \|G(f,\tang,\tang_x)\|_{C([0,T_0];L^{1}(\T))}\right)k^{-2}
\end{align*}
also holds, simply by bounding and integrating \eqref{eq:hateqn} in a straightforward manner. Define
$$
z(t) := |\tang(t)|^2 \qquad \hat{z}_{k}(t) := \fint_{\T} z(t) \mathrm{e}^{-ikx} \, \rd x,
$$
and note that
$$
\hat{z}_k(t) = \sum_{\ell \in \Z} \langle \hat{\tang}_{\ell}(t) , \hat{\tang}_{k-\ell}(t) \rangle.
$$
But then \eqref{eq:hateqn} yields
\begin{align*}
\hat{z}^{\prime}_k(t) &= \sum_{\ell \in \Z} \langle \hat{\tang}^{\prime}_{\ell}(t) , \hat{\tang}_{k-\ell}(t) \rangle + \langle \hat{\tang}_{\ell}(t) , \hat{\tang}^{\prime}_{k-\ell}(t) \rangle\\
&= \sum_{\ell \in \Z} - \ell^2 \langle \hat{\tang}_{\ell}(t) , \hat{\tang}_{k-\ell}(t) \rangle -(k-\ell)^2 \langle \hat{\tang}_{\ell}(t) , \hat{\tang}_{k-\ell}(t) \rangle
+ \langle \hat{G}_{\ell}(t) , \hat{\tang}_{k-\ell}(t) \rangle + \langle \hat{\tang}_{\ell}(t) , \hat{G}_{k-\ell}(t) \rangle,
\end{align*}
since all four series {\color{black} in the last equation }converge absolutely for $t>0$ due to the fact that $\tang(s) \in H^{1}(\T)$ and the decay estimate. Combining this with the previous observation {\color{black}\eqref{eq:sumsum}} then shows
\begin{align*}
\hat{z}^{\prime}_k(t) &= - \sum_{\ell \in \Z} \big( \ell^2 + (k-\ell)^2 \big) \langle \hat{\tang}_{\ell}(t) , \hat{\tang}_{k-\ell}(t) \rangle + 2 \ell(k-\ell) \langle \hat{\tang}_{\ell}(t) , \hat{\tang}_{k-\ell}(t)  \rangle = - k^2 \hat{z}_k(t)
\end{align*}
as long as $t > 0$. For any Fourier coefficient $\hat{z}_{k}(t)$ and all $t \in [0,T_0]$ the relation
$$
\hat{z}_{k}(t) = \mathrm{e}^{- k^2 t} \hat{z}_{k}(0)
$$
therefore holds. Now $\tang_0$ has unit speed, and so $\hat{z}_{0}(0) = 1$ and $\hat{z}_k(0) = 0$ otherwise. Thus $|\tang(t)|^2 = 1$ on $[0,T_0]$ as desired.
\end{proof}

With this proposition in hand, we may now complete the local existence proof for \eqref{eq:gflow}. Let $\tang_0 \in H^{1}(\T)$ have mean zero and unit speed. Suppose that the induced knot $\gamma_{\tang_{0}} \in H^{2}(\T)$ defines a bi-Lipschitz embedding, so in particular its distortion
$$
\delta_{\infty}( \gamma_{\tang_{0}}) := \sup_{ x,y \in \T } \, \frac{ \dd(x-y)}{|\gamma_{\tang_{0}}(x) - \gamma_{\tang_{0}}(y)|}
$$
is finite. Recall now that the embedding $H^{2}(\T) \subset C^{0,1}(\T)$ holds, in that $\gamma \in H^{2}(\T)$ implies
$$
\sup_{x,y \in \T} \frac{ |\gamma(x) - \gamma(y)| }{\dd(x-y)} \leq C \|\ddot{\gamma}\|_{L^{2}(\T)}
$$
for $C>0$ some universal constant. Now take $\tang \in H^{1}(\T)$ with mean zero and unit speed, but otherwise arbitrary. Let $\gammat$ denote its induced embedding. It is then easy to show that a $\delta>0$ exists, depending only on $\delta_{\infty}( \gamma_{\tang_{0}}),$ so that
$$
\delta_{\infty}(\gammat) \leq 2 \delta_{\infty}(\gamma_{\tang_0}) \qquad \text{for all} \qquad \tang \in B_{\delta}(\tang_0),
$$
with $B_{\delta}(\tang_0)$ again denoting an $H^{1}(\T)$-ball. By lemmas {\color{black}\ref{lem:boundF} and \ref{lem:lipF}}, for any $\tang,\sang \in B_{\delta}(\tang_0)$ the uniform estimates
\begin{align*}
&\| \mathbf{F}_{\gammat} \|_{H^{1}(\T)} \leq C\big( \delta_{\infty}(\gamma_{\tang_0}) \big)
&\|\mathbf{F}_{\gammat} - \mathbf{F}_{\gamma_{\sang}}\|_{H^{1}(\T)} &\leq C\big( \delta_{\infty}(\gamma_{\tang_0}) , \|\tang_0\|_{H^{1}(\T)}\big)\| \tang - \sang \|_{H^{1}(\T)}
\end{align*}
therefore hold for the nonlocal forcings induced by any such embedding. By taking $\delta > 0$ smaller if necessary, we may ensure that
\begin{align*}
\min_{x \in \T} \, |\tang(x)| \geq \frac1{2}, \qquad A^{-1}_{\tang} \;\; \text{exists}, \qquad  \|A^{-1}_{\tang}\|_{ {\rm op} } \leq 2 \|A^{-1}_{\tang_0}\|_{ {\rm op} }
\end{align*}
hold for any $\tang \in B_{\delta}(\tang_0)$ as well.
Given such a $\delta>0$ and a time parameter $T>0$ let
$$
X_{T,\delta} := \left\{ \tang \in C([0,T];H^{1}(\T)) : \forall \, t \in [0,T], \;\; \fint_{\T} \tang(t) \, \rd x = 0, \;\; |\tang(t)|^2 =1, \;\; \tang(t) \in B_{\delta}(\tang_0) \right\},
$$
which always contains the constant curve $\tang(t) \equiv \tang_0,$ for instance. For any such $\sang \in X_{T,\delta},$ the corresponding nonlocal forcings $\mathbf{F}_{\gamma_{\sang}}$ satisfy a uniform bound
\begin{equation}\label{eq:unif}
\| \mathbf{F}_{ \gamma_{\sang} } \|_{C([0,T];H^{1}(\T))} \leq C\big( \delta_{\infty}(\gamma_{\tang_0}) \big),
\end{equation}
and so by proposition \ref{lem:firstexistence} the corresponding equation
\begin{equation}\label{eq:xxx}
\tang_{t} =  \tang_{xx} + \mathbf{F}_{ \gamma_{\sang} } + \lambda_{\tang} + \mu_{\tang} \tang \qquad \tang(0) = \tang_0
\end{equation}
has a unique mild solution $\tang \in X_{T,\delta}$ for some $T>0$ sufficiently small. By Proposition \ref{lem:firstexistence}, the uniform bound \eqref{eq:unif} is sufficient to guarantee that the interval of existence $[0,T]$ is uniform over the choice of $\sang \in X_{T,\delta}$, and so the mapping $\mathcal{A}[\sang]$ defined by sending $\sang \in X_{T,\delta}$ to $\tang \in X_{T,\delta},$ with $\tang$ the unique solution of \eqref{eq:xxx}, is therefore well-defined. Set
$$
\Lambda_{\tang} := \lambda_{\tang} + \mu_{\tang} \tang
$$
as the Lagrange multipliers of the solution to \eqref{eq:xxx}. By proposition \ref{prop:linear} {\color{black}(specifically \eqref{eq:stab})}, if $3/2 < \alpha < 2$ the comparison estimate
$$
\|\mathcal{A}[\tang] - \mathcal{A}[\sang]\|_{ C([0,T];H^{1}(\T)) } \leq C(\alpha)T^{1 - \frac{\alpha}{2}}\Big( \| \mathbf{F}_{\gamma_{\tang}} - \mathbf{F}_{\gamma_{\sang}}\|_{C([0,T];L^{1}(\T))} + \| \Lambda_{\mathcal{A}[\tang]} - \Lambda_{\mathcal{A}[\sang]} \|_{C([0,T];L^{1}(\T))}\Big)
$$
therefore holds. Lemmas \ref{lem:boundA} and \ref{lem:boundmult} combine with the fact that $\mathcal{A}[\tang],\mathcal{A}[\sang] \in B_{\delta}(\tang_0)$ on $[0,T]$ to yield
$$
\| \Lambda_{\mathcal{A}[\tang]} - \Lambda_{\mathcal{A}[\sang]} \|_{C([0,T];L^{1}(\T))} \leq K\Big( \|\mathcal{A}[\tang] - \mathcal{A}[\sang]\|_{C([0,T];H^{1}(\T))} + \| \mathbf{F}_{\gamma_{\tang}} - \mathbf{F}_{\gamma_{\sang}}\|_{C([0,T];L^{1}(\T))} \Big)
$$
for $K$ some constant depending on $\delta,\|A^{-1}_{\tang_0}\|_{ {\rm op}},\|\tang_0\|_{H^{1}(\T)}$ and $C\big( \delta_{\infty}(\gamma_{\tang_0}) \big)$ alone. By taking $T>0$ smaller if necessary we may assume $C(\alpha)K T^{1-\frac{\alpha}{2}} \leq 1/2,$ and so the estimate
$$
\|\mathcal{A}[\tang] - \mathcal{A}[\sang]\|_{ C([0,T];H^{1}(\T)) } \leq 2C(\alpha)T^{1 - \frac{\alpha}{2}}\Big( (1+K)\| \mathbf{F}_{\gamma_{\tang}} - \mathbf{F}_{\gamma_{\sang}}\|_{C([0,T];L^{1}(\T))} \Big)
$$
follows. Finally, the Lipschitz estimates of lemma \ref{lem:lipF} show that
$$
\| \mathbf{F}_{\gamma_{\tang}} - \mathbf{F}_{\gamma_{\sang}}\|_{C([0,T];L^{1}(\T))}  \leq K_{2} \|\tang - \sang\|_{H^{1}(\T)}
$$
where $K_{2}$ depends on $\delta_{\infty}(\gamma_{\tang_0})$ and $\|\tang_0\|_{H^{1}(\T)}$ alone. Thus $\mathcal{A}$ contracts on $X_{T,\delta}$ for $T>0$ sufficiently small, finally yielding local existence. All together, we have shown
\begin{theorem}\label{thm:local}
Assume that $K(u,v)$ is $h/p$ homogeneous and either $0$-degenerate or $m$-degenerate. Let $\tang_0 \in H^{1}(\T)$ have mean zero and unit speed, and suppose the induced knot $\gamma_{\tang_0} \in H^{2}(\T)$ defines a bi-Lipschitz embedding. Then there exists a $T>0$ so that \eqref{eq:gflow} has a unique mild solution $\tang \in C([0,T];H^{1}(\T))$ with mean zero and unit speed. Moreover, for all $t \in [0,T]$ the induced knot $\gamma_{\tang(t)} \in H^{2}(\T)$ defines a bi-Lipschitz embedding.
\end{theorem}

As with the energetic bound furnished by lemma \ref{lem:distbound}, lower-order modifications of the kernel $K(u,v)$ do not affect the validity of this local well-posedness result. We have only used the $h/p$ homogeneity and degeneracy hypotheses to show that the nonlocal forcings $\mathbf{F}_{\gammat}$ obey the estimates
\begin{align*}
\|\mathbf{F}_{\gamma}\|_{H^{1}(\T)} &\leq C( \delta_{\infty}(\gamma) , \|\gamma\|_{H^{2}(\T)} ) \nonumber \\
\|\mathbf{F}_{\gamma} - \mathbf{F}_{\phi}\|_{H^{1}(\T)} &\leq C( \delta_{\infty}(\gamma) , \delta_{\infty}(\phi),\|\gamma\|_{H^{2}(\T)},\|\phi\|_{H^{2}(\T)} )\| \gamma - \phi \|_{H^{2}(\T)}
\end{align*}
when restricted to unit speed curves, and the local existence argument only requires these estimates. Thus any forcing $\mathbf{F}_{\gammat}$ that obeys such inequalities for unit speed curves yields local well-posedness of the corresponding flow. In particular, kernels $K(u,v) = K_{0}(u,v) + K_{1}(u,v)$ corresponding to smooth perturbations $K_{1}(u,v)$ of an $h/p$ homogeneous and degenerate kernel are perfectly allowable and do not affect the conclusion of theorem \ref{thm:local}.

\subsection{Regularity and Global Existence}
We may now turn our focus toward more pressing analytical issues regarding the evolution
\begin{equation}\label{eq:gflow2}
\tang_t = \tang_{xx} + \mathbf{F}_{\gamma_{\tang}} + \lambda_{\tang} + \mu_{\tang} \tang,
\end{equation}
now that local-in-time well-posedness of mild solutions is established. We first show that solutions exist in the classical sense and exist for as long as the embedding $\gamma_{\tang(t)}$ remains bi-Lipschitz. Under further hypotheses on the nonlocal kernel $K(u,v)$ we show that classical solutions exist globally and are smooth for all time. The following slight refinement of lemma \ref{lem:boundF} will allow us to accomplish these tasks.

\begin{lemma}\label{lem:diffF}
Assume that $K(u,v)$ is $h/p$ homogeneous and degenerate. Assume that $\gamma \in H^{3}(\T)$ is a unit speed, bi-Lipschitz embedding. Then the nonlocal integral
$$
f_{ \gamma}(x) :=  \lim_{\epsilon \downarrow 0} \; \int_{ \T \cap \{\dd(x-y) \geq \epsilon\} } K_{u}\left( |\gamma(x) - \gamma(y)|^2, \dd^{2}(x-y) \right)(\gamma(x) - \gamma(y) ) \, \rd y
$$
lies in $H^{1}(\T)$ and obeys the estimate
$$
\| f^{\prime}_{\gamma} \|_{L^{2}(\T)} \leq C( \delta_{\infty}(\gamma) , \| \gamma\|_{\dot{H}^{3}(\T)} )
$$
for $C(a,b)$ a continuous function of its arguments.
\end{lemma}
\begin{proof} See appendix, lemma \ref{lem:diffFA} \end{proof}

We may now turn to the task of proving an additional regularity assertion for mild solutions. We begin this task by recording a few observations. Suppose that $\tang \in C([0,T];H^{1}(\T))$ defines a mild solution with $\gamma_{\tang(t)}$ bi-Lipschitz on the entire interval of existence. Thus $\tang(t)$ has mean zero and unit speed on $[0,T]$ and so the quantity
$$
A_{*} := \max_{[0,T]} \ \|A^{-1}_{\tang(t)}\|_{ {\rm op} }
$$
is necessarily finite. This observation follows immediately due to continuity of $\tang(t)$ in the $L^{\infty}(\T)$ norm and the fact that each $A^{-1}_{\tang(t)}$ exists on the full interval of existence. Similarly, let
$$
H_{*} := \max_{ [0,T] } \, \| \tang(t) \|_{H^{1}(\T)}
$$
denote the maximal $H^{1}(\T)$ norm of the solution over the interval of existence. Finally, define
$$
\delta_{*} := \max_{ [0,T] } \, \delta_{\infty}\big( \gamma_{\tang(t)} \big) < +\infty
$$
as the maximal distortion of any induced embedding $\gamma_{\tang(t)}$ over the interval of existence. The finiteness of $H_{*}$ and $\delta_{*}$ follows immediately from the {assumption} that $\tang \in C([0,T];H^{1}(\T))$ defines a mild solution with $\gamma_{\tang(t)}$ bi-Lipschitz on the entire interval of existence. Once again we use
$$
\hat{\tang}_{k}(t) := \fint_{\T} \tang(t) \mathrm{e}^{-ikx} \, \rd x
$$
to denote the $k^{ {\rm th}}$ Fourier coefficient of the solution, and for $k \neq 0$ we let
$$
\hat{\mathbf{F}}_{k}(t) := \fint_{\T} \Big( \mathbf{F}_{\gamma_{\tang}}(t) - \langle \mathbf{F}_{\gamma_{\tang}}(t) + \lambda_{\tang}(t),\tang(t) \rangle\tang(t) \Big) \mathrm{e}^{-ikx} \, \rd x, \qquad \hat{\mu}^{1}_{k}(t) := \fint_{\T} |\tang_x(t)|^2 \tang(t) \, \mathrm{e}^{-ikx} \, \rd x
$$
denote the corresponding coefficients of the regular and singular forcing components, respectively. As $\tang(t) \in H^{1}(\T)$ and $\gamma_{\tang(t)}$ is bi-Lipschitz on $[0,T],$ we have the a-priori uniform $H^{1}(\T)$ bound for the regular component
$$
\max_{[0,T]} \, \sum_{k \in \Z} |k|^2 |\hat{\mathbf{F}}_{k}(t)|^{2} \leq C\big(A_*,\delta_{*},H_{*} )
$$
by lemmas \ref{lem:boundF} and \ref{lem:boundmult}. As $|\tang(t)|^2 = 1$ and $\tang(t) \in H^{1}(\T),$ the uniform estimate
$$
\max_{[0,T]} \; \sup_{k \in \Z} |\hat{\mu}^{1}_{k}(t)| \leq C\big(H_{*}\big)
$$
trivially holds as well. Now take $0 < \epsilon < T$ for $\epsilon > 0$ arbitrary and observe that the equalities
\begin{align}\label{eq:reg1}
\hat{\tang}_{k}(t) &= \mathrm{e}^{ -  k^2 t}\left( \hat{\tang}_k(0) + \int^{t}_{0} \mathrm{e}^{ k^2 s} \big( \hat{\mathbf{F}}_{k}(s)  + \hat{\mu}^{1}_{k}(s) \big) \, \rd s\right) \nonumber \\
\hat{\tang}_{k}(t) &= \mathrm{e}^{ -  k^2 (t-\epsilon)}\left( \hat{\tang}_k(\epsilon) + \int^{t}_{\epsilon} \mathrm{e}^{ k^2 (s-\epsilon)} \big( \hat{\mathbf{F}}_{k}(s)  + \hat{\mu}^{1}_{k}(s) \big) \, \rd s\right)
\end{align}
hold for $t \geq \epsilon$ by definition of a mild solution.

We now plan to combine the equalities \eqref{eq:reg1} with lemma \ref{lem:sima} and a bootstrap argument to show that higher derivatives of $\tang(t)$ exist for positive times. To start this process we recall that for any $t>0$ the estimate
$$
|\hat{\tang}_{k}(t)| \leq \left( C t^{-\frac1{2}}\|\tang_0\|_{H^{1}(\T)} + C\big(A_*,\delta_{*},H_{*} )|k|^{-1}  + C\big(H_{*}\big) \right)k^{-2}
$$
follows by direct integration of \eqref{eq:reg1}. Applying lemma \ref{lem:sima} with the choices {\color{black} $u = \tang_x$ (whose coefficients $k\hat{\tang}_k$ decay like $|k|^{-1}$) and $v = \tang$ (whose coefficients decay like $k^{-2}$) yields
$$
| \hat{\mathbf{w}}_{k} | \leq C(A_{*},\delta_{*},H_{*})(1+t^{-1})|k|^{-1} \qquad (\mathbf{w} = \tang \tang_{x}),
$$
and so we may appeal to lemma \ref{lem:sima} once again with the choices $u = \tang_x$ (decaying like $|k|^{-1}$) and $v = (\tang \tang_x)$ (decaying like $|k|^{-1}$) to establish 
the additional decay estimate
\begin{equation}\label{eq:decay00}
|\hat{\mu}^{1}_{k}(t)| \leq C\big(A_*,\delta_{*},H_{*} )\big( 1 + t^{-\frac{3}{2}} \big) |k|^{-1}\log(1+|k|)
\end{equation}
for $t>0$ arbitrary. Now apply the inequalities $|k||\hat{\mathbf{F}_{k}}(s)| \leq C(A_{*},\delta_{*},H_{*})$ and \eqref{eq:decay00} to \eqref{eq:reg1} and analytically integrate the resulting exponentials to find
\begin{align}\label{eq:reg2}
|\hat{\tang}_{k}(t)| &\leq C \left( |k|^{-1}(t-\epsilon)^{-\frac1{2}}|\hat{\tang}_{k}(\epsilon)| + C\big(A_*,\delta_{*},H_{*} )|k|^{-3}  + O\left( |k|^{-3}\log(1+|k|) \left(1+\epsilon^{-\frac{3}{2}}\right) \right) \right) \nonumber \\
&\leq C\big(\epsilon,A_*,\delta_{*},H_{*} \big) |k|^{-3}\log(1+|k|)
\end{align}
for $2\epsilon < t \leq T$ and $\epsilon>0$ arbitrary. But then $\hat{\tang}_{k}$ decays faster than $|k|^{-p}$ for any $p < 3,$ and so we may remove the logarithmic factor in \eqref{eq:decay00}, and thus in \eqref{eq:reg2}, by iterating this argument once more}. For any $0 < 2\epsilon < t \leq T$ the series
$$
\sum_{k \in \Z} |k|^{4} |\hat{\tang}_{k}(t)|^{2}
$$
therefore converges. As $\epsilon > 0$ was arbitrary, we conclude that $\tang(t) \in H^{2}(\T)$ for any $t>0$ and also that the estimates
\begin{equation}\label{eq:h2bound}
\|\tang(t)\|_{H^{2}(\T)} \leq C\big(\epsilon,A_*,\delta_{*},H_{*}) \qquad \sup_{k \in \Z} |k|^{3}|\hat{\tang}_k(t)| \leq C\big(\epsilon,A_*,\delta_{*},H_{*} )
\end{equation}
hold uniformly on $[\epsilon,T]$ for $\epsilon>0$ arbitrary.

We conclude by iterating this argument once more. By appealing to lemma \ref{lem:diffF}, the uniform $H^{2}(\T)$ bound \eqref{eq:h2bound} on $\tang(t)$ suffices to show that
$$
\| \mathbf{F}_{\gamma_{\tang}}(t) \|_{H^{2}(\T)} \leq C\big(\epsilon,A_*,\delta_{*},H_{*} )
$$
uniformly on $[\epsilon,T]$ as well. Thus the function
$$
\mathbf{F}_{\gamma_{\tang}}(t) - \langle \mathbf{F}_{\gamma_{\tang}}(t) + \lambda_{\tang}(t) , \tang(t) \rangle \tang(t)
$$
lies in $H^{2}(\T)$ with a similar uniform bound, due to the fact that $H^{2}(\T)$ defines a Banach algebra in one spatial dimension. We therefore conclude that the inequality
$$
\max_{ [\epsilon,T] } |k|^{4} |\hat{\mathbf{F}}_k(t)|^2 \leq C\big(\epsilon,A_*,\delta_{*},H_{*} )
$$
holds for any $\epsilon > 0$ arbitrary. From \eqref{eq:h2bound} and lemma \ref{lem:sima} the estimate
$$
\max_{ [\epsilon,T] } |k|^{2}| \hat{\mu}^{1}_{k}(t)| \leq C\big(\epsilon,A_*,\delta_{*},H_{*} )
$$
follows in a similar fashion. Return once again to \eqref{eq:reg1} and integrate to conclude that
\begin{equation}\label{eq:evenbetter}
|\hat{\tang}_k(t)| \leq C\big(\epsilon,A_*,\delta_{*},H_{*} )|k|^{-4},
\end{equation}
on $[\epsilon,T]$, and so in particular the $H^{3}(\T)$ bound
$$
\sum_{k \in \Z} |k|^{6} |\hat{\tang}_k(t)|^{2} \leq C\big(\epsilon,A_*,\delta_{*},H_{*} )
$$
holds on this interval as well.

These bounds then prove more than sufficient for showing that mild solutions are, in fact, classical. If $t>0$ we have that
$$
\hat{\tang}^{\prime}_k(t) = -k^2 \hat{\tang}_k(t) + \hat{\mathbf{F}}_{k}(t) + \hat{\mu}^{1}_{k}(t),
$$
and in particular a term-by-term differentiation of the series
$$
\sum_{k \in \Z} \hat{\tang}_k(t)  \mathrm{e}^{ikx}
$$
is quite easily justified by the previous decay estimates. The original equation
$$
\tang_{t} = \tang_{xx} + \mathbf{F}_{\gamma_{\tang}} + \lambda_{\tang} + \mu_{\tang} \tang
$$
therefore holds for $0 < t \leq T$ in the classical sense. {\color{black} Furthermore, if $t_2 > t_1 \geq \epsilon > 0$ then
\begin{align*}
\hat{\tang}_{k}(t_2) - \hat{\tang}_{k}(t_1) = \left( \re^{-k^2(t_2 - t_1)} - 1 \right)\hat{\tang}_{k}(t_1) + \re^{-k^{2}(t_2-t_{1})}\int^{t_2}_{t_1}\re^{k^2(s-t_1)} \big( \hat{\mathbf{F}}_{k}(s)  + \hat{\mu}^{1}_{k}(s) \big) \, \rd s
\end{align*}
and as $t_2 \to t_1$ both terms on the right hand side converge to zero in the $H^{3}(\T)$ norm. For the first term this follows easily from the fact that
$$
\left( 1 - \re^{-k^2(t_2 - t_1)} \right)^{2} |k|^{6} |\hat{\tang}_{k}(t_1)|^{2} \leq |k|^{6} |\hat{\tang}_{k}(t_1)|^{2},
$$
the fact that $\tang(t_1) \in H^{3}(\T)$ and the dominated convergence theorem. For the second term, we may use the inequality
$$
|k|^{2}|\hat{\mathbf{F}}_k(t) + \hat{\mu}^{1}_{k}(t)| \leq C\big(\epsilon,A_*,\delta_{*},H_{*})
$$
and a direct integration to show
$$
e^{-k^{2}(t_2 - t_1)} \left| \int^{t_2}_{t_1} \re^{k^2(s-t_{1})}\big( \hat{\mathbf{F}}_{k}(s)  + \hat{\mu}^{1}_{k}(s) \big) \, \rd s\right| \leq C\big(\epsilon,A_*,\delta_{*},H_{*})k^{-4}
$$
and then apply the dominated convergence theorem once again. We conclude 
$$
\| \tang(t_1) - \tang(t_2) \|_{H^{3}(\T)} \to 0 \qquad \text{as} \qquad |t_1 - t_2| \to 0,
$$
and so $\tang(t)$ varies continuously in the $H^{3}(\T)$ sense.} That $\mathbf{F}_{\gamma_{\tang}}(t)$ varies continuously with respect to $t$ in the $H^{1}(\T)$ norm follows from lemma \ref{lem:lipF} and the continuity of $\tang(t)$ in the $H^{1}(\T)$ norm. The $H^{3}(\T)$ continuity of $\tang(t)$ then shows that
$$\mu_{\tang}(t)\tang(t) =  |\tang_x(t)|^2 \tang(t) - \langle \mathbf{F}_{\gamma_{\tang}}(t) + \lambda_{\tang}(t), \tang(t) \rangle$$
varies continuously in $H^{1}(\T)$ as well. Thus $\tang_{t}(t)$ varies continuously in $H^{1}(\T)$ by the differential equation itself. These observations combine to yield
\begin{theorem}\label{thm:regularity1}
Suppose $K(u,v)$ satisfies the hypotheses of lemma \ref{lem:boundF}, and let $\tang \in C([0,T];H^{1}(\T))$ denote a mild solution of
\begin{equation}\label{eq:thmflow1}
\tang_{t} =  \tang_{xx} + \mathbf{F}_{\gamma_{\tang}} + \lambda_{\tang} + \mu_{\tang} \tang, \qquad \tang(0) = \tang_{0} \in H^{1}(\T),
\end{equation}
where $\tang(t)$ has mean zero and unit speed. Suppose that each $\gamma_{\tang(t)}$ defines a bi-Lipschitz embedding, and that
$$
A_{*} := \max_{[0,T]} \ \|A^{-1}_{\tang(t)}\|_{ {\rm op} }, \quad \delta_{*} :=  \max_{ [0,T] } \, \delta_{\infty}\big( \gamma_{\tang(t)} \big), \quad H_{*} := \max_{[0,T]} \|\tang(t)\|_{H^{1}(\T)}
$$
remain finite over the interval $[0,T]$ of existence. Then $\tang \in C^{1}((0,T];H^{1}(\T)) \cap C((0,T];H^{3}(\T)),$ and for each $0 < t \leq T$ the differential equation \eqref{eq:thmflow1} holds in the classical sense. Moreover, for any $\epsilon > 0$ the a-priori $H^{3}(\T)$ and Fourier-decay bounds
$$
\|\tang(t)\|_{H^{3}(\T)} \leq C\big(\epsilon,A_*,\delta_{*},H_{*} ) \qquad \sup_{k \in \Z} |k|^{4}|\hat{\tang}_k(t)| \leq C\big(\epsilon,A_*,\delta_{*},H_{*} )
$$
hold for any $t \in [\epsilon,T]$ arbitrary.
\end{theorem}
\noindent If the nonlocal integral $f_{\gamma_{\tang}}$ obeys an additional regularity hypothesis then we may obviously continue iterating along these lines. For instance, if we can augment lemmas \ref{lem:boundF} and \ref{lem:diffF} with addition estimates of the form
\begin{equation}\label{eq:smoothF}
\| f^{(n)}_{\gamma} \|_{L^{2}(\T)} \leq C(n,\delta_{\infty}(\gamma) , \| \gamma \|_{ \dot{H}^{n+2}(\T) } )
\end{equation}
then the corresponding $H^{n+2}(\T)$ bounds
$$
\|\tang(t)\|_{H^{n+2}(\T)} \leq C\big(n,\epsilon,A_{*},\delta_{*},H_{*}) \qquad \sup_{k \in \Z} |k|^{n+3}|\hat{\tang}_k(t)| \leq C\big(n,\epsilon,A_{*},\delta_{*},H_{*})
$$
clearly hold for $t \in [\epsilon,T]$ and $n$ arbitrary. Establishing \eqref{eq:smoothF} requires additional assumptions on the kernel $K(u,v)$, however. As an example, if we modify the $0$-degeneracy or $m$-degeneracy assumptions to include the statement $g \in C^{\infty}([1,\alpha))$ then \eqref{eq:smoothF} holds in general. The proof of this fact mirrors the proof of lemma \ref{lem:diffF} closely. Aside from a more elaborate invocation of the chain rule, the argument simply requires proving that the $j^{ {\rm th}}$ order derivatives of
$$
\eta(x,z) := |\gamma(x+z) - \gamma(x)|^2
$$
decay like $|z|^{3}$ near the origin. Lemma \ref{lem:diffF} proves this for first-order derivatives; the bound for higher-order derivatives then follows by Taylor's theorem and a simple induction. As our regularity results to this point more than suffice for our purposes, we simply state the result of \eqref{eq:smoothF} as a corollary then turn to address the issue of global existence.
\begin{corollary}\label{cor:smooth}
Assume that the nonlocal integral $f_{\gamma}$ obeys
$$
\| f^{(n)}_{\gamma} \|_{L^{2}(\T)} \leq C(n,\delta_{\infty}(\gamma) , \| \gamma \|_{ \dot{H}^{n+2}(\T) } )
$$
for each $n \in \mathbb{N}$ whenever $\gamma \in H^{n+2}(\T)$ has mean zero and unit speed. Let $\tang \in C([0,T];H^{1}(\T))$ denote a mild solution of
\begin{equation}\label{eq:thmflow}
\tang_{t} = \tang_{xx} + \mathbf{F}_{\gamma_{\tang}} + \lambda_{\tang} + \mu_{\tang} \tang, \qquad \tang(0) = \tang_{0} \in H^{1}(\T),
\end{equation}
where $\tang(t)$ has mean zero and unit speed. Suppose that each $\gamma_{\tang(t)}$ defines a bi-Lipschitz embedding and so
$$
A_{*} := \max_{[0,T]} \ \|A^{-1}_{\tang(t)}\|_{ {\rm op} }, \quad \delta_{*} :=  \max_{ [0,T] } \, \delta_{\infty}\big( \gamma_{\tang(t)} \big), \quad H_{*} := \max_{[0,T]} \|\tang(t)\|_{H^{1}(\T)}
$$
remain finite over the interval $[0,T]$ of existence. Then on $(0,T]$ the solution $\tang(t)$ is smooth in the space variable, and for any $\epsilon > 0$ and any $n \in \mathbb{N}$ the bounds
$$
\|\tang(t)\|_{H^{n+2}(\T)} \leq C\big(n,\epsilon,A_{*},\delta_{*},H_{*}) \qquad \sup_{k \in \Z} |k|^{n+3}|\hat{\tang}_k(t)| \leq C\big(n,\epsilon,A_{*},\delta_{*},H_{*})
$$
hold for any $t \in [\epsilon,T]$ arbitrary.
\end{corollary}

We complete our analysis of existence by showing that a-priori $H^{1}(\T)$ and distortion bounds allow for the global continuation of solutions. These a-priori bounds will eventually follow from energy descent. While providing the details of such a continuation argument is perhaps overkill, we nevertheless provide one in the appendix for the sake of being careful. As always suppose that $\tang_0 \in H^{1}(\T)$ has zero mean, unit speed and induces a bi-Lipschitz embedding. Let us once again define
$$
G(\tang,\tang_x) := \mathbf{F}_{\gamma_{\tang}} + \lambda_{\tang} + \mu_{\tang}\tang
$$
as the lower-order forcing in \eqref{eq:thmflow} and set $\mathcal{C}(\tang_0) \subset (0,\infty)$ as the possible existence intervals for mild solutions generated by this initial datum. More specifically, we put
$$
\mathcal{C}(\tang_0) := \left\{ T > 0 : \exists \tang \in C([0,T];H^{1}(\T)), \,\forall t \in [0,T] \;\; \tang(t) = S(t)[\tang_0] + \int^{t}_{0} S(t-s)G(\tang(t),\tang_{x}(s))\, \rd s \right\}.
$$
Local existence implies that $\mathcal{C}(\tang_0)$ is non-empty, so let $T_{*} := \sup\, \mathcal{C}(\tang_0)$ denote its positive supremum. We refer to $T^{*}$ as the \emph{maximal interval of existence}. With this definition in place, we have the following standard continuation result.

\begin{corollary}\label{cor:continuation}
Suppose $K(u,v)$ satisfies the hypotheses of lemma \ref{lem:boundF}, that $\tang_0 \in H^{1}(\T)$ has zero mean, unit speed and induces a bi-Lipschitz embedding. Let $T^{*}$ denote the maximal interval of existence of a mild solution to
\begin{equation*}
\tang_{t} =  \tang_{xx} + \mathbf{F}_{\gamma_{\tang}} + \lambda_{\tang} + \mu_{\tang} \tang, \qquad \tang(0) = \tang_{0} \in H^{1}(\T).
\end{equation*}
If $T^{*} < \infty$ then one of
$$
\limsup_{t \nearrow T^{*}} \, \|\tang(t)\|_{H^{1}(\T)} = + \infty \qquad \text{or} \qquad \limsup_{ t \nearrow \infty } \, \delta_{\infty}\left( \gamma_{ \tang(t)} \right) = +\infty
$$
must hold.
\end{corollary}
\begin{proof} See appendix, corollary \ref{cor:continuationA} \end{proof}
\noindent In other words, global existence of mild solutions follows from uniform $H^{1}(\T)$ and distortion bounds.

\subsection{Energy Descent and Long-Term Behavior}
Corollary \ref{cor:continuation} reduces the task of proving global existence to the task of proving a-priori bounds. Under the proper hypothesis on the kernel $K(u,v)$ these bounds follow from descent of the energy
$$
E(t) := E_{ {\rm b,K} }(\tang(t)) = \frac1{2} \int_{ \T} | \dot{\tang}(x,t)|^2 \, \rd x + \frac1{4} \int_{\T \times \T} K\left( |\gamma_{\tang(t)}(x) - \gamma_{\tang(t)}(y)|^2 , \dd^2(x-y) \right) \, \rd x \rd y
$$
inducing the flow, a property that we now establish. The only possible subtlety here arises from the presence of Lagrange multipliers, but in fact we can easily show that the usual energy-decay relation
$$
E^{\prime}(t) = - \| \tang_{t}(t) \|^{2}_{L^{2}(\T)}
$$
for an $L^{2}(\T)$ gradient flow remains valid even in their presence. To see this, assume $h/p$ homogeneity and either $0$-degeneracy or $m$-degeneracy of $K(u,v)$ as always. The induced curves $\gamma_{\tang(t)}$ then lie in $C^{1}( (0,T];H^{2}(\T))$ by theorem \ref{thm:regularity1}, and in particular both $\gamma(x,t)$ and $\tang(x,t)$ are continuously differentiable in both variables for positive times. Define $\eta(x,y,t)$ as $\eta(x,y,t) := K\left( |\gamma(x,t) - \gamma(y,t)|^2 , \dd^{2}(x-y) \right),$ so that if $\dd(x-y)>0$ then the relations
\begin{align*}
\eta_{t}(x,y,t) &= 2 K_{u}\left( |\gamma(x,t) - \gamma(y,t)|^2 , \dd^{2}(x-y) \right)\langle \gamma(x,t) - \gamma(y,t) , \partial_{t} \gamma(x,t) - \partial_{t} \gamma(y,t) \rangle \\
&= 2 g( \alpha(x,y) ) \dd^{-2(p+1)}(x-y) \langle \gamma(x,t) - \gamma(y,t) , \partial_{t} \gamma(x,t) - \partial_{t} \gamma(y,t) \rangle
\end{align*}
hold, where $\alpha(x,y) := \dd^{2}(x-y)/|\gamma(x,t)-\gamma(y,t)|^2$ as before. After noting that $\gamma_{t}$ is Lipschitz in space, since $\gamma_{tx} = \tang_{t}$ lies in $H^{1}(\T) \subset C^{0}(\T)$, we may argue as in lemma \ref{lem:boundF} to show that
\begin{equation}\label{eq:timediff}
|\eta_{t}(x,y,t)| \leq C( \delta_{\infty}(\gamma_{\tang(t)}),\|\tang_{t}\|_{L^{2}(\T)}) \dd^{-q}(x-y)
\end{equation}
for $q < 1$ some integrable exponent. A few simple calculations then give
\begin{align*}
\frac{\rd}{\rd t} \left( E_{K}( \gamma_{\tang(t)} ) \right) &= \frac1{4} \int_{\T \times \T} \eta_{t}(x,y,t) \, \rd y \rd x = \frac1{4} \int_{\T} \lim_{\epsilon \downarrow 0} \left( \int_{\dd(x-y)\geq \epsilon } \eta_{t}(x,y,t) \, \rd y \right) \, \rd x \\
&= \int_{\T} \lim_{\epsilon \downarrow 0} \langle \gamma_{t}(x,t) , f_{\gamma,\epsilon}(x) \rangle \, \rd x = \langle \gamma_t, f_{\gamma} \rangle_{L^{2}(\T)} = - \langle \tang_{t}(t) , \F_{\gamma_{\tang(t)}} \rangle_{L^{2}(\T)},
\end{align*}
where \eqref{eq:timediff} justifies differentiating under the integral and passing to the limit in the first and second equalities. The third {\color{black} equality} follows by symmetry of the integrand defining $f_{\gamma,\epsilon},$ the fourth equality follows from the $L^{2}(\T)$ convergence of $f_{\gamma,\epsilon}$ to $f_{\gamma}$ (c.f. lemma \ref{lem:boundF}) while the fifth equality follows from integration by parts. In a similar fashion, the Fourier-decay \eqref{eq:evenbetter} gives way more than enough regularity to justify the expression
$$
\frac{\rd}{\rd t} \left( \frac{1}{2} \int_{\T} |\dot{\tang}(x,t)|^2 \, \rd x \right) = - \langle \tang_{t}(t), \tang_{xx}(t) \rangle_{L^{2}(\T)}
$$
as well. We then simply recognize that $ \langle \tang_{t}, \lambda_{\tang} + \mu_{\tang} \tang \rangle_{L^{2}(\T)} = 0$ {\color{black} since $\tang$ has zero mean and constant speed}, then combine these calculations to find
$$
E^{\prime}(t) = - \langle \tang_{t}(t), ( \tang_{xx} + \lambda_{\tang} + \mu_{\tang}\tang)(t) \rangle_{L^{2}(\T)} - \langle \tang_{t}(t) , \F_{\gamma_{\tang(t)}}  \rangle_{L^{2}(\T)} = -\|\tang_{t}(t)\|^{2}_{L^{2}(\T)}
$$
as claimed. In particular, we may conclude that the energy identity
\begin{equation}\label{eq:enerid}
E(t) = E(0) - \int^{t}_{0} \|\tang_{t}(s)\|^{2}_{L^{2}(\T)} \, \rd s \qquad E(t) \leq E(0)
\end{equation}
holds for as long as a mild solution exists. In particular, if $K(u,v)$ satisfies the hypotheses of lemma \ref{lem:distbound} then we {\color{black} can apply \eqref{eq:log} to obtain the a-priori bounds for the bending energy and distortion{\color{black}
\begin{align*}
\frac{1}{2}\| \tang(t)\|^{2}_{\dot{H}^{1}(\T)} &\leq E(t) + 1 \leq E(0) + 1 \\
\delta_{\infty}(\gamma_{\tang(t)}) & \leq c^{-1}_{0}(p,h) \mathrm{exp} (  C_{0}(p,h)(E^{2}(0)+E^{3}(0))/2  )
\end{align*}
for $c_{0}(p,h),C_0(p,h)$ from lemma \ref{lem:distbound} fixed, postitive constants.}} Combining this observation with corollary \ref{cor:continuation} yields our main result ---

\begin{theorem}\label{thm:global}
Suppose $K(u,v)$ satisfies the hypotheses of lemmas \ref{lem:distbound}, \ref{lem:boundF}, that $\tang_0 \in H^{1}(\T)$ has mean zero, unit speed and that $\gamma_{\tang_0}$ defines a bi-Lipschitz embedding. Then
\begin{enumerate}[ \rm (a) ]
\item The initial value problem
$$
\tang_t = \tang_{xx} + \lambda_{\tang} + \F_{ \gamma_{\tang}} + \mu_{\tang} \tang \quad \text{with} \quad \tang(0) = \tang_{0}
$$
has a unique, globally existing classical solution $\tang \in C([0,\infty);H^{1}(\T))\cap C((0,\infty);H^{3}(\T)) \cap C^{1}((0,\infty);H^{1}(\T))$ with mean zero and unit speed.
\item For all $t \in [0,\infty)$ the map $\tang_0 \mapsto \tang(t)$ is Lipschitz near $\tang_0$ with respect to the $H^{1}(\T)$ topology: any two solutions $\tang(t),\sang(t)$ obey $\|\tang(t) - \sang(t)\|_{H^1(\T)} \leq C(t)\| \tang_0 - \sang_0\|_{H^{1}(\T)}$ whenever $\sang_0 \in H^{1}(\T)$ has mean zero, unit speed and $\| \tang_0 - \sang_0\|_{H^{1}(\T)}$ is sufficiently small.
\item The energy identity
$$
E(0) - E(t) = \int^{t}_{0} \|\tang_{t}(s)\|^{2}_{L^{2}(\T)} \, \rd s
$$
holds for all time.
\item The induced curves $\gamma_{ \tang(t)}$ define bi-Lipschitz embeddings, and moreover there exist finite constants $H_*,\delta_*$ so that the a-priori bounds
$$
\sup_{ t \in [0,\infty) } \, \| \tang(t) \|_{H^{1}(\T)} := H_{*} < +\infty \qquad \text{and} \qquad \sup_{ t \in [0,\infty) } \, \delta_{\infty}\big( \gamma_{\tang(t)} \big) := \delta_{*} < +\infty
$$
hold.
\item Each $\gamma_{\tang(t)}$ lies in the same ambient isotopy class as the initial curve, i.e. $\gamma_{ \tang(t)} \in \K_{\gamma_{\tang_0}}$ for all time.
\item If, in addition, the nonlocal forcing $f_{\gamma}$ obeys
$$
\|f^{(n)}_{\gamma}\|_{L^{2}(\T)} \leq C\big( \delta_{\infty}(\gamma) , \|\gamma\|_{H^{n+2}(\T)} \big)
$$
whenever $\gamma$ has unit speed, then there exist finite constants $c_{n}(\epsilon),C_{n}(\epsilon)$ so that
$$
\sup_{ t \in [\epsilon,\infty)} \, \|\tang(t)\|_{H^{n+2}(\T)} \leq c_{n}(\epsilon) \qquad \text{and} \qquad \sup_{ t \in [\epsilon,\infty)} \, \sup_{k \in \Z} \, |k|^{n+3} |\hat{\tang}_k(t)| \leq C_{n}(\epsilon)
$$
hold for $\epsilon > 0$ arbitrary.
\end{enumerate}
\end{theorem}
\begin{proof}
Parts {(a), (c)} and (d) simply summarize previous results. Part (e) follows immediately from lemma \ref{lem:ambclass} since the $\gamma_{\tang(t)}$ are continuous in time with respect to the $C^{1,\frac1{2}}(\T)$ topology. Part (f) follows from corollary \ref{cor:smooth} provided that the quantity
$$
A_{*} := \sup_{ t \in [0,\infty) } \, \| A^{-1}_{\tang(t)} \|_{ {\rm op} }
$$
remains bounded. Let $\vv(t) \in \R^3$ denote a unit length eigenvector corresponding to the minimal eigenvalue of $\lambda_{ {\rm min}}(A_{\tang(t)}),$ so that
$$
\lambda_{ {\rm min}}(A_{\tang(t)}) = \langle \vv(t) ,  A_{\tang(t)} \vv(t) \rangle = \fint_{\T} (1 - \langle \vv(t) , \tang(x,t) \rangle^2 ) \, \rd x.
$$
By the Borsuk-Ulam theorem, there exist antipodal points $x<y=x+\pi$ so that $\langle \gamma_{\tang(t)}(x),\vv(t) \rangle = \langle \gamma_{\tang(t)}(y) , \vv(t) \rangle,$ and so for any two such points
\begin{align*}
\gamma_{ \tang(t) }(y) - \gamma_{\tang(t)}(x) &= \gamma_{ \tang(t) }(y) - \langle\gamma_{\tang(t)}(y) , \vv(t) \rangle\vv(t) - \gamma_{\tang(t)}(x) + \langle\gamma_{\tang(t)}(x) , \vv(t) \rangle \vv(t) \\
&= \int^{y}_{x} \big[ \tang(z,t) - \langle \vv(t) , \tang(z,t) \rangle \vv(t) \big] \, \rd z \\
|\gamma_{ \tang(t) }(y) - \gamma_{\tang(t)}(x)| & \leq \int^{y}_{x} \big| \tang(z,t) - \langle \vv(t) , \tang(z,t) \rangle \vv(t) \big| \, \rd z = \int^{y}_{x} \left( 1 - \langle \vv(t) , \tang(z,t) \rangle^2 \right)^{\frac1{2}} \, \rd z \\
& \leq \pi \sqrt{ 2 \lambda_{ {\rm min}}(A_{\tang(t)}) }
\end{align*}
with the final inequality following by Cauchy-Schwarz. Thus
$$
\delta^{-1}_{\infty}\big(\gamma_{ \tang(t)} \big) \leq \frac{|\gamma_{ \tang(t) }(y) - \gamma_{\tang(t)}(x)|}{\dd(x-y)} = \frac{|\gamma_{ \tang(t) }(y) - \gamma_{\tang(t)}(x)|}{\pi} \leq \sqrt{ 2 \lambda_{ {\rm min}}(A_{\tang(t)}) },
$$
and so $\lambda_{ {\rm min}}(A_{\tang(t)}) \geq \delta^{-2}_{\infty}\big(\gamma_{ \tang(t)} \big)/2$ which gives the desired claim $A_{*} \leq 2 \delta^{2}_{*}$. Part (b) then follows from this uniform estimate and the proof of theorem \ref{thm:local}. For $\sang_0$ near $\tang_0$ the induced curve $\gamma_{\sang_0}$ induces a bi-Lipschitz embedding, and thus $\sang(t)$ exists globally with uniform bounds as in (d). For any $T>0,$ the {\color{black} contraction mapping argument preceeding} theorem \ref{thm:local} shows
$$
\|\tang - \sang\|_{C([0,T];H^{1}(\T))} \leq C\|\tang_0 - \sang_0\|_{ \dot{H}^{1}(\T)} + K T^{\frac1{5}}\|\tang - \sang\|_{C([0,T];H^{1}(\T))}
$$
for some universal constant $C\geq1$ and some constant $K>0$ that depends only on $(A_*,\delta_*,H_*)$ and the analogous quantities defined using $\sang(t)$. In particular, $K>0$ remains bounded uniformly in time. Now choose $T$ so that $T^{\frac1{5}}K = 1/2$ to find
$$
\|\tang - \sang\|_{C([0,T];H^{1}(\T))} \leq 2C\|\tang_0 - \sang_0\|_{ \dot{H}^{1}(\T)}.
$$
Applying the same argument with $\tang(T),\sang(T)$ in place of $\tang_0,\sang_0$ gives
$$
\|\tang - \sang\|_{C([T,2T];H^{1}(\T))} \leq 2C\|\tang(T) - \sang(T)\|_{ \dot{H}^{1}(\T)} \leq 4C^2\|\tang_0 - \sang_0\|_{ \dot{H}^{1}(\T)}.
$$
The desired claim
$$
\|\tang - \sang\|_{C([0,nT];H^{1}(\T))} \leq 2C\|\tang((n-1)T) - \sang((n-1)T)\|_{ \dot{H}^{1}(\T)} \leq (2C)^n\|\tang_0 - \sang_0\|_{ \dot{H}^{1}(\T)}
$$
then follows for $n \in \mathbb{N}$ arbitrary by induction.
\end{proof}
\noindent With theorem \ref{thm:global} in hand, well-known results then suffice to produce the usual, crude picture of the global behavior of solutions. Specifically, for any $\tang_0 \in H^{1}(\T)$ to which the statement of theorem \ref{thm:global} applies let $\tang(t;\tang_0)$ denote the corresponding global solution. We know the set of $H^{1}(\T)$ limit points
$$
\omega(\tang_0) := \left\{ \tang \in H^{1}(\T) : \exists t_{n} \nearrow \infty, \;\; \|\tang(t_n;\tang_0) - \tang\|_{H^{1}(\T)} \overset{n \to \infty}{\longrightarrow} 0\right\}
$$
generated by the initial datum $\tang_0$ is non-empty and consists entirely of critical points. Thus the elliptic system
$$
\tang_{xx} + G(\tang,\tang_x) = 0
$$
holds for $\tang \in \omega(\tang_0)$ any limit point. Moreover, any $\tang \in \omega(\tang_0)$ inherits the regularity of part (f) and the convergence of $\tang(t_n;\tang_0)$ to $\tang$ occurs in the $H^{k+2}(\T)$ topology, with $k$ denoting any integer $k \in \N$ so that the assumption in part (f) holds. All elements  $\tang \in \omega(\tang_0)$ induce bi-Lipschitz embeddings $\gammat$ that lie in the same ambient isotopy class as $\gamma_{\tang_0}$ as well.

Deducing further properties of $\omega(\tang_0)$ requires a more in-depth investigation of equilibria themselves. We provide a brief analysis of such critical points in the next subsection, but first pause to place theorem \ref{thm:global} in a more concrete and applicable setting. If we take
$$
K(u,v) = \frac{v^{q}}{u^{q}} \quad (2q-\text{Distortion})
$$
then $K$ satisfies the $h/p$ homogeneity hypothesis with $h(\alpha) = \alpha^{q}$ and $p = 0,$ and so $0$-degeneracy applies to this family of examples. Moreover, if $q \geq 1$ then lemma \ref{lem:distbound} applies as well, as does the hypothesis in part (f) for all $n \in \N$ due to the remarks {preceding} corollary \ref{cor:smooth}. Thus theorem \ref{thm:global} applies in its entirety for this family, and in particular all critical points it generates are smooth. If we take
$$
K(u,v) = \left( \frac1{u^j} - \frac1{v^j} \right)^{q} \quad (\text{O'Hara})
$$
instead then $h/p$ homogeneity applies with $h(\alpha) = (\alpha^{j}-1)^{q}$ and $p = jq$ for this family of examples. If $jq \geq 1$ and $q$ is sufficiently large, say $j = 1/q$ and $q \geq 5$, then $h(\alpha)$ is $\geq \hspace{-.4pc}3$-degenerate and so theorem \ref{thm:global} (a-e) applies. This regime represents the interesting instances of this family, as the $j\to 0,q \to \infty$ limit lies at the heart of the motivation for the O'Hara family --- the convergence of $E_{K}$ to $\log \delta_{\infty}$ occurs in this limit. Part (f) also applies with $n=1,$ so in particular all critical points have at least $H^{3}(\T)$ regularity. In general, critical points will have a higher degree of smoothness if $h(\alpha)$ itself does, say for $q\geq 5$ an integer. {\color{black} The regularity theory for critical points of the O'Hara family and the M\"{o}bius energy (i.e. in the absence of the bending energy) is both established and more difficult \cite{BR13,BRS}; We mention these particular examples to reveal that theorem \ref{thm:global} provides a relatively useful and general result, in the  sense that it proves broad enough to cover many of the kernels $K(u,v)$ that have generated prior interest in the literature. This is the main contribution of the theorem --- i.e. finding a relatively general and applicable hypothesis under which theorem \ref{thm:global} holds.}

\subsection{Equilibrium Analysis}
All of the corresponding energies for these examples, namely these particular instances of the $(2q-\text{Distortion})$ family and the O'Hara family, have the standard unit circle $\gamma_{ {\rm circ} }$ as their unique minimizers. Of course this is not enough, in general, to even know that $\gamma_{ {\rm circ} }$ has a non-trivial basin of attraction under the dynamics of
\begin{equation}\label{eq:dynagain}
\tang_{t} = \tang_{xx}  + \mathbf{F}_{ \gammat } + \lambda_{\tang} + \mu_{\tang} \tang := \tang_{xx} + G(\tang,\tang_x).
\end{equation}
Establishing a result along these lines requires a further analysis, and we now turn to this task. For simplicity and concreteness we will largely abandon our attempt at generality from this point forward. Instead we shall largely focus on the $2$-Distortion case $K(u,v) = v/u,$ although the techniques employed are straightforward in principle and readily generalize to other cases.

We begin by extracting the leading-order behavior of \eqref{eq:dynagain} near an equilibrium. We shall slightly abuse notation throughout this process by suppressing time dependence and by implicitly interpreting integrals in the principal value sense where necessary. Let $\tangc$ denote such an equilibrium and $\gamma_{\tangc}$ the bi-Lipschitz embedding that it induces. Decompose a solution $\tang(t)$ near $\tangc$ as $\tang = \tangc + \psi$ and let $\gammat = \gamma_{\tangc} + \Psi$ denote the corresponding decomposition of the induced curve. Define $\Delta \gamma_{\tangc}(x,y) := \gamma_{\tangc}(x) - \gamma_{\tangc}(y),$ and define $\Delta \Psi(x,y)$ and $\Delta \gammat(x,y)$ similarly. For $0 \leq \xi \leq 1$ let us write $\alpha_{\xi}(x,y)$ for the ratio
$$\alpha_{\xi}(x,y) := \frac{ \dd^{2}(x-y)}{ (1-\xi) | \Delta \gamma_{\tangc}(x,y)|^2 + \xi | \Delta \gammat(x,y)|^2 }$$ and then reserve $\ac(x,y) := \dd^{2}(x-y)/|\Delta \gamma_{\tangc}(x,y)|^2$ for the $\xi=0$ case and $\alpha_{\gammat}(x,y)$ for the $\xi=1$ case. Consider first the nonlocal forcing $\mathbf{F}_{\gammat}$ that arises as the mean zero primitive of the principal value integral $f_{\gammat}$ given by
\begin{align*}
f_{\gammat}(x) &= \int_{\T} K_{u}\big(|\gammat(x) - \gammat(y)|^2, \dd^2(x-y) \big)(\gammat(x) - \gammat(y) ) \, \rd y \\
&= g(1)L_{p}[\gammat](x) + \int_{\T} \left( g\big( \alpha_{\gammat}(x,y) \big)  - g(1)  \right) \frac{ \gammat(x) - \gammat(y) }{\dd^{2(p+1)}(x-y) } \, \rd y,
\end{align*}
with the final equality holding in the $h/p$ homogeneous case.  A straightforward but tedious calculation allows us to decompose $f_{\gammat}$ as
$$
f_{\gammat} = f_{ \gamma_{\tangc} } + \ell_{\tangc}[\Psi] + r_{\tangc}(\Psi),
$$
where we define a linear operator $\ell_{\tangc}[\Psi]$ according to
\begin{align*}
\ell_{\tangc}[\Psi](x) &:= \int_{\T} \frac{ g\big( \ac(x,y) \big) \Delta \Psi(x,y)}{\dd^{2(p+1)}(x-y)} \, \rd y - 2 \int_{\T} g^{\prime}\big( \ac(x,y) \big)\ac^2(x,y)
\frac{ \Delta \gamma_{\tangc}(x,y) \otimes \Delta \gamma_{\tangc}(x,y) }{ \dd^{2(p+2)}(x-y)  }\Delta \Psi(x,y) \, \rd y
\end{align*}
and a non-linear, quadratic remainder $r_{\tangc}(\Psi)$ according to
\begin{align}
\label{eq:bigremainder}
r_{\tangc}(\Psi)(x) &:=-2 \int_{\T} g^{\prime}\big( \ac(x,y) \big)\ac^2(x,y)
\frac{ \Delta \Psi(x,y) \otimes \Delta \Psi(x,y) }{ \dd^{2(p+2)}(x-y)  }\Delta \gamma_{\tangc}(x,y) \, \rd y \\
& - \int_{\T} |\Delta \Psi(x,y)|^2 g^{\prime}\big( \ac(x,y) \big) \ac^2(x,y) \frac{ \Delta \gammat(x,y) }{\dd^{2(p+2)}(x-y)} \, \rd y \nonumber \\
&+ \int_{\T} \left(\int^{1}_{0} (1-\xi) \bar{g}\big( \alpha_{\xi}(x,y) \big) \, \rd \xi \right)\left( \frac{2\langle \Delta \gamma_{\tangc}(x,y) , \Delta \Psi(x,y)\rangle + |\Delta \Psi(x,y)|^2}{\dd^{p+3}(x-y)}\right)^2\Delta \gammat(x,y) \, \rd y \nonumber
\end{align}
provided that we define $\bar{g}(\alpha) := 2 g^{\prime}(\alpha)\alpha^3 + g^{\pprime}(\alpha)\alpha^{4}$ in the last line. This decomposition of $f_{\gammat}$ induces a corresponding decomposition
\begin{align}\label{eq:Fdecomp}
\mathbf{F}_{\gammat} = \mathbf{F}_{\gamma_{\tangc}} + L_{\tangc}[\Psi] + \mathbf{R}_{\tangc}(\Psi)
\end{align}
by taking the mean zero primitive, with the definitions of each term {occurring} in the obvious way. We may then conclude by using this decomposition of $\mathbf{F}_{\gammat}$ to extract an analogous decomposition of the Lagrange multipliers. Let $\lambda_{\tang} = \lambda_{\tangc} + \ell_{\lambda,\tangc}[\psi] + r_{\lambda,\tangc}(\psi)$ and $\mu_{\tang} = \mu_{\tangc} + \ell_{\mu,\tangc}[\psi] + r_{\mu,\tangc}(\psi)$ denote a decomposition of the Lagrange multipliers into their constant, linear and quadratic components. The relations
$$
\lambda_{\tang} = -\fint_{\T} \mu_{\tang}\tang\, \rd x, \quad \mu_{\tang}|\tang|^2 = |\tang_x|^2 - \langle \mathbf{F}_{\gammat} + \lambda_{\tang} , \tang \rangle \quad \text{and} \quad A^{-1}_{\tang} = A^{-1}_{\tangc} + A^{-1}_{\tang}(A_{\tangc} - A_{\tang})A^{-1}_{\tangc}
$$
then combine with the decomposition \eqref{eq:Fdecomp} in a straightforward way to show
\begin{align*}
A_{\tangc}\ell_{\lambda,\tangc}[\psi] &= \fint_{\T} \left( \langle\mathbf{F}_{\gamma_{\tangc}} + \lambda_{\tangc}, \psi\rangle + \langle L_{\tangc}[\Psi] , \tangc\rangle \right) \tangc \, \rd x + 2 \fint_{\T}(\mu_{\tangc}\langle\tangc,\psi\rangle- \langle \tangc^{\prime},\psi_x\rangle) \tangc \, \rd x - \fint_{\T}\mu_{\tangc}\psi \, \rd x \\
\ell_{\mu,\tangc}[\psi] &= 2( \langle \tangc^{\prime} , \psi_x \rangle - \mu_{\tangc}\langle \tangc,\psi\rangle) - \langle\mathbf{F}_{\gamma_{\tangc}} + \lambda_{\tangc}, \psi\rangle - \langle L_{\tangc}[\Psi] + \ell_{\lambda,\tangc}[\psi] , \tangc\rangle.
\end{align*}
The precise structure of the higher-order terms $r_{\lambda,\tangc}(\psi),r_{\mu,\tangc}(\psi)$ is unenlightening, but it is easy to see that $r_{\mu,\tangc}(\psi)$ and $r_{\lambda,\tangc}(\psi)$ obey estimates similar to lemma \ref{lem:boundmult}. Specifically, $r_{\mu,\tangc}(\psi)$ decomposes into singular and regular components
\begin{align*}
r_{\mu,\tangc}(\psi) &= \frac{ |\psi_x|^2 - \bar{r}_{\mu,\tangc}(\psi)}{|\tang|^2} \qquad \text{and} \\
 \bar{r}_{\mu,\tangc}(\psi) &:= \langle \mathbf{R}_{\tangc}(\psi) + r_{\lambda,\tangc}(\psi),\tang\rangle + \langle L_{\tangc}[\psi] + \ell_{\lambda,\tangc}[\psi],\psi\rangle + \mu_{\tangc}|\psi|^2 + \ell_{\mu,\tangc}[\psi]\left( 2\langle\tangc,\psi\rangle + |\psi|^2 \right)
\end{align*}
respectively, while $r_{\lambda,\tangc}(\psi)$ is constant in space and obeys
\begin{align*}
|r_{\lambda,\tangc}(\psi)| &\leq C\left[ \|A^{-1}_{\tang}\|_{ {\rm op} }\|\psi\|_{H^{1}(\T)} \left(  | \ell_{\lambda,\tangc}[\psi]| + \|L_{\tangc}[\Psi]\|_{L^{1}(\T)} + \left( \|\mathbf{F}_{\gammat}\|_{L^{1}(\T)} + \|\tangc\|_{H^{1}(\T)} + 1 \right)\|\psi\|_{H^{1}(\T)} \right)\right. \\
&  \left. +  \|A^{-1}_{\tang}\|_{ {\rm op} }\| \mathbf{R}_{\tangc}(\psi)\|_{L^{1}(\T)} \right]\left( \min_{\T} |\tang|^2 \right)^{-1}
\end{align*}
for $C$ some modest universal constant. Arguments similar to lemma \ref{lem:boundF} show that $\mathbf{R}_{\tangc}(\psi)$ depends quadratically on $\psi$ in the $H^{1}(\T)$ sense. {\color{black} Specifically, the integrands defining $r_{\tangc}(\Psi)$ (c.f. \eqref{eq:bigremainder}) and $f_{\gamma,\epsilon}$ (c.f. lemma \ref{lem:boundF}) exhibit the same behavior along the diagonal $x=y,$ and so the proof of lemma \ref{lem:boundF} applies.} We may therefore conclude that
\begin{equation}\label{eq:linearized}
\psi_{t} = \psi_{xx} + L_{ {\tangc}}[\Psi] + \ell_{\lambda,\tangc}[\psi] + \ell_{\mu,\tangc}[\psi]\tangc + \mu_{\tangc}\psi + O\left( \|\psi \|^{2}_{H^{1}(\T)} \right)
\end{equation}
formally governs the dynamics to leading-order near an equilibrium. We may then use \eqref{eq:linearized} to at least gain a modicum of dynamical insight. We consider two cases, namely the standard circle $\tangc = \gamma_{ {\rm circ}}^{\prime}$ under the $2$-Distortion $K(u,v) = v/u$ and the double-covered circle $\tangc = \gamma_{ {\rm dc}}^{\prime}$ with $K=0$, i.e. a pure bending energy flow.

\subsubsection{The Standard Circle}
We begin by analyzing the standard circle near equilibrium for $K$ non-zero. If we set $\tangc = \gamma_{ {\rm circ}}^{\prime}$ and allow $K(u,v)$ as any bi-variate kernel then we have
\begin{align}\label{eq:lamkdef}
\mathbf{F}_{\gamma_{\tangc}} &= -\lambda_{1}(K_u) \tangc,  &&\lambda_{k}(K_u) := \int_{\T} K_{u}\left( 2(1-\cos z), z^2\right)(1-\cos kz) \, \rd z,  &&A_{\tangc} = \mathrm{diag}\left(1/2,1/2,1\right), \\
\lambda_{\tangc} &= 0 , &&|\tangc^{\prime}| = 1, && \mu_{\tangc} = 1 + \lambda_1(K_u), \nonumber
\end{align}
which readily suffice for a direct verification that $\tangc = \gamma_{ {\rm circ}}^{\prime}$ always defines an equilibrium. A special case of the general theory from \cite{von2016localization} {\color{black} (c.f. equation (14) in \cite{von2016localization})} shows that we may decompose any  mean zero $\psi \in L^{2}(\T;\R^3)$ as
\begin{equation}\label{eq:vsh}
\psi = \pt_0 \tangc + \pn_0 \normc + \sum_{ k \in \Z_0 } \left( \pn_k \normc + ik \pt_k \tangc\right) \re^{ikx} + \sum_{k \in \Z_0} \pb_k \binc  \re^{ikx} \qquad \Z_0 := \Z \setminus \{0\}, \; \pn,\pb,\pt \in \mathbb{C},
\end{equation}
with $\normc := \tangc^{\prime}$ and $\binc := \tangc \times \normc$ representing a pointwise orthonormal basis. {\color{black} Up to a change of basis on the coefficients, the decomposition \eqref{eq:vsh} is the Fourier transform.} The restrictions $\pn_1 = \pt_1, \pn_{-1} = \pt_{-1}$ hold since $\psi$ has mean zero. The decomposition \eqref{eq:vsh} induces an analogous decomposition
\begin{align}\label{eq:vsh2}
\Psi &= \pn_0 \tangc - \pt_0 \normc + \frac{\pt_{1} \normc + i\pn_{1} \tangc}{2i}\re^{ix}  +  \frac{\pt_{-1} \normc -i \pn_{-1} \tangc }{-2i}\re^{-ix} + \sum_{k \in \Z_0} (ik)^{-1}\pb_k \binc  \re^{ikx} \nonumber \\
&+ \sum_{ k \in \Z_{0,1} } (ik)^{-1}\left( \frac{k^2(\pn_k - \pt_k)}{k^2-1}\normc + ik \frac{k^2\pt_k - \pn_k}{k^2-1} \tangc \right) \re^{ikx} \qquad \Z_{0,1} := \Z \setminus \{0,1,-1\}
\end{align}
of its mean zero primitive, and following \cite{von2016localization,von2016anisotropic} we shall use the shorthand
\begin{align*}
&
\begin{bmatrix}
\pn_0 \\ \pt_0 \\ \pb_0
\end{bmatrix}
\overset{ \Psi }{\longrightarrow}
\begin{bmatrix}
0 & -1 & 0\\
1 & 0 & 0 \\
0 & 0 & 0
\end{bmatrix}
\begin{bmatrix}
\pn_0 \\ \pt_0 \\ \pb_0
\end{bmatrix}
\qquad \qquad \qquad
\begin{bmatrix}
\pn_{\pm 1} \\ \pt_{\pm 1} \\ \pb_{\pm 1}
\end{bmatrix}
\overset{ \Psi }{\longrightarrow}
\frac1{\pm i}
\begin{bmatrix}
0 & \frac1{2} & 0\\
\frac1{2} & 0 & 0 \\
0 & 0 & 1
\end{bmatrix}
\begin{bmatrix}
\pn_{\pm 1} \\ \pt_{\pm 1} \\ \pb_{\pm 1}
\end{bmatrix}\\
&
\begin{bmatrix}
\pn_{k} \\ \pt_{k} \\ \pb_{k}
\end{bmatrix}
\overset{ \Psi }{\longrightarrow}
\frac1{(ik)(k^2-1)}
\begin{bmatrix}
k^2 & -k^2 & 0\\
-1 & \;\,\,k^2 & 0 \\
0 & 0 & k^2-1
\end{bmatrix}
\begin{bmatrix}
\pn_{k} \\ \pt_{k} \\ \pb_{k}
\end{bmatrix}
\;\;\; (k^2 \neq 0, 1),
\qquad \pp_k \overset{ \Psi }{\longrightarrow} A_{k,\Psi}\pp_k
\end{align*}
to compactly represent the operation $\psi \mapsto \Psi$ in terms of its actions $\{ A_{k,\Psi} \}_{k \in \Z}$ on coefficients. Of course, the fact that the differential operators $\partial_x,\partial_{xx}$ have convenient representations
\begin{align*}
&\pp_k \overset{\partial_x}{\longrightarrow} A_{k,\partial_x} \pp_k \qquad \quad
A_{0,\partial_x} = \begin{bmatrix}
0 & 1 & 0\\
-1 & 0 & 0 \\
0 & 0 & 0
\end{bmatrix}
\qquad\;\; A_{k,\partial_x} = \begin{bmatrix}
ik & ik & 0\\
i/k & ik & 0 \\
0 & 0 & ik
\end{bmatrix} \\
&\pp_k \overset{\partial_{xx}}{\longrightarrow} A_{k,\partial_{xx}} \pp_k \qquad
A_{0,\partial_{xx}} = -\begin{bmatrix}
 1 & 0 & 0\\
0 & 1 & 0 \\
0 & 0 & 0
\end{bmatrix}
\qquad A_{k,\partial_{xx}} = -\begin{bmatrix}
1+k^2 & 2k^2 & 0\\
2 & 1+k^2 & 0 \\
0 & 0 & k^2
\end{bmatrix}
\end{align*}
follows from standard Fourier analysis as well.

The major benefit of the representation \eqref{eq:vsh} comes from the analyses in \cite{von2016localization,von2016anisotropic}, which show that the nonlocal operators in \eqref{eq:linearized} also have a convenient description of the form $\pp_k \mapsto  A_k \pp_k$ for some sequence $\{ A_k \}_{k \in \Z}$ of easily-computable matrices. Applying the analysis from \cite{von2016localization,von2016anisotropic} is aided by first defining the quantities
\begin{align}\label{eq:sigkdef}
\lambda^{\pm,\mp}_{k}(M) &:= \int_{\T} M\left( 2(1-\cos z) , z^2 \right)(1\pm\cos z)(1 \mp \cos kz) \, \rd z \nonumber \\
\sigma_{k}(M) &:= \int_{\T} M\left( 2(1-\cos z) , z^2 \right)\sin z \sin kz \, \rd z
\end{align}
for $M(u,v)$ an arbitrary bi-variate kernel. Appealing to these definitions, quoting {\color{black} theorem 3.4 of \cite{von2016localization}} and simplifying then shows that the operator $\Psi \mapsto \ell_{\tangc}[\Psi]$ has the representation
\begin{equation}
A_{k,\ell_{\tangc}} =
\begin{bmatrix}\label{eq:elldefmtx}
\left( \frac{\lambda^{+,-}_{k} + \lambda^{-,+}_{k}}{2} \right)(K_u) + \lambda^{-,+}_{k}(uK_{uu}) & k \sigma_{k}(K_u + uK_{uu}) & 0\\
k^{-1} \sigma_{k}(K_u + uK_{uu}) & \left( \frac{\lambda^{+,-}_{k} + \lambda^{-,+}_{k}}{2} \right)(K_u)+ \lambda^{+,-}_{k}(uK_{uu})& 0 \\
0 & 0 & ( \frac{\lambda^{+,-}_{k} + \lambda^{-,-}_{k}}{2})(K_u)
\end{bmatrix}
\end{equation}
for $k\neq 0,$ while for $k=0$ the action of $A_{0,\ell_{\tangc}}$ is obtained by setting $k=0$ in the diagonal entries of $A_{k,\ell_{\tangc}}$ and setting all off-diagonals to zero. {\color{black} An alternative approach to deriving \eqref{eq:elldefmtx} involves substituting the expansion \eqref{eq:vsh} and appealing to the convolution theorem.} We may then represent the composite mapping $\psi \mapsto L_{\tangc}[\Psi] \Leftrightarrow \psi \mapsto \Psi \mapsto \ell_{\tangc}[\Psi] \mapsto L_{\tangc}[\Psi]$ via $\pp_k \mapsto A_{k,\Psi}A_{k,\ell_{\tangc}}A_{k,\Psi}\pp_k$ and obtain the corresponding matrix representation $A_{k,L_{\tangc}} := A_{k,\Psi}A_{k,\ell_{\tangc}}A_{k,\Psi}$ by simple multiplication.

We finally obtain an analogous fine-grained characterization of the full operator \eqref{eq:linearized} by combining these representations in a straightforward manner. If we use \eqref{eq:vsh} to decompose $L_{\tangc}[\Psi]$ into components
$$
L_{\tangc}[\Psi] = \pn_{L_{\tangc}}\normc + \pt_{L_{\tangc}}\tangc + \pb_{L_{\tangc}}\binc,
$$
and then combine the relation $\langle \ell_{\lambda,\tangc}[\psi] , \binc \rangle = -(1 + \lambda_{1}(K_u))\pb_0\binc$ with the fact that
\begin{align*}
\langle \ell_{\lambda,\tangc}[\psi] , \normc \rangle &= (1 - \lambda_{1}(K_u))\left( \pn_{1} \re^{ix} +  \pn_{-1} \re^{-ix} \right) + \left( \pt_1 \re^{ix} + \pt_{-1} \re^{-ix} \right) \\
& - (\re^{ix}\pt_{L_{\tangc},1} + \re^{-ix}\pt_{L_{\tangc},-1} )
\end{align*}
then we easily see that \eqref{eq:linearized} reduces to
\begin{align*}
\psi_{t} &= \psi_{xx} + (1 + \lambda_{1}(K_u))\psi +  \pn_{L_{\tangc}}\normc + \pb_{L_{\tangc}}\binc + ( 2\langle \normc , \psi_x \rangle - 2\mu_{\tangc}\langle\tangc,\psi\rangle + \lambda_{1}(K_u)\pt )\tangc \\
&+\langle \ell_{\lambda,\tangc}[\psi],\normc\rangle\normc - (1+\lambda_{1}(K_u))\pb_0 \binc + O\left( \|\psi\|^{2}_{H^{1}(\T)} \right)
\end{align*}
at leading order. Recalling the  definition $A_{k,L_{\tangc}} := A_{k,\Psi}A_{k,\ell_{\tangc}}A_{k,\Psi}$ then allows us to write the effect of \eqref{eq:linearized} as the operation $\pp_k \to A_{k} \pp_k$ with
\begin{align}
A_{0} &= 0 \qquad
A_{\pm 1} = -\begin{bmatrix} 0 & 1 & 0\\
                                               0 &  1 & 0 \\
			              0 & 0 & 0 \end{bmatrix}
 \nonumber \\ \label{eq:thatone}
 A_{k} &=
\begin{bmatrix}
\lambda_{1}(K_u) - k^2 + a_k & -2k^2 + b_k & 0\\
0 & - k^2 & 0 \\
0 & 0 & (1 + \lambda_{1}(K_u)) - (k^4 + \lambda_k(K_u))/k^2
\end{bmatrix}\\ \label{eq:thisone}
a_k &:= -\left( \frac{ (k-1)^2\lambda_{k+1} + (k+1)^2 \lambda_{k-1} }{2(k^2-1)^2}\right)(K_u + uK_{uu}) + \frac{( \lambda_{k} - \lambda_{1})(uK_{uu})}{k^2-1}, \nonumber \\
\frac{b_k}{k} &:= -\left( \frac{ (k-1)^2\lambda_{k+1} - (k+1)^2 \lambda_{k-1} }{2(k^2-1)^2}\right)(K_u + uK_{uu}) \qquad \text{if} \qquad k>0, \nonumber \\
a_{-k} &= a_{k} \quad b_{-k} = b_{k} \qquad \text{if} \qquad k<0,
\end{align}
furnishing the appropriate sequence $\{ A_{k} \}_{k \in \Z}$ of matrices. The entries of (\ref{eq:thatone},\ref{eq:thisone}) follow from \eqref{eq:elldefmtx}, the definition (\ref{eq:lamkdef}) of $\lambda_{k}(M)$ for a bi-variate kernel $M(u,v)$ and a few basic trigonometric identities.

The convenient representation (\ref{eq:thatone},\ref{eq:thisone}) allows us to make a few final observations. When acting on the space of mean zero functions, the operator $L[\psi]$ represented as $\pp_k \mapsto A_{k}\pp_k$ necessarily has at least a four dimensional kernel, with
$$
\mathrm{Span}\left\{ \re^{ix}\binc, \re^{-ix}\binc , \normc ,\tangc  \right\} \subset \mathrm{ker}(L)
$$
providing a basis of eigenfunctions. The first three components of this basis obviously arise due to the family of equilibria $\{  \Theta \tangc \}_{\Theta \in SO(3)}$ obtained by appealing to rotation invariance. The last element arises due to the equilibria $\{ r \tangc \}_{r \in \R}$ obtained by appealing to scale invariance. Moreover, in the particular case that $K(u,v) = v/u$ is the kernel given by the $2$-Distortion then the remainder of the point spectrum $\sigma_{ {\rm pt}}(L)$ lies strictly to the left of the imaginary axis. Indeed, it follows easily from the triangular form of $A_k$ that
$$
\sigma_{{\rm pt} }(L)\setminus \{0\} = \left\{ - k^2 \right\}_{ k \in \mathbb{N}_{0} } \cup \left\{ \lambda_{1}(K_u) - k^2 + a_k , (1 + \lambda_{1}(K_u)) - (k^4 + \lambda_k(K_u))/k^2 \right\}_{ k \geq 2 }
$$
provided the non-trivial diagonal entries in $A_k$ are negative and bounded away from zero, and this latter fact follows relatively easily from the definitions
\begin{align*}
&\lambda_{k}(K_u) =  -\frac1{4}\int_{\T} \frac{ z^2(1-\cos kz)}{(1-\cos z)^2} \, \rd z, \qquad \lambda_{k}(uK_{uu}) = \frac1{2}\int_{\T} \frac{ z^2(1-\cos kz)}{(1-\cos z)^2}\, \rd z = 2 \lambda_{k}(K_u + uK_{uu})
\end{align*}
together with the Fourier integral identity
$$
\lambda_{k}(uK_{uu}) = \int_{\T} \frac{ 1-\cos kz}{1-\cos z} \, \rd z + \int_{\T} \frac{z^2/2-(1-\cos z)}{(1-\cos z)^2}(1-\cos kz) \, \rd z = 2\pi(|k| + \varepsilon_{k}), \quad |\varepsilon_{k}| < 1
$$
and a few simple estimates for the remainders. Thus (\ref{eq:thatone},\ref{eq:thisone}) reveals that
$$
\mathrm{Span}\left\{ \re^{ix}\binc, \re^{-ix}\binc , \normc ,\tangc  \right\} = \mathrm{ker}(L)
$$
for the $2$-Distortion, and so well-known {\color{black} approaches from dynamical systems (e.g. \cite{kapitula2013spectral} theorem 4.3.5) suffice to guarantee orbital stability of the family $\gamma_{ {\rm circ}}$ of standard circles. In other words, if $\tang_0 \in H^{1}(\T)$ has mean zero and we can find a $r_0 \in \mathbb{R}, \Theta_0 \in SO(3)$ so that $\| \tang_0 - r_{0}\Theta_{0} \tangc \|_{H^{1}(\T)}$ is sufficiently small, then there exist constants $C,\nu>0$ and unique selections $r(\tang_0) \in \mathbb{R}, \Theta(\tang_0) \in SO(3)$ so that exponential convergence to equilibrium
$$
\|r(\tang_0)\Theta(\tang_0) \tangc - \tang(t;\tang_0) \|_{H^{1}(\T)} \leq C \re^{-\nu t}
$$
holds. The key analytical points are as follows. First, the kernel of $L$ should arise from differentiable symmetries (here they arise from scaling and rotation invariance). Second, the semi-group $e^{Lt}$ should obey a bound of the form $\|\re^{Lt}\Pi f\| \leq C\mathrm{e}^{-\nu t}\|f\|$ with $\Pi f$ denoting the projection complementary to  the kernel of $L$; as the semi-group $\re^{Lt}$ corresponds to the family of matrix exponentials $\mathrm{e}^{A_k t}$, the block-diagonal decomposition \eqref{eq:thisone} suffices for proving the desired bound by reducing the problem to exponentials of ordinary, $2 \times 2$ matrices (see also section 4 of \cite{von2016localization}). Finally, the deviation of the full dynamics from the linearized system $\psi_{t} = L[\psi]$ should introduce a quadratic error (c.f. \eqref{eq:linearized}).} A similar conclusion applies for the entire $2q$-Distortion family. Moreover, the characterization (\ref{eq:thatone},\ref{eq:thisone}) remains valid for kernels $K(u,v)$ that do not necessarily admit standard circles as either local or global minimizers, and so this characterization of the dynamics \eqref{eq:dynagain} is potentially important for other applications.

\subsubsection{The Double-Covered Circle}\label{sssec:dc}
For the sake of comparison and contrast, we conclude by analyzing the double-covered circle under a pure bending energy flow. If we let $\tangc := \gamma^{\prime}_{ {\rm dc}} $ then we have that
\begin{align*}
\normc &:= \frac1{2}\tangc^{\prime},  && \tangc^{\pprime} = -4 \tangc &&A_{\tangc} = \mathrm{diag}\left(1/2,1/2,1\right), \\
\lambda_{\tangc} &= 0 , &&|\tangc^{\prime}| = 2, && \mu_{\tangc} = 4,
\end{align*}
which readily suffice for a verification that $\tangc$ defines a critical point of the bending energy. As with the unknot we shall decompose $\psi$ as
$$
\psi(x) = \pt(x) \tangc(x) + \pn(x) \normc(x) + \pb(x) \binc \qquad \qquad \pn,\pb,\pt \in \mathbb{C},
$$
with $\binc := \tangc \times \normc$ constant in space, but instead of \eqref{eq:vsh} the usual Fourier decomposition
$$
\pt = \sum_{k \in \Z} \pt_k \re^{ikx} \qquad \pn = \sum_{k \in \Z} \pn_k \re^{ikx} \qquad \pb = \sum_{k \in \Z_0} \pb_k \re^{ikx}
$$
of each component proves more convenient in this case. The relations $\pn_2 = -i \pt_2, \pn_{-2} = i \pt_{-2}$ now encode the fact that $\psi$ has mean zero. Elementary calculations then show
\begin{align*}
\psi_x &= (\pn_x + 2\pt )\normc + (\pt_x - 2 \pn)\tangc + \pb_x\binc \\
\psi_{xx} &= (\pn_{xx} - 4\pn + 4\pt_x )\normc + (\pt_{xx} - 4\pt - 4\pn_x)\tangc + \pb_{xx}\binc,
\end{align*}
which combine with the definitions of $\ell_{\lambda,\tangc},\ell_{\mu,\tangc}$ and the computation of $A_{\tangc}$ to give
\begin{align*}
\langle \ell_{\lambda,\tangc}[\psi] , \normc \rangle & = 4( \re^{2ix}\pn_{2} + \re^{-2ix} \pn_{-2}) + 4i (\pt_{-2} \re^{-2ix} - \pt_{2} \re^{2ix}), \qquad \langle \ell_{\lambda,\tangc}[\psi] , \binc \rangle = -4\pb_0=0 \\
\ell_{\mu,\tangc}[\psi] &= 4\pn_{x} - \langle \ell_{\lambda,\tangc}[\psi],\tangc\rangle.
\end{align*}
Finally, let us further decompose $\psi$
\begin{align*}
&\psi(x) = ( \pn(x;2) + \pn(x;\Z_2) )\normc(x) + ( \pt(x;2) + \pt(x;\Z_2) )\tangc(x) + \pb(x)\binc \\
&\pn(x;2) := \re^{2ix}\pn_{2} + \re^{-2ix} \pn_{-2} \quad \qquad \pt(x;2) := \re^{2ix}\pt_{2} + \re^{-2ix} \pt_{-2}\\
&\pn(x;\Z_2) := \sum_{k \in \Z_{2}} \pn_k \re^{ikx}  \quad \text{~~~~~~~~~~~~~~~~~~} \pt(x;\Z_2) := \sum_{k \in \Z_{2}} \pt_k \re^{ikx} \qquad \Z_{2} := \Z \setminus \{2,-2\}
\end{align*}
into its second mode components and their complements. The mean zero restriction $\pn_2 = -i \pt_2, \pn_{-2} = i \pt_{-2}$ then allows us to find that the full operator
\begin{align*}
\psi \mapsto \psi_{xx} + \ell_{\lambda,\tangc}[\psi] + \ell_{\mu,\tangc}[\psi]\tangc + \mu_{\tangc}\psi
\end{align*}
from \eqref{eq:linearized} has the more convenient representation
\begin{align}\label{eq:dcfinal}
\pn(\cdot;2)\normc + \pt(\cdot;2)\tangc &\mapsto \pn_{xx}(\cdot;2)\normc +  \pt_{xx}(\cdot;2)\tangc, \nonumber \\
\pn(\cdot;\Z_2)\normc + \pt(\cdot;\Z_2)\tangc &\mapsto (\pn_{xx}(\cdot;\Z_2) + 4 \pt_{x}(\cdot;\Z_2))\normc +  \pt_{xx}(\cdot;\Z_2)\tangc, \nonumber \\
\pb \binc &\mapsto (\pb_{xx} + 4 \pb)\binc
\end{align}
with respect to this decomposition.

{\color{black} That the double-covered circle is critical for the bending energy \cite{LS84}, and moreover, that the standard circle is the only stable critical point of the bending energy \cite{LS85} are known facts.} In addition, the representation \eqref{eq:dcfinal} confirms that the double-covered circle forms an unstable saddle-point of the bending energy flow. Nevertheless, it still has a quite large stable manifold. Essentially, any mean zero $\psi$ of the form
$$
\psi_{ {\rm stab} } = \left( \sum_{k \in \Z_0 } \pn_k \re^{ikx} \right)\normc + \left( \sum_{k \in \Z_0} \pt_k \re^{ikx} \right)\tangc + \left( \sum_{k \in \Z, |k| \geq 3} \pb_k \re^{ikx} \right) \binc
$$
induces an initial condition $\tang_0 = \tangc + \varepsilon \psi_{ {\rm stab} }$ that lies on the stable manifold to leading order in $\varepsilon$ and therefore, to leading order at least, lies in the basin of attraction of the double covered circle.  The standard trefoil that lies on the torus $(r - 1/2)^2 + z^2 = \veps^2$ takes the form
$$
\tang = \tangc + \veps \psi \qquad \psi = \cos 3x \tangc + \frac{3}{2} \sin 3x \, \normc - 3 \cos 3x \, \binc,
$$
and from \eqref{eq:dcfinal} this induces an initial condition $\tang_0 := \tangc + \veps \psi$ that, to leading order, lies in the basin of attraction of the double-covered circle. This observation conforms with the analysis of {\cite{GRvdM}}, which shows that the double-covered circle is the ``natural'' bending energy minimizer for the trefoil knot type. The more interesting point here is that we can find \emph{unknotted} curves on the stable manifold as well. For any $|\veps_{1}| \leq 1/4$ and any $\veps_{2}$ arbitrary consider the corresponding tangent field
$$
\tang_{ (\veps_{1},\veps_{2})}(x) = \tangc(x) - \veps_1 \sin x \, \normc(x)
+\veps_2 \cos 5x \, \binc,
$$
which in turn yields
$$
\gamma_{\tang_{ (\veps_{1},\veps_{2})} }(x)  = \gamma_{ {\rm dc} }(x) +
\begin{bmatrix}
\veps_1 (3 \cos x - \cos 3x)/6 \\
\veps_1 (2\sin^3 x)/3 \\
\veps_2 (\sin 5x)/5
\end{bmatrix}
$$
as the corresponding induced curve. If $p:\mathbb{R}^3\rightarrow \mathbb{R}^2$ denotes the standard projection map $p(x,y,z)=(x,y)$ then the curve $p(\gamma_{\tang_{ (\veps_{1},\veps_{2})} })$ is a regular projection with a single double point. Moreover, since $|\veps_1|\leq 1/4$ a simple calculation shows that the two preimages of this double point on the curve $\gamma_{\tang_{ (\veps_{1},\veps_{2})} }$ have distinct $z$-components. Thus $\gamma_{\tang_{ (\veps_{1},\veps_{2})} }$ corresponds to an embedding of the circle and $\gamma_{\tang_{ (\veps_{1},\veps_{2})} }$ has a knot diagram with a single crossing. In other words, for any $|\veps_{1}| \leq 1/4$ and any $\veps_{2}$ the curve $\gamma_{\tang_{ (\veps_{1},\veps_{2})} }$ is unknotted. Given such an initial condition $\tang_{ (\veps_{1},\veps_{2})},$ the representation \eqref{eq:dcfinal} then shows
$$
\tang_{ (\veps_{1},\veps_{2})}(t) = \tangc(x) - \veps_1 \re^{-t} \sin x \, \normc(x)
+\veps_2 \re^{-21 t} \cos 5x \, \binc
$$
to leading order, and so $\tang_{ (\veps_{1},\veps_{2})}(t)$ induces an unknotted curve $\gamma_{\tang_{(\veps_{1},\veps_{2})}(t)}$ for all finite time. To leading order at least, this construction therefore yields a family of curves $\gamma_{ \tang(t)}$ that flow under the bending energy, remain unknotted for all time and collapse onto the double-covered circle in the limit. We explore this possibility numerically in the next section.

\section{Numerical Experiments}

We conclude our study by briefly summarizing a set of numerical experiments analyzing the long-term behavior of the flow
\begin{equation}\label{eq:numerics}
\tang_{t}(t) = \tang_{xx}(t) + \mathbf{F}_{ \gamma_{\tang(t)} } + \lambda_{\tang(t)} + \mu_{\tang(t)}\tang(t) := \tang_{xx}(t) + G(\tang(t),\tang_{x}(t)).
\end{equation}
These experiments partially validate our {analytical} study of dynamics and supplement it with an illustration of global, rather than local, dynamics. For all experiments, we numerically integrate \eqref{eq:numerics} using the following relatively simple procedure. Given a time discretization $0 = t_0 < t_1 < \cdots < t_{n} < t_{n+1}$ with $dt_{n} := t_{n+1} - t_{n}$ and $\tang^{n} := \tang(t_n)$ we approximate \eqref{eq:numerics} using the second-order in time discretization
\begin{align}\label{eq:discretetime}
\frac{\tang ^{n+1} - \tang^{n} }{dt_n} &= \frac1{2}\left( \tang^{n+1}_{xx} + \tang^{n}_{xx} \right) + G(\tang^{n},\tang^{n}_{x}) + \frac{dt_n}{2 dt_{n-1}}\left( G(\tang^{n},\tang^{n}_{x}) - G(\tang^{n-1},\tang^{n-1}_{x}) \right) \\
\left(\mathrm{Id} - \frac{dt_n}{2} \partial_{xx}\right) \tang^{n+1} &= \tang^n + dt_{n}\left( \frac1{2}\tang^{n}_{xx} + G(\tang^{n},\tang^{n}_{x}) + \frac{dt_n}{2 dt_{n-1}}\left( G(\tang^{n},\tang^{n}_{x}) - G(\tang^{n-1},\tang^{n-1}_{x}) \right) \right) := \mathbf{r}^n \nonumber
\end{align}
with periodic boundary conditions. We compute $\mathbf{F}_{ \gamma_{\tang^n} }$ by analogy to the decomposition in lemmas \ref{lem:lapest},\ref{lem:boundF} --- we compute the singular component $g(1)L_{p}[\gamma_{\tang^n}]$ spectrally and the regular component $f_{\gamma_{\tang^n}} - g(1)L_{p}[\gamma_{\tang^n}]$ using quadrature. We then form $\mathbf{r}^{n}$ and solve the resulting screened Poisson equation $(\mathrm{Id} - (dt_n/2)\partial_{xx})\tang^{n+1} = \mathbf{r}^{n}$ spectrally. Finally, we use an adaptive selection of $dt_n$ at each step. {\color{black} We first select $dt_n$ to guarantee energy descent by imposing a discrete version of the energy equality \eqref{eq:enerid}, in that 
$$
E(t_{n+1}) = E(t_{n}) - \int^{t_{n+1}}_{t_{n}} \| \tang_{t}(s) \|^{2}_{L^{2}(\T)} \, \rd s \quad \rightarrow \quad E(t_{n+1}) \leq E(t_{n}) - \frac{dt_{n}}{2} \| \tang_{t}(t_{n})\|^{2}_{L^{2}(\T)}
$$
for all $dt_n$ sufficiently small; we simply choose $dt_n$ small enough to guarantee the latter inequality at all stages of the simulation. By taking $dt_n$ smaller if necessary we may also ensure that the inequality
$$
\sup_{x\neq y} \frac{ | ( \gamma_{\tang^{n+1}}(x) - \gamma_{\tang^{n}}(x) ) - ( \gamma_{\tang^{n+1}}(y) - \gamma_{\tang^{n}}(y) )|}{\vartheta(x-y)} \leq \frac1{2}\delta^{-1}_{\infty}( \gamma_{\tang^n} )
$$
holds between successive timesteps. If we linearly interpolate between $\gamma_{\tang^n}$ and $\gamma_{\tang^{n+1}}$ on $[t_n,t_{n+1}]$ via
$$
\gamma(t) = \gamma_{\tang^{n}} + \frac{(t-t_{n})}{dt_n}\left( \gamma_{\tang^{n+1}} - \gamma_{\tang^n} \right),
$$
then $\delta_{\infty}( \gamma(t) ) \leq 2\delta_{\infty}( \gamma_{\tang^n} )$ on $[t_n,t_{n+1}]$ by the triangle inequality. In particular, self-intersections do not occur.} Overall we obtain a second-order accurate algorithm.

We begin with a numerical exploration of the ``dynamic'' Smale conjecture. In other words, all unknotted initial conditions $\tang_0$ should flow under \eqref{eq:numerics} to the unit circle whenever $K(u,v)$ has properties analogous to the $2$-Distortion, namely $E_{K}(\gamma)$ should have $\gamma_{ {\rm circ}}$ as the global minimizer and $E_{ {\rm b},K}(\gamma)$ should bound the distortion from above. Figure \ref{fig:supercoil} illustrates the dynamics of \eqref{eq:numerics} for a complicated embedding of the unknot meant to exhibit properties loosely analogous to those observed in the natural packing of DNA. The spherical confinement of the initial embedding proves analogous to DNA packing in a bacteriophage, and we also mimic the ``supercoiling'' observed in packed DNA by combining a high amplitude coiling of period $\lenG/8$ and a small amplitude coiling of period $\lenG/128$ to form the initial condition. We simulate \eqref{eq:numerics} using the $12$-Distortion $K(u,v) = (v/u)^{6}$ in order to incorporate the $|\x-\y|^{-12}$ repulsive strength of a Lennard-Jones style kernel between distinct points $\x,\y$ on the curve. At $t=0$ we first scale the initial curve $\gamma_{0}$ so that we have an initial balance $E_{b}(\gamma_0) = E_{ K}(\gamma_0)$ between constituent energies. We then use \eqref{eq:discretetime} to simulate the dynamics forward in time. The dynamics essentially occur in three distinct stages. Figure \ref{fig:supercoil} displays the resulting induced curves $\gamma_{\tang(t)}$ at various time points during these stages of evolution; the color of a point along $\gamma$ indicates its distance to the center of mass of the evolving structure. During the first stage of the evolution, shown in panels (c,d) of figure \ref{fig:supercoil}, curvature dominates the dynamics as it forces a rapid dissipation of the small period coiling of the initial condition. After most oscillations have damped, a second phase occurs where both repulsion and diffusion operate non-trivially. The larger scale coils of $\gamma_{\tang(t)}$ dissipate slowly and simultaneously experience an outward radial {repulsion}, as shown in figure \ref{fig:supercoil} (d,e,f), while the few small {period coils} that remain quickly dissipate. By $t=.0011$ only large scale oscillations remain, and in the final phase (figure \ref{fig:supercoil} (g,h)) curvature once again dominates the evolution as it dissipates these remaining oscillations until the unknot eventually appears.

\begin{figure}[h]
\begin{minipage}{45em}
\centering
\begin{minipage}{10em}
\centering
\subfigure[$t=0$]{ \includegraphics[width=1.5in]{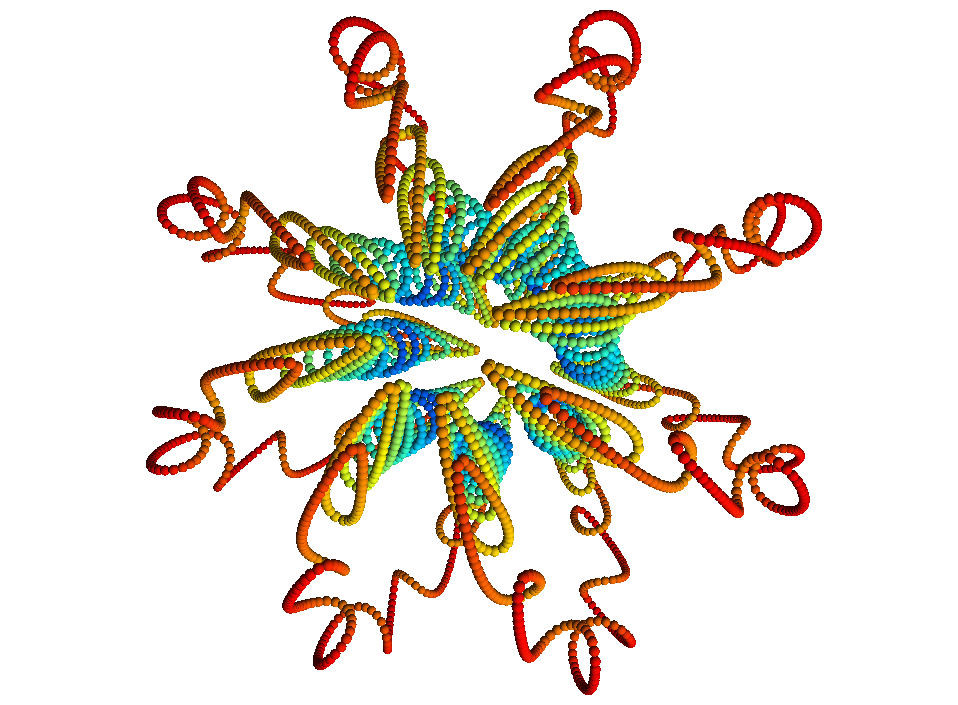} } \\ \subfigure[$t=0$]{ \includegraphics[width=1.5in]{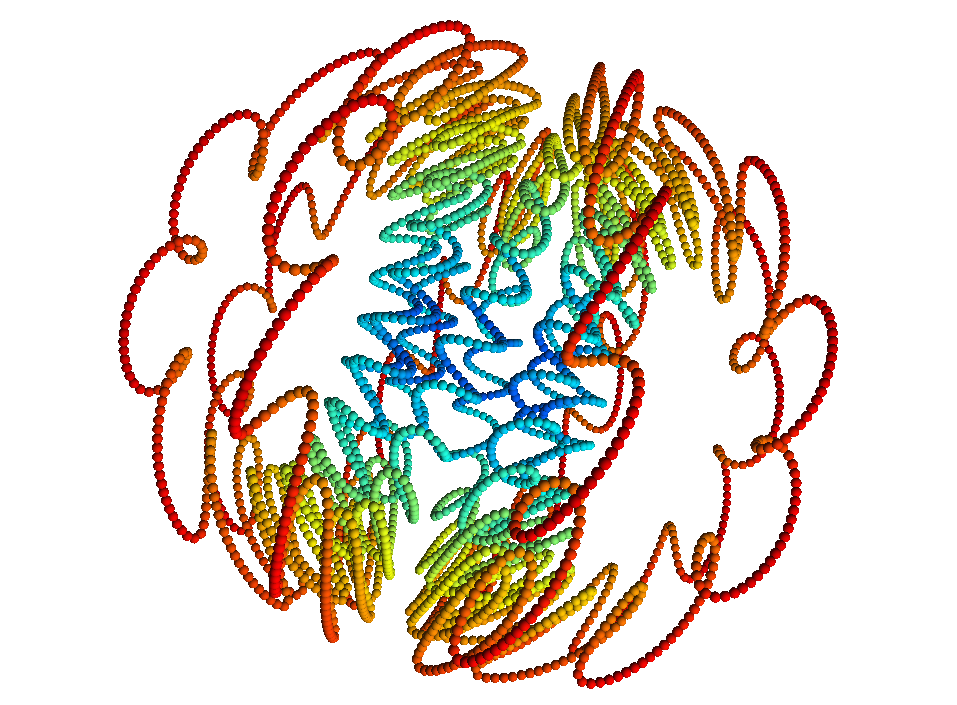} }
\end{minipage}\hspace{.1in}
\begin{minipage}{10em}
\centering
\subfigure[$t=5.435 \times 10^{-6}$]{ \includegraphics[width=1.5in]{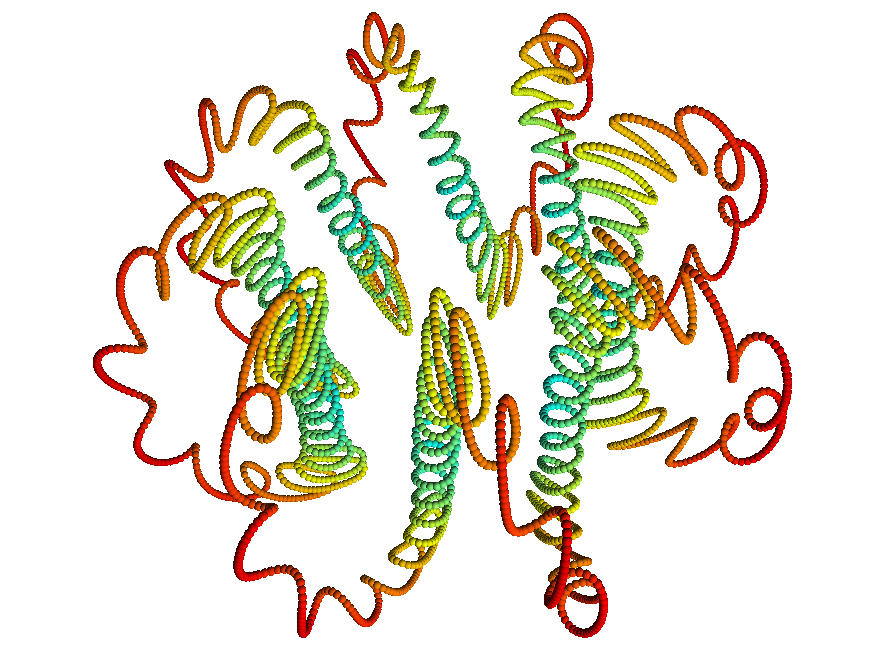} } \\ \subfigure[$t=8.9275 \times 10^{-5}$]{ \includegraphics[width=1.5in]{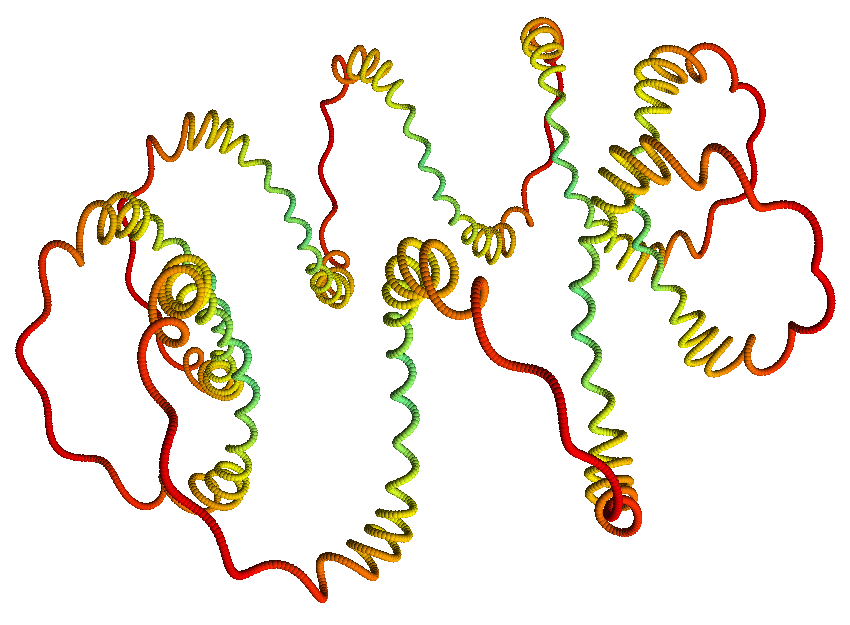} }
\end{minipage}\hspace{.25in}
\begin{minipage}{10em}
\centering
\subfigure[$t=2.7616 \times 10^{-4}$]{ \includegraphics[width=1.5in]{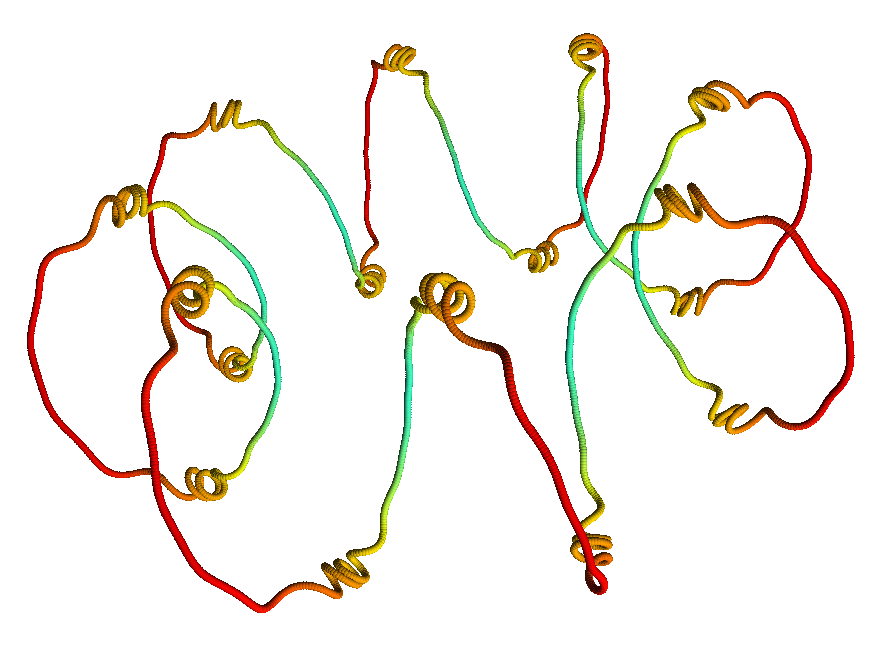} } \\ \subfigure[$t=0.0011$]{ \includegraphics[width=1.5in]{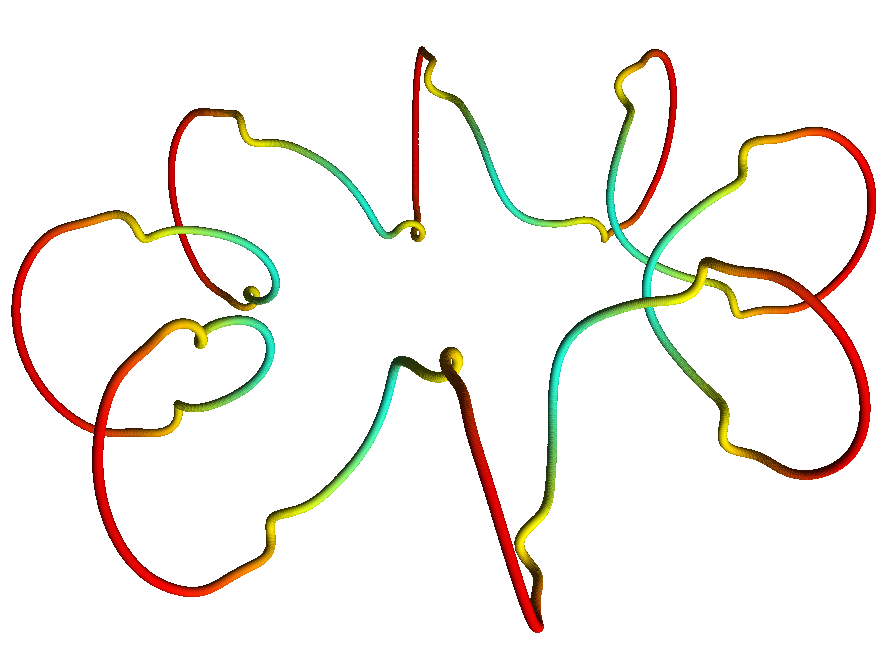} }
\end{minipage}\hspace{.1in}
\begin{minipage}{10em}
\centering
\subfigure[$t=0.0065$]{ \includegraphics[width=1.5in]{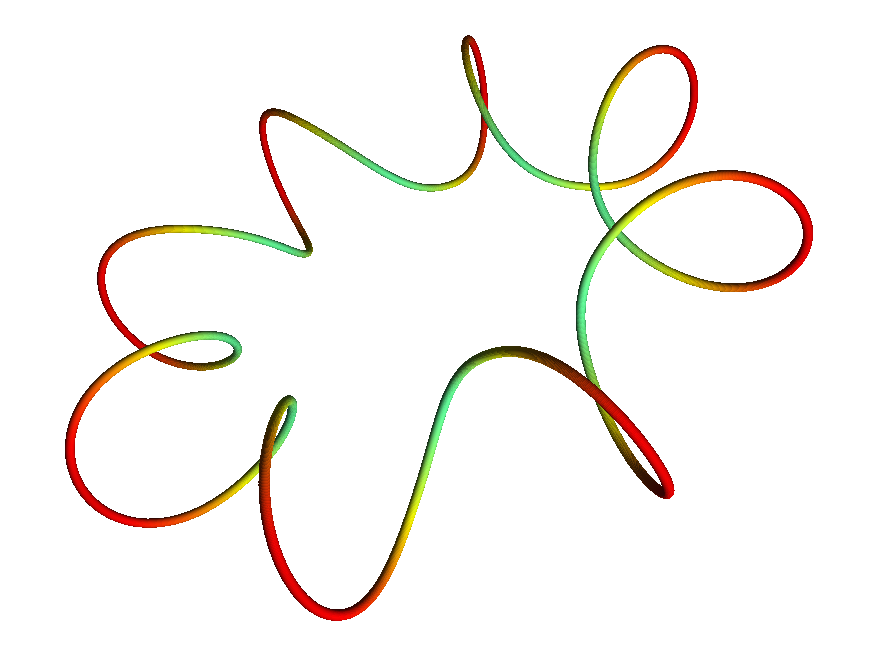} } \\ \subfigure[$t=0.35$]{ \includegraphics[width=1.5in]{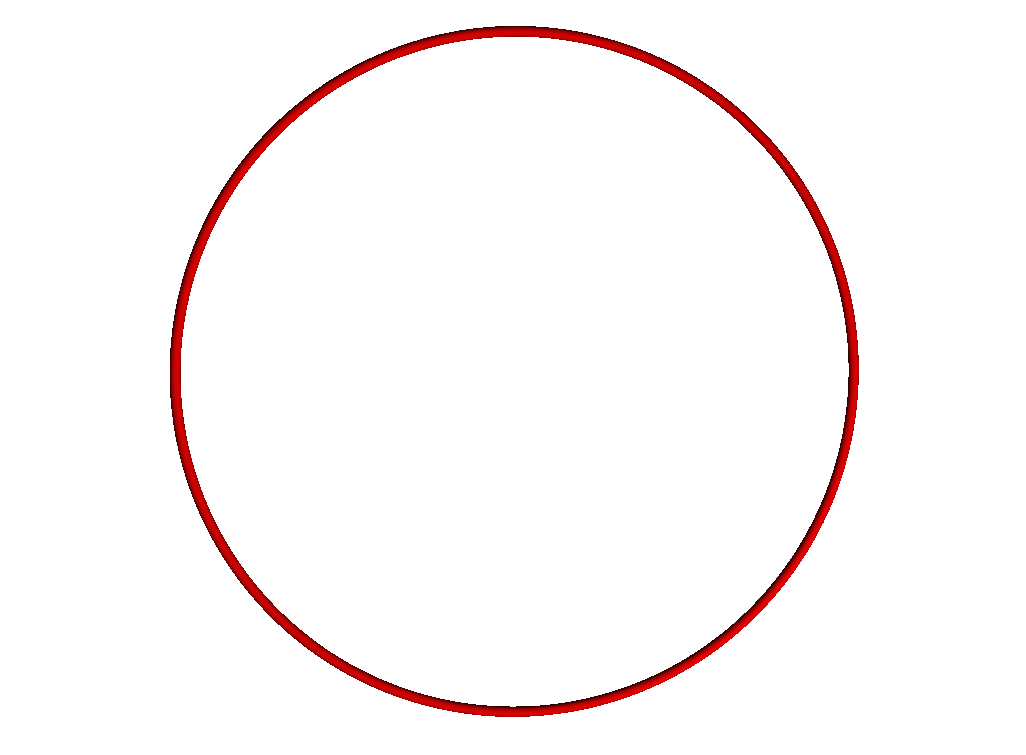} }
\end{minipage}
\end{minipage}
\caption{Unfolding of an unknotted ``supercoiled DNA'' molecule. Color at each point indicates the distance of that point to the center of mass.}
\label{fig:supercoil}
\end{figure}

\begin{figure}[h]
\centering
{
\subfigure[$t=0$]{ \includegraphics[width=2.0in]{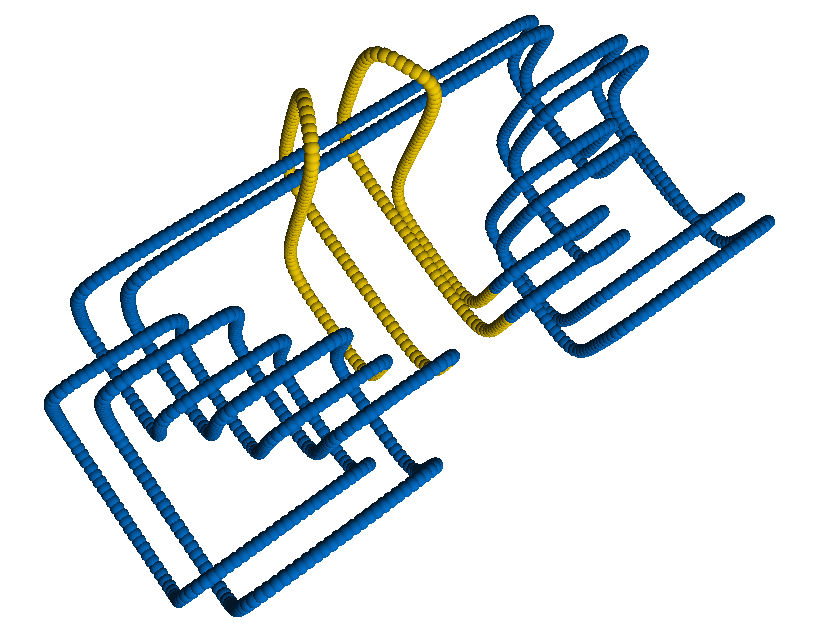} } \subfigure[$t=0.0163$]{ \includegraphics[width=2.0in]{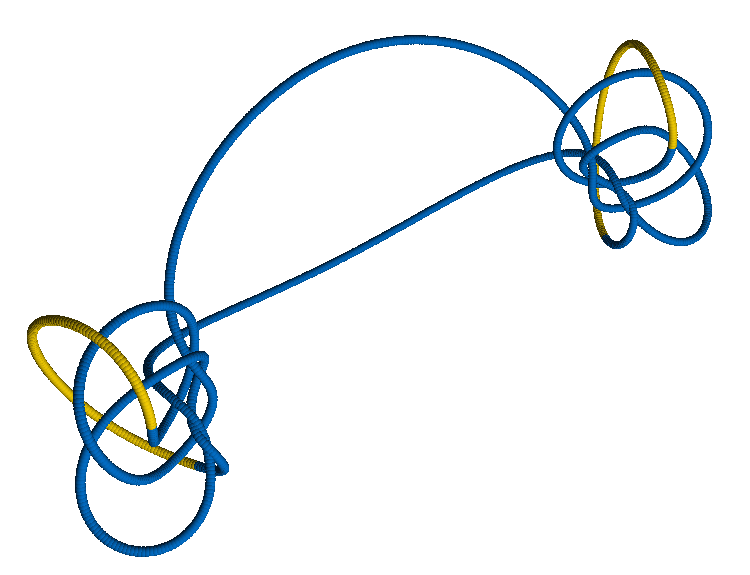} } \subfigure[$t=0.0563$ (left lobe)]{ \includegraphics[width=2.0in]{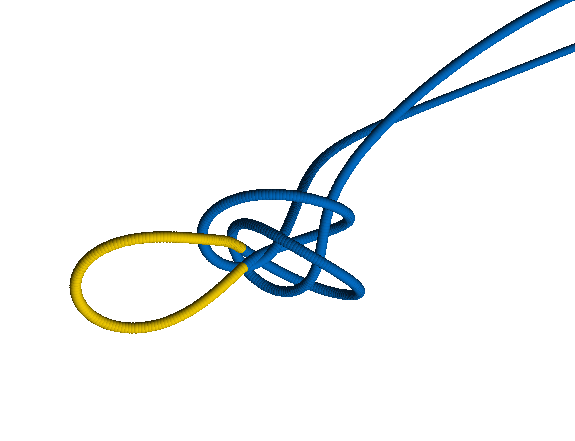} } \\
\subfigure[$t=0.1193$ (left lobe)]{ \includegraphics[width=2.0in]{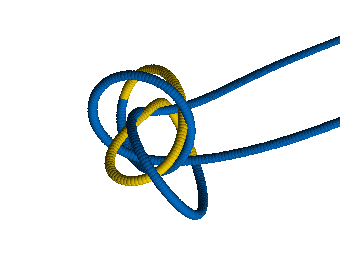} } \subfigure[$t=0.1617$ (left lobe)]{ \includegraphics[width=2.0in]{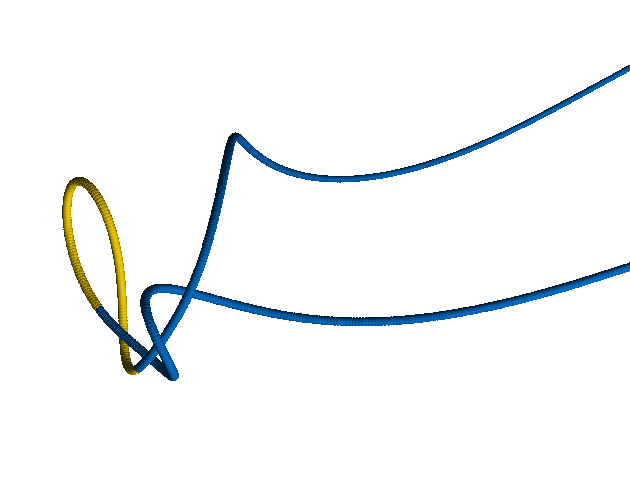} } \subfigure[$t=1$]{ \includegraphics[width=2.0in]{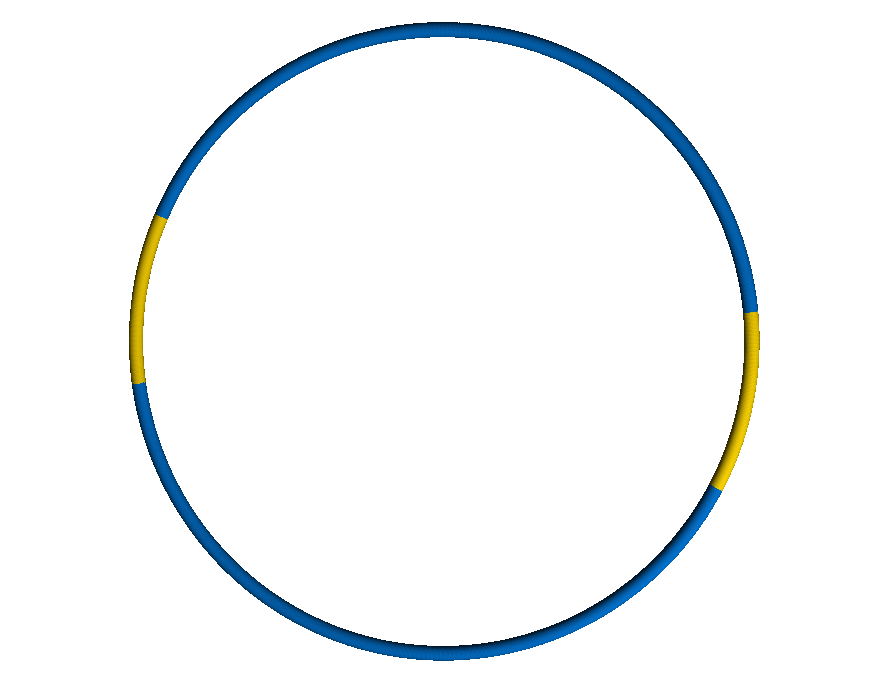} }
}
\caption{Unfolding of a ``hard'' unknot. At all times points in yellow correspond to the initial ``clasps'' highlighted in the initial condition.}
\label{fig:gordian}
\end{figure}

Figure \ref{fig:gordian} also explores the ``dynamic'' Smale conjecture, but for a pathologically difficult embedding of the unknot. It illustrates the dynamics of \eqref{eq:numerics} for  an example lying within a well-known infinite family of ``hard'' unknots constructed by taking the connected sum of two knots (in this instance, both trefoils), doubling the resulting knot by forming an index-two cable of the connected sum and finally performing surgery on the cable knot in a neighborhood of the summing sphere in order to produce an unknot. The original knot deforms to the standard unknot by first isotoping the yellow strands around the large knotted bodies on either symmetric side of the structure. Pulling along the ``tracks'' of the doubled trefoil structures then eliminates the knotted bodies, which results in a planar embedding of the unknot. The dynamics of \eqref{eq:numerics} provide an approximate version of this process. We simulate \eqref{eq:numerics} using the $2$-Distortion $K(u,v) = v/u$ for this example with an initial condition $\gamma_0$ scaled to have initial equipartition of energy $E_{b}(\gamma_0) = E_{ {\rm K}}(\gamma_0)$ as before. Panels (b,c) of figure \ref{fig:gordian} show the initial phase wherein the yellow strands isotope around the knotted bodies. This process has finished by $t=.0563$, and afterward the two yellow strands begin to pull through the trefoil structures. Panels (c,d) show the onset and midpoint of this process for the left side of the {structure}; the dynamics for the right of the structure are simply the mirror image across the initial plane of symmetry. By $t=.1617,$ shown in panel (e), the embedding differs from the standard embedding of the unknot only by trivial twisting. The remaining dynamics simply perform local rotations of arcs of the embedding to perform the final simplification of the structure into the standard unknot.

\begin{figure}[h]
\centering
{
\subfigure[$t=0$]{ \includegraphics[width=2.0in]{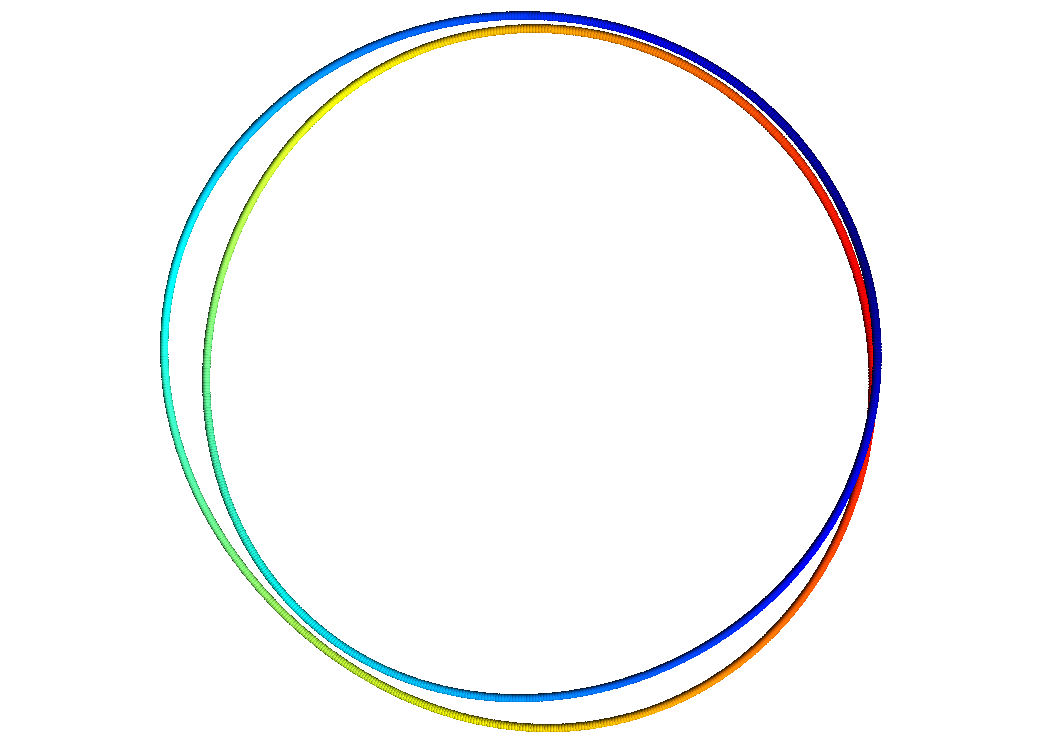} } \subfigure[$t=0.0529$]{ \includegraphics[width=2.0in]{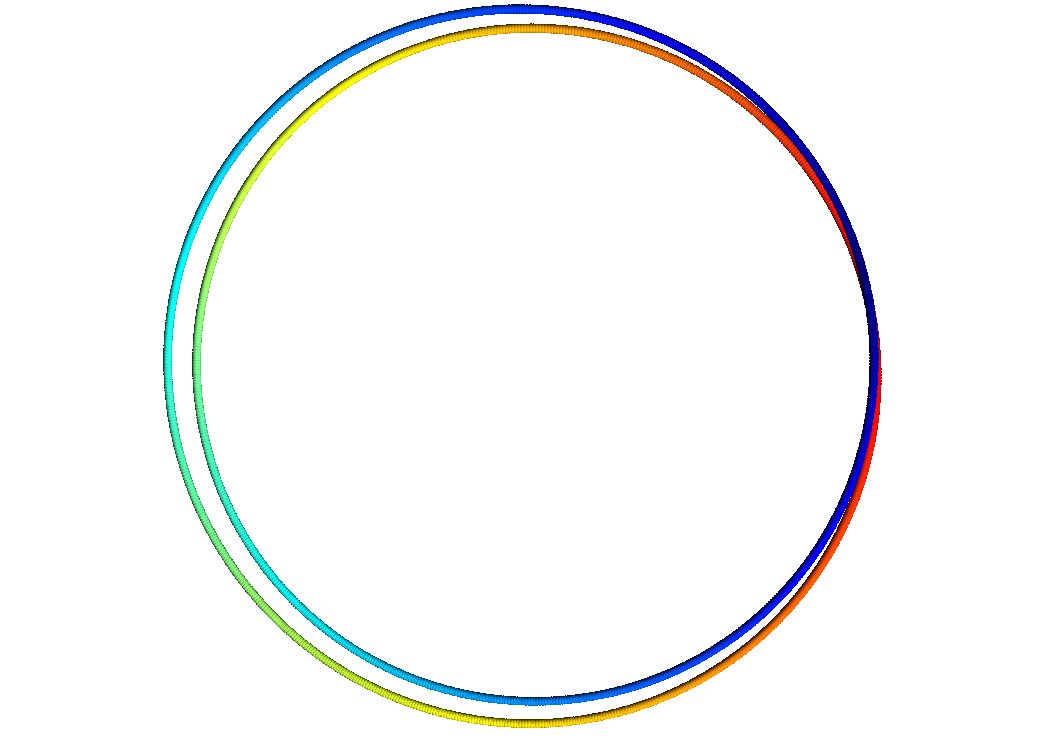} } \subfigure[$t=2.5$]{ \includegraphics[width=2.0in]{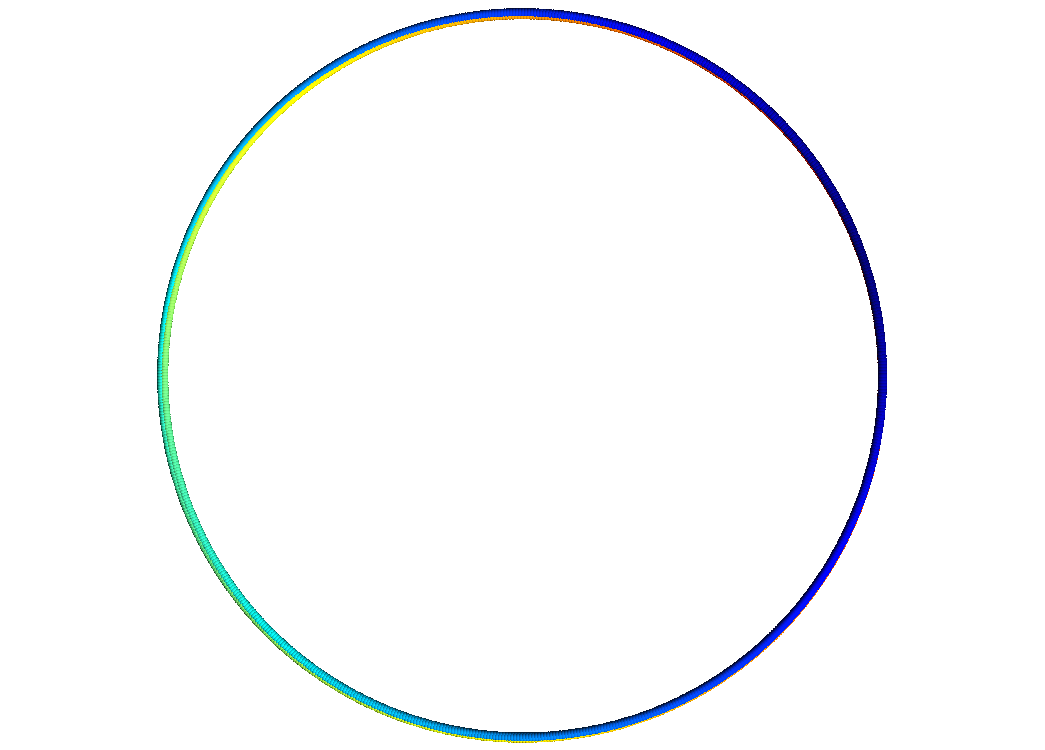} } \\
\subfigure[$t=0$]{ \includegraphics[width=2.0in]{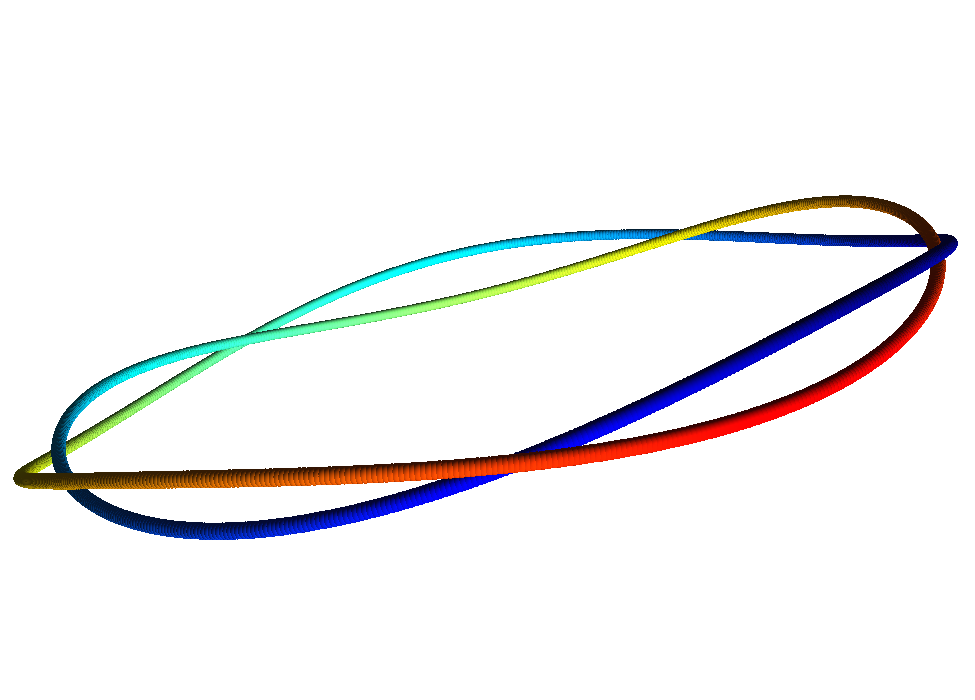} } \subfigure[$t=0.0529$]{ \includegraphics[width=2.0in]{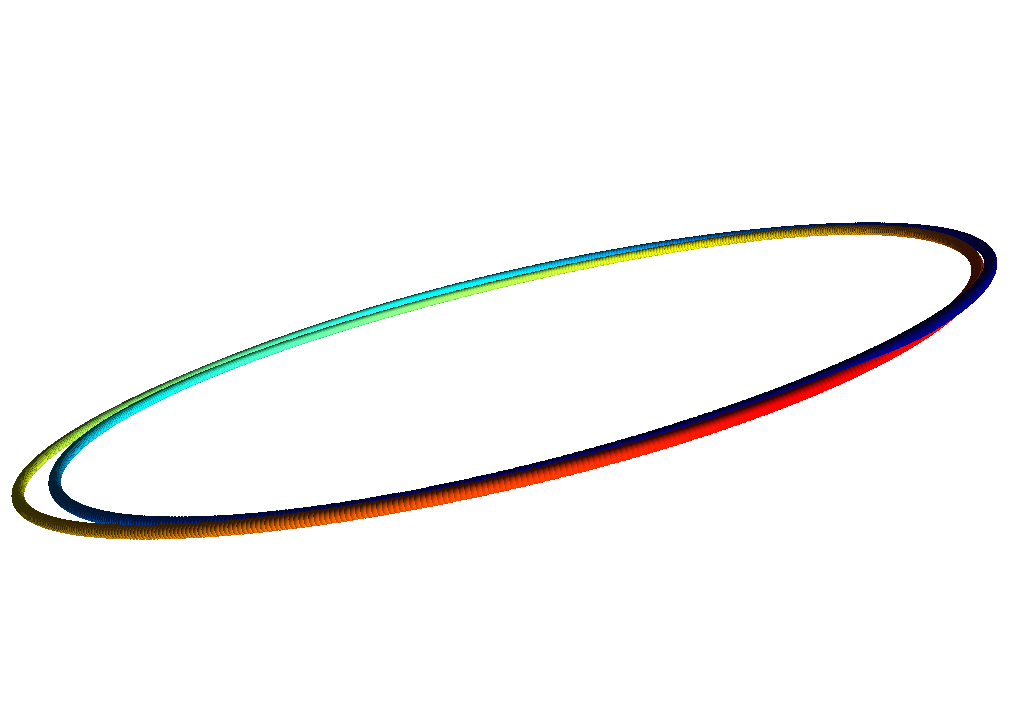} } \subfigure[$t=2.5$]{ \includegraphics[width=2.0in]{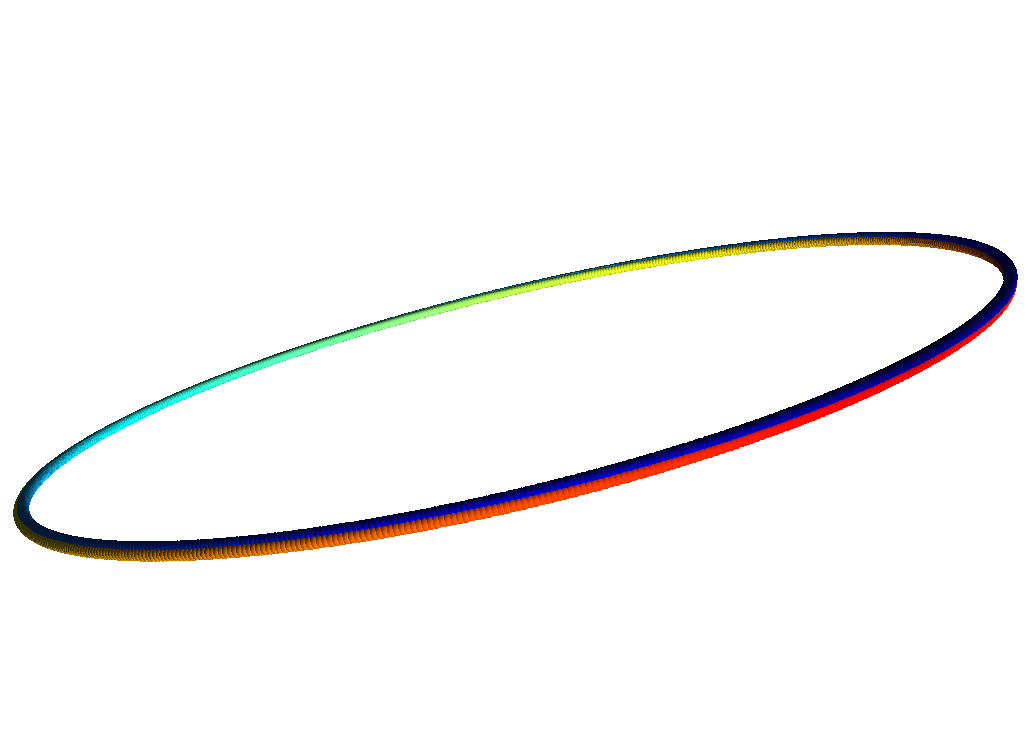} }
}
\caption{The dynamics of a pure bending energy flow that collapses an unknot onto the double-covered circle. Color indicates the path-distance along the embedding from a fixed reference point.}
\label{fig:dc}
\end{figure}

We conclude our experiments with a numerical simulation illustrating how, in the absence of the repulsive energy $E_{K}(u,v),$ the dynamics of \eqref{eq:numerics} can flow unknots into the double-covered circle. We use the unknotted initial condition
$$
\tang_{ (\veps_{1},\veps_{2})}(x) = \tangc(x) - \veps_1 \sin x \, \normc(x)
+\veps_2 \cos 5x \, \binc,
$$
from section \ref{sssec:dc} with $\veps_2 = .15$ and $\veps_1 = \veps_2/2$ and simulate a pure bending energy flow. Figure \ref{fig:dc} illustrates that the dynamics do, in fact, obey our approximation
$$
\tang_{ (\veps_{1},\veps_{2})}(t) = \tangc(x) - \veps_1 \re^{-t} \sin x \, \normc(x)
+\veps_2 \re^{-21 t} \cos 5x \, \binc
$$
to leading-order. The binormal component of $\gamma_{\tang(t)}$ (i.e. the $\e_3$-axis) vanishes exponentially quickly, as shown in panel (b,d), in comparison to the dynamics {occuring} in the $\tangc-\normc$ plane. These planar dynamics then simply equilibrate the inner and outer ``circles'' until both coincide, yielding the double-covered circle. In particular, at each stage there exists a projection of $\gamma_{\tang(t)}$ onto the  $\tangc/\normc$ plane with a single crossing, as in figure \ref{fig:dc} (b), and while their preimages are exponentially close they are, in fact, distinct. In other words $\gamma_{\tang(t)}$ remains unknotted for all time yet $\gamma_{\tang(\infty)}$ is emphatically not the global minimizer of the bending energy within the {ambient} isotopy class $\K$ of the unknot. So while the ``dynamic'' Smale conjecture plausibly holds in the presence of a repulsive energy $E_{K}(u,v),$ this example demonstrates that it is emphatically false in its absence of such a repulsive force. This example is perhaps not surprising, but it nevertheless cautions against a naive intuition about how a realization of the ``dynamic'' Smale conjecture might unfold. Specifically, the fact that an energy yields $\gamma_{ {\rm circ}}$ as its global minimizer in no way implies that a corresponding notion of ``convexity'' with respect to the unknot type must hold.

\section{Conclusion}
{\color{black}
Our present focus lies on those particular instances of the general family \eqref{eq:introenergy} that contain two main physical effects, namely curvature-driven diffusion and nonlocal repulsion. {\color{black} Combining these effects yields a strong energetic model from a combinatorial complexity point-of-view, as lemma \ref{lem:mainenergy} shows.} A weak repulsion (as weak as the $2-$Distortion) combines with a fractional diffusion ($\partial^{s}_{xx}$ for $s > \frac1{2}$) to give a rather strong (say finite-to-one) invariant. In the physically relevant case ($s = 1$) these effects lead to a set of dynamical models with well-behaved dynamics, in the sense that they exhibit desirable global existence and regularity properties for a large class of nonlocal kernels (c.f. theorem \ref{thm:global}). These dynamics themselves, in particular their doubly-nonlocal nature, appear novel and also admit reliable numerical approximation wherein self-intersections are easy to avoid. A further analysis highlights a few aspects of the dynamics themselves. In the presence of repulsion, techniques for nonlocal dynamical systems show that standard circles are orbitally stable at a local level, while at a global level numerical experiments reveal that the basin of attraction for circles contains rather complex initial embeddings. We use a similar combination of techniques to illustrate a simple but important observation, that in the absence of repulsion there exist unknotted initial configurations that remain unknotted for all time under a bending energy flow, yet the dynamics do not asymptotically converge to the globally minimal, unknotted circle. Finally, the overall methodology, and in particular the equilibrium analysis, extends readily to related settings such as systems of interacting curves (e.g. knots and links) or to nonlocal interactions involving both attraction and repulsion. Studying variants along these lines provide fertile ground for future work, as does a more in-depth study of the weighted crossing number \eqref{eq:gcn} from a purely knot-theoretic point-of-view.}

\bibliographystyle{plain}
\bibliography{bibenergy}

\section{Appendix}

\begin{lemma}{ {\rm (Lemma \ref{lem:distbound}) } }
\label{lem:distboundA}
Assume that $K(u,v)$ satisfies the $h/p$ homogeneity property
$$
K\left( \frac{v}{\alpha} , v \right) = h(\alpha) v^{-p} \quad \text{for all} \quad \alpha \geq 1
$$
for some exponent $p \geq 0$ and some function $h \in C^{1}((1,\infty) )$. Assume that the lower bound $h^{\prime}(\alpha) \geq \lambda_{h} > 0$ holds on $[2,\infty)$ as well. If $\gamma \in C^{0,1}(\T)$ has unit speed then
$$
\log\left( 1 + c_{0} \delta_{\infty}(\gamma) \right) \leq C_{0} \left( 1 + E_{K}(\gamma) \right) \dd^{-2}_{\infty}(\gamma)
$$
for $C_0(p,h),c_0(p,h) > 0$ finite constants. Moreover, if $\gamma \in H^{2}(\T)$ then
\begin{equation*}
\log\left( 1 + c_{0} \delta_{\infty}(\gamma) \right) \leq C_0 \left( 1 + E_{K}(\gamma) \right) \|\ddot{\gamma}\|^{4}_{L^{2}(\T)},
\end{equation*}
and so $\gamma$ is bi-Lipschitz if $E_{K}(\gamma)$ is finite.
\end{lemma}

\begin{proof}
Let $(x_0,y_0)$ denote a distortion-realizing pair so that
$$
\dd(x_0 - y_0) = \dd_{\infty}(\gamma),
$$
where $\dd_{\infty}(\gamma)$ denotes the maximal distance between a distortion realizing pair. For any two such points $\gamma(x_0)$ and $\gamma(y_0)$ define
$$
z_0 := \dd(x_0 - y_0) > 0, \qquad d_0 := |\gamma(x_0) - \gamma(y_0)| \qquad \text{and} \qquad \delta_{0} := z_0/d_0 = \delta_{\infty}(\gamma)
$$
and use the $1$-Lipschitz property of $\gamma$ along with the inequality $2xy \leq \sigma x^{2} + y^{2}/\sigma$ to estimate
\begin{align*}
|\gamma(x) - \gamma(y)|^{2} & \leq\left( (2 + \sigma)(\dd^{2}(x-x_0) + \dd^{2}(y-y_0) ) + \left(1 + \frac{2}{\sigma}\right) d^{2}_{0} \right)
\end{align*}
for $x,y \in \T$ any pair of arbitrary points and $\sigma > 0$ an arbitrary parameter. Define $\alpha(x,y)$ as
$$
\alpha(x,y) := \frac{ \dd^{2}(x-y) }{|\gamma(x) - \gamma(y)|^2} \geq \frac{ \dd^{2}(x-y) }{ \left( (2 + \sigma)(\dd^{2}(x-x_0) + \dd^{2}(y-y_0) ) + \left(1 + \frac{2}{\sigma}\right) d^{2}_{0} \right) },
$$
the latter inequality holding due to the previous estimate. In particular, if the pair $(x,y)$ lies in the set
$$
A_{\epsilon}(z_0) := \left\{ (x,y) \in \T \times \T : \dd(x-x_0) \leq \frac{\epsilon}{2} z_0, \;  \dd(y-y_0) \leq \frac{\epsilon}{2} z_0\right\}
$$
then $(1-\epsilon)z_0 \leq \dd(x-y) \leq (1+\epsilon)z_0$ and so
$$
\alpha(x,y) \geq \frac{ (1-\epsilon)^{2} z^2_0 }{ (1+ \frac{\sigma}{2})z^2_0 \epsilon^{2} + (1+ \frac{2}{\sigma})d^2_0} = \frac{ (1-\epsilon)^2 }{ (1+ \frac{\sigma}{2}) \epsilon^{2} + (1+ \frac{2}{\sigma})\delta^{-2}_{\infty}(\gamma) }
$$
holds as well. As $\delta_{\infty}(\gamma) \geq \pi/2,$ the lower bound
$$
\alpha(x,y) \geq  \frac{ \pi^2(1-\epsilon)^2 }{ (1+ \frac{\sigma}{2}) \pi^{2} \epsilon^{2} + 4(1+ \frac{2}{\sigma}) } \geq 2
$$
holds for all $\sigma$ sufficiently large and $\epsilon$ sufficiently small. {\color{black}For example, for $\epsilon= \frac{1}{1000}$ and $\sigma=10$,} it therefore follows that $h( \alpha(x,y) ) \geq h(2) + \lambda_{h}( \alpha(x,y) - 2)$ on $A_{\epsilon},$ and so the lower bound
\begin{align*}
&4(1+\epsilon)^{2p}z^{2p}_0 E_{K}(\gamma) \geq ( h(2) - 2 \lambda_{h} ) |A_{\epsilon}| + \lambda_{h} \int_{A_{\epsilon}(z_0)} \alpha(x,y)  \, \rd y \rd x \\
&\geq ( h(2) - 2 \lambda_{h} ) |A_{\epsilon}| + \lambda_{h}(1-\epsilon)^{2}z^{2}_{0} \int_{ \T \times \T }  \frac{ \mathbf{1}_{\{ \dd(x-x_0) \leq \frac{\epsilon}{2} z_0 \} }(x)  \mathbf{1}_{\{ \dd(y-y_0) \leq \frac{\epsilon}{2} z_0 \} }(y)  }{ \left( (2 + \sigma)(\dd^{2}(x-x_0) + \dd^{2}(y-y_0) ) + \left(1 + \frac{2}{\sigma}\right){\color{black}d}^{2}_{0} \right) } \, \rd y \rd x
\end{align*}
follows. Now make the change of variables $\xi := x - x_0$ and $\eta = y - y_0$ and use the $2\pi$-periodicity of $z \mapsto \dd(z)$ to find
\begin{align*}
4(1+\epsilon)^{2p}z^{2p}_0 E_{K}(\gamma) + (  2 \lambda_{h} - h(2) ) |A_{\epsilon}| &\geq \lambda_{h}(1-\epsilon)^{2}z^{2}_{0} \int_{ \T \times \T }  \frac{ \mathbf{1}_{ \{ |\xi| \leq \frac{\epsilon}{2} z_0 \} }(\xi) \mathbf{1}_{ \{ |\eta| \leq \frac{\epsilon}{2} z_0 \} }(\eta)  }{ \left( (2 + \sigma)( \xi^{2} + \eta^{2} ) + \left(1 + \frac{2}{\sigma}\right){\color{black}d}^{2}_{0} \right) } \, \rd \eta \rd \xi,\\
&\geq \lambda_{h}(1-\epsilon)^{2}z^{2}_{0} \int_{ \T }\int^{  \frac{\epsilon}{2}z_0 }_{0} \frac{ s \, \rd s \rd \theta }{ \left( (2 + \sigma) s^{2}  + \left(1 + \frac{2}{\sigma}\right){\color{black}d}^{2}_{0} \right) },
\end{align*}
where the last inequality follows by shifting to polar coordinates. Performing the final integration then shows
\begin{align*}
4(1+\epsilon)^{2p}z^{2p-2}_0 E_{K}(\gamma) - (h(2) - 2 \lambda_{h} ) |A_{\epsilon}|z^{-2}_{0} &\geq
{\color{black}\lambda_{h}}\frac{ \pi(1-\epsilon)^{2}  }{(2+\sigma)}  \log\left( 1 +  \frac{ \sigma \epsilon^{2}}{4}\delta^2_\infty(\gamma)  \right).
\end{align*}
The final claim also follows from this inequality by noting that {\color{black}(c.f.\eqref{eq:z0lower})}.
$$
z^2_0 \geq {\color{black}\frac{1}{16}} \left( \frac{ \delta_{\infty}(\gamma) - 1 }{\delta_{\infty}(\gamma) + 1 } \right)^{4}\| \ddot{\gamma}\|^{-4}_{L^{2}(\T)} \geq c \|\ddot{\gamma}\|^{-4}_{L^{2}(\T)}
$$
for $c>0$ a universal constant.
\end{proof}

\begin{lemma}{ {\rm (Lemma \ref{lem:mainenergy})} }
\label{lem:mainenergyA}
Assume that $\gamma \in C^{0,1}(\T)$ and that $\gamma$ has finite distortion. Assume further that $\gamma$ has constant speed. Then
\begin{equation*}
c_{p}(\gamma) \leq  C_{p}\left( \frac{2\pi}{\lenG} \right)^2  \|\gamma\|^{2}_{ \dot{H}^{\frac{3+p}{2}}(\T)} \delta^3_{\infty}(\gamma),
\end{equation*}
and if $-1 < p < 1$ then the constant
$$
C_{p} := 2\sqrt{2}(2\pi)^{p+1}\left( \int^{\infty}_{0} \frac{(1-\cos(u)) }{u^{2+p} } \, \rd u \right)^{\frac1{2}} \left( \int^{\infty}_{0}\frac{(1-\cos(u))^{2} + (\sin u - u)^2 }{u^{4+p} }  \, \rd u\right)^{\frac1{2}}
$$
is finite.
\end{lemma}

\begin{proof}
That $\gamma$ has constant speed yields $|\dot{\gamma}(x)| = \lenG/2\pi,$ and so the relations
$$
\frac{2\pi \delta_{\infty}(\gamma)}{\lenG} = \sup_{ \{(x,y) : \dd(x-y)>0\} } \, \frac{\dd(x-y)}{|\gamma(x) - \gamma(y)|} \qquad \text{and} \qquad \kappa_{\gamma}(x) = \frac{4\pi^2 |\ddot{\gamma}(x)|}{\mathrm{len}^{2}(\gamma)}
$$
furnish the distortion and pointwise curvature. Note first that the equality
$$
\langle \dot{\gamma}(x) , \dot{\gamma}(x+z) , \gamma(x+z) - \gamma(x) \rangle = \langle \dot{\gamma}(x) , \dot{\gamma}(x+z) - \dot{\gamma}(x), \gamma(x+z) - \gamma(x) - z\dot{\gamma}(x) \rangle
$$
holds by the definition of and the trilinearity of the triple product. Set $A := |\dot{\gamma}(x)| = \lenG/2\pi$ and note that the inequality
$$
| \langle \dot{\gamma}(x) , \dot{\gamma}(x+z) - \dot{\gamma}(x), \gamma(x+z) - \gamma(x) - z\dot{\gamma}(x) \rangle| \leq A |\dot{\gamma}(x+z) - \dot{\gamma}(x)| | \gamma(x+z) - \gamma(x) - z\dot{\gamma}(x) |
$$
also holds by definition of the triple product. Moreover, $w_{p}(x,z) = (2\pi/|z|)^{p}$ as $\gamma$ has constant speed. Finally $|\gamma(x+z) - \gamma(x)| \geq \lenG|z|/(2\pi \delta_{\infty}(\gamma) )$ by definition of the distortion. Thus
\begin{align*}
&\int_{\T \times \T} w_{p}(x,z) \frac{ | \langle \dot{\gamma}(x) , \dot{\gamma}(x+z) , \gamma(x+z) - \gamma(x) \rangle | }{|\gamma(x+z) - \gamma(x)|^{3}} \, \rd x \rd z \leq \\
&(2\pi)^{p}\delta^{3}_{\infty}(\gamma)\left( \frac{2\pi}{\lenG} \right)^{2}\int_{\T \times \T} \frac{|\dot{\gamma}(x+z) - \dot{\gamma}(x)| }{|z|^{1+p/2}}\frac{| \gamma(x+z) - \gamma(x) - z\dot{\gamma}(x) |}{|z|^{2+p/2}} \, \rd x \rd z.
\end{align*}
Now apply Cauchy-Schwarz to find
\begin{align*}
&\int_{\T \times \T} w_{p}(x,z) \frac{ | \langle \dot{\gamma}(x) , \dot{\gamma}(x+z) , \gamma(x+z) - \gamma(x) \rangle | }{|\gamma(x+z) - \gamma(x)|^{3}} \, \rd x \rd z \leq (2\pi)^{p}\delta^{3}_{\infty}(\gamma)\left( \frac{2\pi}{\lenG}\right)^2\mathcal{I}^{1/2}_{1} \mathcal{I}^{1/2}_{2},\\
\mathcal{I}_{1} := &\int_{\T \times \T} \frac{|\dot{\gamma}(x+z) - \dot{\gamma}(x)|^{2} }{|z|^{2+p}} \, \rd x \rd z \qquad \mathcal{I}_{2} := \int_{\T \times \T} \frac{|\gamma(x+z) - \gamma(x) - z \dot{\gamma}(x)|^{2} }{|z|^{4+p}} \, \rd x \rd z.
\end{align*}
But $\dot{\gamma} \in L^{2}(\T)$ and Parseval imply
$$
\frac1{2\pi} \mathcal{I}_{1} = \sum_{k \in \Z} |k|^{2}|\hat{\gamma}_{k}|^2 \int_{\T} \frac{ 2(1-\cos(kz)) }{|z|^{2+p} } \, \rd z = \sum_{k \in \Z} 4 |k|^{3+p}|\hat{\gamma}_{k}|^2 \int^{\pi |k|}_{0} \frac{(1-\cos(u)) }{u^{2+p} } \, \rd u,
$$
where the last equality follows from a simple change of variables. A similar argument shows that
$$
\frac1{2\pi} \mathcal{I}_{2} = \sum_{k \in \Z} 2 |k|^{3+p}|\hat{\gamma}_{k}|^2 \int^{\pi |k|}_{0} \frac{(1-\cos(u))^{2} + (\sin u - u)^2 }{u^{4+p} } \, \rd u
$$
as well. Collecting constants and extending the region of integration then yields the claim.
\end{proof}

\begin{lemma}{ {\rm (Lemma \ref{lem:lapest}) } }
\label{lem:lapestA}
If $\gamma \in H^{2}(\T)$ and $0 \leq p < \frac1{2}$ then the linear operator
\begin{equation}\label{eq:linpv}
L_{p}[\gamma] := \pv \, \int_{\T} \frac{\gamma(x)-\gamma(y)}{\dd^{2(p+1)}(x-y)} \, \rd y
\end{equation}
defines an $H^{1-2p}(\T)$ function. As a Fourier multiplier it acts as a (negative) fractional Laplacian
$$
\hat{\gamma}_k \to \lambda_{k} \hat{\gamma}_k \qquad \lambda_{k} = C_{p,k} |k|^{1+2p} \qquad 0 < C_{k,p} = O(1) \quad {\color{black}\text{as} \quad |k|\rightarrow \infty,}
$$
and in particular the operator-norm estimate
$$
\|L_{p}[\gamma]\|_{ \dot{H}^{1-2p}(\T)} \leq C_{p} \|\gamma\|_{\dot{H}^{2}(\T)}
$$
holds.
\end{lemma}
\begin{proof}
Consider first the linear operator $L_{\epsilon,p}[\gamma]$ defined, for $\epsilon >0,$ by the expression
$$
\int_{\T \cap \{\dd(x-y) > \epsilon\} } \frac{\gamma(x)-\gamma(y)}{\dd^{2(p+1)}(x-y)} \, \rd y =  \int^{\pi}_{-\pi} \frac{\gamma(x)-\gamma(x+z)}{ z^{2(p+1)} } \mathbf{1}_{ \{|z| > \epsilon\} }(z) \, \rd z.
$$
The last equality follows from periodicity of $\dd(x-y)$ and a change of variables. Now expand $\gamma(x)$ in a Fourier series to conclude
$$
\int^{\pi}_{-\pi} \frac{\gamma(x)-\gamma(x+z)}{ z^{2(p+1)} } \mathbf{1}_{ \{ |z| > \epsilon\} }(z) \, \rd z = \sum_{k \in \Z} \hat{\gamma}_k \lambda_{k}(\epsilon) \mathrm{e}^{ikx}, \qquad \lambda_{k}(\epsilon) := 2 \int^{\pi}_{\epsilon} \frac{1 - \cos(kz)}{ z^{2(1+p)} } \, \rd z,
$$
with equality holding in the $L^{2}(\T)$ sense. For $|k| \geq 1,$ a change of variables then gives
\begin{align*}
\lambda_{k}(\epsilon) = 2 \int^{\pi}_{\epsilon} \frac{1 - \cos(kz)}{ z^{2(1+p)} } \, \rd z &= 2 |k|^{1+2p} \int^{k\pi}_{k\epsilon} \frac{1 - \cos(u) }{ u^{2(1+p)} } \, \rd u \\
&= 2|k|^{1+2p} \left( \int^{\infty}_{k \epsilon } \frac{1-\cos(u)}{u^{2(1+p)} } \, \rd u + |k|^{-(1+2p)} O(1) \right),
\end{align*}
where the $O(1)$ error holds uniformly in $k$ and in $\epsilon$. As $0 \leq p < \frac1{2}$ and $\gamma \in H^{2}(\T),$ the desired estimate
$$
\sum_{k \in \Z} |k|^{2(1-2p)} |\hat{\gamma}_k|^{2} \lambda^{2}_{k}(0) \leq C^{2}_{p} \sum_{k \in \Z} |k|^{4} |\hat{\gamma}_k|^{2} \leq C^{2}_{p} \|\gamma\|^{2}_{\dot{H}^{2}(\T)}
$$
then follows.
\end{proof}

\begin{lemma}{ {\rm (Lemma \ref{lem:boundF})} }
\label{lem:boundFA}
Assume that $K(u,v)$ is $h/p$ homogeneous and either $0$-degenerate or $m$-degenerate. Assume that $\gamma \in H^{2}(\T)$ is a unit speed, bi-Lipschitz embedding. Then for $0 < \epsilon < 1$ the function
$$
f_{ \gamma ,\epsilon}(x) :=  \int_{ \T \cap \{ \dd(x-y) \geq \epsilon\} } K_{u}\left( |\gamma(x) - \gamma(y)|^2, \dd^{2}(x-y) \right)(\gamma(x) - \gamma(y) ) \, \rd y
$$
has mean zero. Moreover, $f_{ \gamma ,\epsilon}$ obeys a uniform $L^{2}(\T)$ bound
$$
\| f_{ \gamma ,\epsilon} \|_{L^{2}(\T)} \leq C(\delta_{\infty}(\gamma),\|\ddot{\gamma}\|_{L^{2}(\T)})
$$
for $C(\delta_{\infty}(\gamma),\|\ddot{\gamma}\|_{L^{2}(\T)})$ an $\epsilon$-independent constant.
\end{lemma}
\begin{proof}
That $f_{\gamma,\epsilon}$ has mean zero follows from {\color{black} the antisymetry of the integrand}. As $\gamma \in H^{2}(\T)$ defines a unit speed, bi-Lipschitz embedding, there exists a $\delta(x,y)$ so that
$$
|\gamma(x)-\gamma(y)|^{2} = ( 1 - \delta(x,y) )\dd^{2} (x-y) \qquad 0 \leq \delta(x,y) \leq 1 - 1/\delta^{2}_{\infty}(\gamma)
$$
for $\delta_{\infty}(\gamma)$ the distortion of the curve. As $\gamma$ has unit speed, $\delta(x,y)$ also satisfies
$$
\delta(x,y) \leq \kappa^{*}_{\gamma}(x) \dd(y-x) \qquad \text{and} \qquad \delta(x,y) \leq \sqrt{2\pi} \| \ddot{\gamma} \|_{L^{2}(\T)} \dd^{ \frac1{2} }(y-x)
$$
for $ \kappa^{*}_{\gamma}(x)$ the curvature maximal function (c.f. {\color{black}\eqref{eq:deltabound}}). Define $\alpha(x,y) := (1-\delta(x,y))^{-1},$ which obeys the bounds $1 \leq \alpha(x,y) \leq \delta^{2}_{\infty}(\gamma)$. Then by the $h/p$-homogeneity hypothesis,
$$
K_{u}\left( |\gamma(x) - \gamma(y)|^{2}, \dd^{2}(x-y) \right) = g\big( \alpha(x,y) \big) \dd^{-2(p+1) }(x-y) \qquad g(\alpha) = -\alpha^{2}h^{\prime}(\alpha).
$$
The estimate then follows from this expression by appealing to one of the degeneracy assumptions.

First consider the $0$-Degenerate case. Then $p < \frac1{2}$ by assumption, and so for any $(x,y)$ there exists $1 \leq \xi(x,y) \leq \alpha(x,y)$ so that
\begin{align*}
g\big( \alpha(x,y) \big) - g(1) = g^{\prime}\big( \xi(x,y) \big) \frac{ \delta(x,y) }{1 - \delta(x,y)}\quad \longrightarrow\quad\left| g\big( \alpha(x,y) \big) - g(1) \right| \leq C(\delta_{\infty}(\gamma))\delta(x,y)
\end{align*}
by the mean value theorem. Combining these inequalities then gives
\begin{align*}
&f_{\gamma,\epsilon}(x) = g(1)L_{p,\epsilon}[\gamma] + \tilde{f}_{\gamma,\epsilon}(x),\\
&\tilde{f}_{\gamma,\epsilon}(x) := \int_{ \T \cap \{ \dd(x-y) \geq \epsilon\} } \left[ g \big( \alpha(x,y) \big) - g(1)  \right] \frac{ \gamma(x) - \gamma(y) }{\dd^{2(p+1)}(x-y)} \, \rd y \\
&|\tilde{f}_{\gamma,\epsilon}(x)| \leq C(\delta_{\infty}(\gamma))\kappa^{*}_{\gamma}(x) \int_{ \T \cap \{ \dd(x-y) \geq \epsilon\} } \frac{ \dd^2(y-x) }{\dd^{2(p+1)}(x-y)} \, \rd y = C(\delta_{\infty}(\gamma))\kappa^{*}_{\gamma}(x)\int^{\pi}_{\epsilon} z^{-2p} \, \rd z.
\end{align*}
Thus $|\tilde{f}_{\gamma,\epsilon}(x)| \leq C(\delta_{\infty}(\gamma))\kappa^{*}_{\gamma}(x),$ and the latter lies in $L^{2}(\T)$ by the $L^{2}(\T)$ boundedness of the curvature maximal function. Combining this with the previous lemma for $L_{p,\epsilon}[\gamma]$ yields the claim in this case.

Now suppose $K(u,v)$ is degenerate with $m>0,$ and so there exists $1 \leq \xi(x,y) \leq \alpha(x,y)$ with
$$
g\big( \alpha(x,y) \big) = \frac1{(m+1)!} \frac{ \delta^{m+1}(x,y)}{ (1-\delta(x,y))^{m+1} } g^{ (m+1) }\big( \xi(x,y) \big)
$$
due to the degeneracy assumption. Thus
\begin{align*}
&f_{\gamma,\epsilon}(x) = \int_{ \T \cap \{ \dd(x-y) \geq \epsilon\} } K_{u}\left( |\gamma(x) - \gamma(y)|^2, \dd^{2}(x-y) \right)(\gamma(x) - \gamma(y) ) \, \rd y = \\
&\int_{ \T \cap \{ \dd(x-y) \geq \epsilon\} } \frac1{(m+1)!} \frac{ \delta^{m+1}(x,y)}{ (1-\delta(x,y))^{m+1} } g^{ (m+1) }\big( \xi(x,y) \big) \dd^{-2(p+1)}(x-y)(\gamma(x)-\gamma(y))\, \rd y,\\
&|f_{\gamma,\epsilon}(x)| \leq C(\delta_{\infty}(\gamma))\|\ddot{\gamma}\|^{m}_{L^{2}(\T)}\kappa^{*}_{\gamma}(x) \int^{\pi}_{\epsilon} z^{m/2-2p} \, \rd z \leq C(\delta_{\infty}(\gamma),\|\ddot{\gamma}\|_{L^{2}(\T)})\kappa^{*}_{\gamma}(x).
\end{align*}
The last inequality follows from the fact that $m/2 - 2p > -1$ by assumption. Thus $f_{\gamma,\epsilon}(x)$ is bounded in $L^{2}(\T),$ uniformly with respect to $\epsilon$ as claimed.
\end{proof}

\begin{lemma}{ {\rm (Lemma \ref{lem:lipF}) } }
\label{lem:lipFA}
Assume that $K(u,v)$ is $h/p$ homogeneous and either $0$-degenerate or $m$-degenerate. Let $\gamma,\phi \in H^{2}(\T)$ denote unit speed, bi-Lipschitz embeddings. Then for $0 < \epsilon < 1$ the uniform Lipschitz estimates
$$
\| f_{\gamma,\epsilon} - f_{\phi,\epsilon} \|_{L^{2}(\T)} \leq C(\delta_{\infty}(\gamma),\delta_{\infty}(\phi),\|\gamma\|_{H^{2}(\T)},\|\phi\|_{H^{2}(\T)} )\|\gamma - \phi\|_{H^{2}(\T)}
$$
hold, where $C(\delta_{\infty}(\gamma),\delta_{\infty}(\phi),\|\gamma\|_{H^{2}(\T)},\|\phi\|_{H^{2}(\T)} )$ denotes an $\epsilon$-independent constant.
\end{lemma}
\begin{proof}
Take $\delta_{\gamma}(x,y)$ and $\delta_{\phi}(x,y)$ so that
$$
|\gamma(x) - \gamma(y)|^{2} = ( 1 - \delta_{\gamma}(x,y) )\dd^{2}(x-y) \qquad |\phi(x) - \phi(y)|^{2} = ( 1 - \delta_{\phi}(x,y) )\dd^{2}(x-y)
$$
as in the proof of the previous lemma. Define $\alpha_{\gamma}(x,y)$ and $\alpha_{\phi}(x,y)$ similarly.

Assume first that $0$-Degeneracy holds. By {defining} $\Psi^{+}(x) := \gamma(x) + \phi(x)$ and $\Psi^{-}(x) := \gamma(x) - \phi(x),$ a trivial calculation shows that
\begin{align*}
f_{\gamma,\epsilon} - f_{\phi,\epsilon} &= g(1)L_{p,\epsilon}[\Psi^{-}] + \int_{ \dd(x-y) > \epsilon } \Big( g(\alpha_{\gamma}(x,y) ) - g(1) \Big)\frac{ \Psi^{-}(x) - \Psi^{-}(y) }{\dd^{2(p+1)}(x-y) } \, \rd y  \\
&+ \int_{ \dd(x-y) > \epsilon } \Big( g( \alpha_{\gamma}(x,y) ) - g( \alpha_{\phi}(x,y) ) \Big) \frac{ \phi(x) - \phi(y) }{ \dd^{2(p+1) }(x-y) } \, \rd y := \mathrm{I} + \mathrm{II} + \mathrm{III}.
\end{align*}
The fact that $\| \mathrm{I} \|_{L^{2}(\T)} \leq C\| \gamma - \phi\|_{H^{2}(\T)}$ follows immediately from lemma \ref{lem:lapest}. As for $\mathrm{II},$ arguing as in the previous lemma shows
\begin{align*}
|\mathrm{II}| &\leq C(\delta_{\infty}(\gamma) ) \int_{ \dd(x-y) > \epsilon } \delta_{\gamma}(x,y) \frac{ |{\color{black}\Psi^{-}(x) - \Psi^{-}(y)}| }{\dd^{2(p+1)}(x-y) } \, \rd y \\
&\leq C(\delta_{\infty}(\gamma) )\| \dot{\Psi}^{-} \|_{L^{\infty}(\T)}\kappa^{*}_{\gamma}(x) \int_{ \dd(x-y) > \epsilon } \dd^{-2p}(x-y)  \, \rd y,
\end{align*}
and so $\mathrm{II}$ obeys
$$
\| \mathrm{II} \|_{L^{2}(\T)} \leq C(\delta_{\infty}(\gamma) , \|\ddot{\gamma}\|_{L^{2}(\T)} )\| \gamma - \phi\|_{H^{2}(\T)}
$$
as desired. Finally, the mean value theorem once again furnishes a $1 \leq \xi(x,y) \leq \max\{ \delta^{2}_{\infty}(\gamma), \delta^{2}_{\infty}(\phi)\}$ so that
\begin{align*}
\mathrm{III} = \int_{ \dd(x-y) > \epsilon } g^{\prime}(\xi(x,y)) \Big( \alpha_{\gamma}(x,y) - \alpha_{\phi}(x,y) \Big)\frac{ \phi(x) - \phi(y) }{\dd^{2(p+1)}(x-y) } \, \rd y.
\end{align*}
Noting that
\begin{equation}\label{eq:alphadif}
\alpha_{\gamma}(x,y) - \alpha_{\phi}(x,y) = \frac{ \delta_{\gamma}(x,y) - \delta_{\phi}(x,y) }{ (1-\delta_{\gamma}(x,y))(1-\delta_{\phi}(x,y)) }
\end{equation}
then yields the estimate
$$
|\mathrm{III}| \leq C(\delta_{\infty}(\gamma),\delta_{\infty}(\phi) )\|\dot{\phi}\|_{L^{\infty}(\T)} \int_{ \dd(x-y) > \epsilon }  |\delta_{\gamma}(x,y) - \delta_{\phi}(x,y)| \dd^{-2p-1}(x-y) \, \rd y
$$
for the third term. From the definition of $\delta_{\gamma}(x,y)$ and $\delta_{\phi}(x,y),$ it follows that
\begin{equation}\label{eq:deltadif}
|\delta_{\gamma}(x,y) - \delta_{\phi}(x,y)| = \frac{ | \langle \Psi^{-}(x) - \Psi^{-}(y) , \Psi^{+}(x) - \Psi^{+}(y) \rangle|}{\dd^{2}(x-y)}.
\end{equation}
That $\gamma$ and $\phi$ have unit speed then gives
$$
|\langle \Psi^{-}(x) - \Psi^{-}(y) , \Psi^{+}(x) - \Psi^{+}(y) \rangle| \leq \dd^{3}(x-y) \Big( \| \dot{\Psi}^{-} \|_{L^{\infty}(\T)} ( \ddot{\Psi}^{+} )^{*}(x) + \| \dot{\Psi}^{+} \|_{L^{\infty}(\T)} (\ddot{\Psi}^{-})^{*}(x)\Big),
$$
which yields the desired estimate
\begin{align*}
|\mathrm{III}| &\leq C(\delta_{\infty}(\gamma),\delta_{\infty}(\phi) )\|\dot{\phi}\|_{L^{\infty}(\T)}\Big( \| \dot{\Psi}^{-} \|_{L^{\infty}(\T)} ( \ddot{\Psi}^{+} )^{*}(x) + \| \dot{\Psi}^{+} \|_{L^{\infty}(\T)} (\ddot{\Psi}^{-})^{*}(x)\Big)\\
\|\mathrm{III}\|_{L^{2}(\T)} &\leq C(\delta_{\infty}(\gamma),\delta_{\infty}(\phi),\|\gamma\|_{H^{2}(\T)},\|\phi\|_{H^{2}(\T)} )\|\gamma - \phi\|_{H^{2}(\T)}
\end{align*}
for the third term. Combining these estimates yields the claim
$$
\| f_{\gamma,\epsilon} - f_{\phi,\epsilon} \|_{L^{2}(\T)} \leq C(\delta_{\infty}(\gamma),\delta_{\infty}(\phi),\|\gamma\|_{H^{2}(\T)},\|\phi\|_{H^{2}(\T)} )\|\gamma - \phi\|_{H^{2}(\T)}
$$
in the $0$-Degenerate case.

Assume now that $m$-Degeneracy holds. Then
\begin{align*}
f_{\gamma,\epsilon} - f_{\phi,\epsilon} &= \int_{ \dd(x-y) > \epsilon } g( \alpha_{\gamma}(x,y) ) \frac{ \Psi^{-}(x) - \Psi^{-}(y)}{\dd^{2(p+1)}(x-y)} \, \rd y \\
&+ \int_{ \dd(x-y) > \epsilon } \big( g( \alpha_{\gamma}(x,y) ) - g( \alpha_{\phi}(x,y) ) \frac{ \phi(x) - \phi(y)}{\dd^{2(p+1)}(x-y)} \, \rd y := \mathrm{I} + \mathrm{II}
\end{align*}
as before. That $\mathrm{I}$ obeys the claimed bound then follows by using the argument of the previous lemma. Now decompose $\mathrm{II}$ into two pieces,
$$
\mathrm{II} = \int_{ \{ \dd(x-y) > \epsilon\} \cap \{ \alpha_{\gamma} \geq \alpha_\phi \} } + \int_{ \{ \dd(x-y) > \epsilon\} \cap \{ \alpha_{\gamma} < \alpha_\phi \} } := \mathrm{II}_{\mathrm{A}} + \mathrm{II}_{ {\rm B}}.
$$
For $\mathrm{II}_{ {\rm A} },$ there exists a $1 \leq \alpha_{\phi}(x,y) \leq \xi(x,y) \leq \alpha_{\gamma}(x,y)$ so that
$$
\mathrm{II}_{ {\rm A} } = \int_{ \{ \dd(x-y) > \epsilon\} \cap \{ \alpha_{\gamma} \geq \alpha_\phi \} }  g^{\prime}\big( \xi(x,y) \big) \Big( \alpha_{\gamma}(x,y) - \alpha_{\phi}(x,y) \Big) \frac{ \phi(x) - \phi(y)}{\dd^{2(p+1)}(x-y)} \, \rd y.
$$
Furthermore, there exists a $1 \leq \eta(x,y) \leq \xi(x,y) \leq \alpha_{\gamma}(x,y)$ so that
\begin{align*}
g^{\prime}\big( \xi(x,y) \big) &= g^{ (m+1) } \big( \eta(x,y) \big) \frac{ ( \xi(x,y) - 1)^{m} }{m!}  \\ \left|g^{\prime}\big( \xi(x,y) \big)\right| &\leq C(\delta_{\infty}(\gamma))( \alpha_{\gamma}(x,y) - 1)^{m} \leq C(\delta_{\infty}(\gamma))\delta_{\gamma}^{m}(x,y).
\end{align*}
As a consequence, this bound furnishes the inequality
$$
|\mathrm{II}_{ {\rm A} }| \leq C( \delta_{\infty}(\gamma) ) \| \dot{\phi} \|_{L^{\infty}(\T)} \int_{  \dd(x-y) > \epsilon } \delta^{m}_{\gamma}(x,y) | \alpha_{\gamma}(x,y) - \alpha_{\phi}(x,y) | \dd^{-(2p+1)}(x-y) \, \rd y.
$$
Recalling {\color{black}\eqref{eq:deltabound} and the bounds \eqref{eq:alphadif} and \eqref{eq:deltadif}}
\begin{align*}
\delta^{m}_{\gamma}(x,y) &\leq C( \|\ddot{\gamma}\|_{L^{2}(\T)} )\dd^{ \frac{m}{2} }(x-y) \\ | \alpha_{\gamma}(x,y) - \alpha_{\phi}(x,y) | &\leq C(\delta_{\infty}(\gamma),\delta_{\infty}(\phi))\dd(y-x)
\Big( \| \dot{\Psi}^{-} \|_{L^{\infty}(\T)} ( \ddot{\Psi}^{+} )^{*}(x) + \| \dot{\Psi}^{+} \|_{L^{\infty}(\T)} (\ddot{\Psi}^{-})^{*}(x)\Big)
\end{align*}
and arguing as before then completes the estimate
$$
\|\mathrm{II}_{ {\rm A} } \|_{L^{2}(\T)} \leq C(\delta_{\infty}(\gamma),\delta_{\infty}(\phi),\|\gamma\|_{H^{2}(\T)},\|\phi\|_{H^{2}(\T)} )\|\gamma - \phi\|_{H^{2}(\T)}
$$
for this term. The estimate for $\|\mathrm{II}_{ {\rm B} }\|_{L^{2}(\T)}$ then follows by the same argument with $\phi$ in place of $\gamma$, so the claimed bound holds in the $m$-Degenerate case as well.
\end{proof}

\begin{lemma}{ {\rm (Lemma \ref{lem:boundA}) } }
\label{lem:boundAA}
Suppose that $\tang,\sang \in H^{1}(\T)$ and that
$$
\min_{x \in \T} \, |\tang(x)| \geq \frac1{C},\quad \min_{x \in \T} \, |\sang(x)| \geq \frac1{C}
$$
for $0 < C < \infty$ a positive constant. If $\tang$ is non-constant then
$$
A_{\tang} := \mathrm{Id} - \fint_{\T} \frac{ \tang \otimes \tang }{|\tang|^2} \, \rd x
$$
is non-singular. The operator norm estimate
$$
\|A_{\tang} - A_{\sang}\|_{ {\rm op} } \leq 2 C^2\|\tang + \sang\|_{L^{2}(\T)}\| \tang - \sang\|_{L^{2}(\T)}
$$
also holds, and so $A^{-1}_{\sang}$ exists and obeys
\begin{align*}
\|A^{-1}_{\sang}\|_{ {\rm op} } &\leq \frac{ \| A^{-1}_{\tang} \|_{\rm op} }{1- \| A^{-1}_{\tang} \|_{\rm op} \|A_{\tang} - A_{\sang}\|_{ {\rm op} } }\\
\| A^{-1}_{\sang} - A^{-1}_{\tang} \|_{ {\rm op} } &\leq 2 C^{2}\| A^{-1}_{{\color{black}\sang}} \|_{ {\rm op} }\| A^{-1}_{\tang} \|_{ {\rm op} } \|\tang + \sang\|_{L^{2}(\T)}\| \tang - \sang\|_{L^{2}(\T)}
\end{align*}
whenever $\|\tang - \sang\|_{L^2(\T)}$ is sufficiently small.
\end{lemma}
\begin{proof}
Select $\vv \in \R^3$ with unit length so that the quadratic form $\langle \vv , A_{\tang} \vv \rangle$ realizes the minimal eigenvalue. A simple computation reveals
$$
 \lambda_{ \min }( A_{\tang} ) = \langle \vv , A_{\tang} \vv \rangle = 1 - \fint_{\T} \left< \vv , \frac{ \tang}{|\tang|} \right> ^{2} \, \rd x \geq 0
$$
by Cauchy-Schwarz. As $\tang(x)/|\tang(x)|$ varies continuously on $\T,$ equality holds if and only if $\tang(x)$ and $\vv$ are everywhere collinear. But $\tang(x)$ is non-constant, so equality cannot occur. Thus $\lambda_{ \min }( A_{\tang} ) > 0$ and so $A^{-1}_{\tang}$ exists.

Now take $\vv \in \R^3$ with unit length but otherwise arbitrary. Then
\begin{align*}
\langle \vv, (A_{\tang} - A_{\sang}) \vv \rangle &= \fint_{\T} \left( \frac{\langle \vv, \tang \rangle^2}{|\tang|^2} - \frac{\langle \vv, \sang \rangle^2 }{|\sang|^2} \right) \, \rd x \\
&= \fint_{\T} \langle \vv, \tang \rangle^2 \left( \frac1{ |\tang|^2 } - \frac1{|\sang|^2} \right) \, \rd x + \fint_{\T} \frac1{ |\sang|^2 }\big( \langle \vv, \tang \rangle^2 - \langle \vv, \sang \rangle^2 \big) \, \rd x\\
& \leq 2 C^2 \|\tang + \sang\|_{L^{2}(\T)}\|\tang - \sang\|_{L^{2}(\T)},
\end{align*}
and so the claimed operator norm estimate
$$\|A_{\tang} - A_{\sang}\|_{ {\rm op} } \leq 2C^2\|\tang + \sang\|_{L^{2}(\T)}\| \tang - \sang\|_{L^{2}(\T)}$$
follows. The simple eigenvalue inequality $|\lambda_{\min}(A_{\sang}) - \lambda_{\min}(A_{\tang})| \leq \| A_{\tang} - A_{\sang}\|_{ {\rm op}}$ yields invertibility of $A_{\sang}$ for $\| A_{\tang} - A_{\sang}\|_{ {\rm op}}$ sufficiently small, while the trivial identity $A^{-1}_{\sang} - A^{-1}_{\tang} = A^{-1}_{\sang}( A_{\tang} - A_{\sang} ) A^{-1}_{\tang}$ yields the final estimate.
\end{proof}

\begin{proposition}{ {\rm (Proposition \ref{prop:linear}) } }
\label{prop:linearA}
Suppose $f \in C([0,T];L^{1}(\T))$ for $T>0$ arbitrary. Then for any initial datum $u_0 \in H^{1}(\T)$ the linear initial value problem
\begin{equation}\label{eq:ez_lin}
u_{t} =  u_{xx} + f, \qquad u(x,0) = u_0(x)
\end{equation}
has a unique mild solution $u \in C([0,T]; H^{1}(\T) )$. For any $3/2 < \alpha < 2,$ the solution obeys the estimate
$$
\|u\|_{C([0,T];\dot{H}^{1}(\T))} \leq C(\alpha)\left( \|u_0\|_{ \dot{H}^1(\T) } + T^{1 - \frac{\alpha}{2}}\|f\|_{C([0,T];L^1(\T))}  \right)
$$
for $C(\alpha)$ a universal constant.
\end{proposition}

{\color{black}
\begin{proof}
For each $k \in \Z$ define the coefficients of the solution as
\begin{equation}\label{eq:wkcoeff}
\hat{u}_{k}(t) = \mathrm{e}^{-k^{2}t}\hat{u}_{k}(0) + \int^{t}_{0} \mathrm{e}^{-k^2(t-s)}\hat{f}_{k}(s) \, \rd s := \hat{v}_{k}(t) + \hat{z}_{k}(t),
\end{equation}
so that for $3/2 < \alpha < 2$ fixed the fact that $|\hat{f}_{k}(s)| \leq \|f\|_{C([0,T];L^{1}(\T))}$ gives
\begin{align*}
|\hat{z}_{k}(t)||k| &\leq \|f\|_{C([0,T];L^{1}(\T))}|k|^{(1-\alpha)} \int^{t}_{0} |k|^{\alpha}\mathrm{e}^{-k^{2}(t-s)}\, \rd s \leq C(\alpha)\|f\|_{C([0,T];L^{1}(\T))}|k|^{(1-\alpha)} \int^{t}_{0} (t-s)^{-\frac{\alpha}{2}} \, \rd s\\
& \leq C(\alpha)\|f\|_{C([0,T];L^{1}(\T))}|k|^{(1-\alpha)} t^{1 - \frac{\alpha}{2}}.
\end{align*}
The series $\{ |\hat{z}_{k}(t)||k| \}_{k \in \Z}$ therefore lies in $\ell^{2}(\Z)$, and so for $t \in [0,T]$ both functions
$$
z(x,t) := \sum_{k \in \Z} \hat{z}_{k}(t) \mathrm{e}^{ikx} \qquad \text{and} \qquad v(x,t) := \sum_{k \in \Z} \hat{v}_{k}(t) \mathrm{e}^{ikx}
$$
define $H^{1}(\T)$ functions with estimates
$$
\|z\|_{C([0,T];\dot{H}^{1}(\T))} \leq C(\alpha)\|f\|_{C([0,T];L^{1}(\T))}T^{1 - \frac{\alpha}{2}} \qquad \|v\|_{C([0,T];\dot{H}^{1}(\T))} \leq \|u_0\|_{ \dot{H}^{1}(\T)}.
$$
The claimed bound for $\|u\|_{C([0,T]; \dot{H}^{1}(\T))}$ therefore follows from the triangle inequality. The claim that $u \in C([0,T];H^{1}(\T))$ follows in a similar way, i.e. by first shifting the temporal origin from $t=0$ to $t = t_1$ and then applying the argument estimating $z(x,t)$ to the difference $u(t_2) - S(t_2-t_1)[u(t_1)]$ on the interval $[t_1,t_2]$. Finally, the claim that
$$
v(x,t) = \int_{\T} K(x-y,t) u_{0}(y) \, \rd y \qquad \text{and} \qquad z(x,t) = \int^{t}_{0} \int_{\T} K(x-y,t-s)f(y,s)\, \rd y \rd s
$$
for $t>0$ is standard. For $t>0$ simply define
$$
K(x-y,t) := \frac1{2\pi} \sum_{k \in \Z} \re^{-k^2 t} \re^{ik(x-y)} 
$$
and apply the Fubini-Tonelli theorem. Thus $u$ with coefficients defined by \eqref{eq:wkcoeff} yields the unique mild solution.
\end{proof}
}

\begin{lemma}{ {\rm (Lemma \ref{lem:diffF}) } }
\label{lem:diffFA}
Assume that $K(u,v)$ is $h/p$ homogeneous and degenerate. Assume that $\gamma \in H^{3}(\T)$ is a unit speed, bi-Lipschitz embedding. Then the nonlocal integral
$$
f_{ \gamma}(x) :=  \lim_{\epsilon \downarrow 0} \; \int_{ \T \cap \{\dd(x-y) \geq \epsilon\} } K_{u}\left( |\gamma(x) - \gamma(y)|^2, \dd^{2}(x-y) \right)(\gamma(x) - \gamma(y) ) \, \rd y
$$
lies in $H^{1}(\T)$ and obeys the estimate
$$
\| f^{\prime}_{\gamma} \|_{L^{2}(\T)} \leq C( \delta_{\infty}(\gamma) , \| \gamma\|_{\dot{H}^{3}(\T)} )
$$
for $C(a,b)$ a continuous function of its arguments.
\end{lemma}

\begin{proof}
Assume first that the claim holds for all $\gamma$ that are bi-Lipschitz, unit speed and smooth. Take a sequence $\gamma_n \to \gamma$ of smooth functions with unit speed converging to $\gamma$ in the $H^{3}(\T)$ norm. In particular, the convergence
$$
\delta_{\infty}( \gamma_ n ) \to \delta_{\infty}(\gamma)
$$
holds as well, as does the limit $f_{\gamma_n} \to f_{\gamma}$ in $L^{2}(\T)$ by lemma \ref{lem:lipF}. The bound
$$
\| f^{\prime}_{\gamma_n} \|_{L^{2}(\T)} \leq C( \delta_{\infty}(\gamma_n) , \| \gamma_n\|_{\dot{H}^{3}(\T)} )
$$
guarantees that $f^{\prime}_{\gamma_n} \to \tilde{f}$ weakly for some $ \tilde{f} \in L^{2}(\T)$ after passing to a subsequence if necessary. Then for $\psi$ any smooth, $2\pi$-periodic function the equality
$$
\int_{\T} \langle f^{\prime}_{\gamma_n} , \psi \rangle \, \rd x = -\int_{\T} \langle f_{\gamma_n} , \psi^{\prime} \rangle \, \rd x
$$
holds for all $n,$ which upon passing to the limit using weak convergence shows that
$$
\int_{\T} \langle \tilde{f} , \psi \rangle \, \rd x = -\int_{\T} \langle f_{\gamma} , \psi^{\prime} \rangle \, \rd x
$$
for $\psi$ any smooth function. Thus $\tilde{f}$ provides a weak derivative {\color{black}for $f_{\gamma}$}. But then
$$
\|\tilde{f}\|_{L^{2}(\T)} \leq \liminf_{ n \to \infty } \| f^{\prime}_{\gamma_n} \|_{L^{2}(\T)} \leq \liminf_{n \to \infty} C( \delta_{\infty}(\gamma_n) , \| \gamma_n\|_{\dot{H}^{3}(\T)} ) = C( \delta_{\infty}(\gamma) , \| \gamma\|_{\dot{H}^{3}(\T)} ),
$$
and so the claim holds in the general case.

It therefore suffices to prove the claim for smooth functions. Given $x,y \in \T$ define $\delta(x,y)$ according to
$$
|\gamma(x) - \gamma(y)|^2 = ( 1 - \delta(x,y) )\dd^2(x-y) \qquad 0 \leq \delta(x,y) \leq 1 - 1/\delta^2_{\infty}(\gamma)
$$
and $\alpha(x,z) := (1-\delta(x,z))^{-1}$ as well. A simple argument using Taylor's theorem shows that
$$
\delta(x,y) \leq \| \ddot{\gamma}\|_{L^{\infty}(\T)}\dd(x-y).
$$
In the $0$-degenerate case,
$$
f_{\gamma}(x) = g(1) L_{p}[\gamma](x) + \int_{\T} \left(  K_{u}\left( |\gamma(x) - \gamma(x+z)|^2, z^2 \right) - g(1){\color{black}|z|^{-2(p+1)}} \right)(\gamma(x) - \gamma(x+z) ) \, \rd z := \mathrm{I} + \mathrm{II}.
$$
{\color{black}The proof of lemma} \ref{lem:lapest} shows that
$$
\| L_{p}[\gamma] \|^{2}_{\dot{H}^{2-2p}(\T)} = \sum_{k \in \Z} |k|^{4-4p}|\hat{\gamma}_k|^2 |\lambda_{k}|^2 \leq C_{p}\sum_{k \in \Z} |k|^{6}|\hat{\gamma}_k|^2 = C_{p} \| \gamma \|^{2}_{  \dot{H}^{3}(\T)},
$$
but $p < 1/2$ and so the first term lies in $H^{1}(\T)$ as desired. For the second term, first define $\eta(x,z) := |\gamma(x) - \gamma(x+z)|^2$ and then let
$$
\xi(x,z) := \left(  K_{u}(\eta(x,z),z^2) - g(1)z^{-2(p+1)} \right)(\gamma(x) - \gamma(x+z) )
$$
denote its integrand. Note that $\langle \dot{\gamma}(x), \ddot{\gamma}(x) \rangle = 0$ implies
\begin{align}\label{eq:etabd}
\frac1{2} \eta_{x}(x,z) &= \langle \gamma(x+z) - \gamma(x) , \dot{\gamma}(x+z) - \dot{\gamma}(x) \rangle \nonumber \\
|\eta_{x}(x,z)| &\leq C |z|^3 \left( (\gamma^{\ppprime})^*(x) + \|\gamma^{\pprime}\|_{L^{\infty}(\T)}(1 + \|\gamma^{\ppprime}\|_{L^{1}(\T)})  \right)
\end{align}
for $C>0$ a universal constant. {\color{black} The fact that
\begin{align*}
\partial_{x} \xi(x,z) &= K_{uu}( \eta(x,z), z^2 ) \eta_{x}(x,z)(\gamma(x) - \gamma(x+z) )  + \left( g\big( \alpha(x,z) \big) - g(1) \right)( \dot{\gamma}(x) - \dot{\gamma}(x+z) )z^{-2(p+1)}
\end{align*}
combines with the homogeneity and $0$-degeneracy assumptions to show
\begin{align*}
\partial_{x} \xi(x,z) &= -\alpha^{2}(x,z)g^{\prime}\big( \alpha(x,z) \big)  \eta_{x}(x,z)(\gamma(x) - \gamma(x+z) )z^{-2(p+2)}\\
  &+ \left( g\big( \alpha(x,z) \big) - g(1) \right)( \dot{\gamma}(x) - \dot{\gamma}(x+z) )z^{-2(p+1)},
\end{align*}
and so we may apply the bound \eqref{eq:etabd} for $\eta_{x}(x,z)$ and the continuous embedding $L^{\infty}(\T) \subset \dot{H}^{1}(\T)$ to the first term and the argument of lemma \ref{lem:boundF} (with $\dot{\gamma}$ in place of $\gamma$) to the second term to conclude}
$$
|\partial_{x} \xi(x,z)| \leq C( \delta_{\infty}(\gamma), \|\gamma\|_{ \dot{H}^{3}(\T) } ) |z|^{-2p}( 1 + (\gamma^{\ppprime})^*(x)  ).
$$
As the maximal function $(\gamma^{\ppprime})^*(x)$ is uniformly bounded for smooth functions and the latter expression is integrable, the differentiation formula
$$
( f_{\gamma}(x) - L_{p}[\gamma](x) )^{\prime} = \int_{\T} \partial_{x}\xi(x,z) \, \rd z
$$
is therefore valid. The claimed bound
$$
\| ( f_{\gamma} - L_{p}[\gamma]  )^{\prime}\|^2_{L^{2}(\T)} \leq \int_{\T} \left( \int_{\T} |\partial_{x}\xi(x,z)| \, \rd z\right)^{2} \, \rd x \leq C( \delta_{\infty}(\gamma), \|\gamma\|_{ \dot{H}^{3}(\T) } )
$$
then follows by the $L^{2}(\T)$ boundedness of the maximal function. This completes the proof in the $0$-degenerate case.

The $m$-degenerate case follows in a similar fashion. In this instance the expression
$$
f_{\gamma}(x) = \int_{\T } K_{u}\left( |\gamma(x) - \gamma(x+z)|^2, z^2 \right)(\gamma(x) - \gamma(x+z) ) \, \rd z,
$$
is valid. Define $\eta(x,z)$ and $\xi(x,z)$ as before. Note that $g(\alpha)$ having a root of order $(m+1)$ at $\alpha = 1$ shows that $-\alpha^{2}g^{\prime}(\alpha)$ has a root of order $m$ at $\alpha = 1$ as well. Arguing as in lemma \ref{lem:boundF} yields the estimate
\begin{align*}
|\partial_{x}\xi(x,z)| &\leq C( \delta_{\infty}(\gamma) )\left(  \delta^{m}(x,z) |z||\eta_{x}(x,z)||z|^{-2(p+2)} + \delta^{m+1}(x,z)|z| \| \ddot{\gamma}\|_{L^{\infty}(\T)}|z|^{-2(p+1)}  \right) \\
& \leq C( \delta_{\infty}(\gamma), \|\gamma\|_{\dot{H}^{3}(\T)} )\left(  |z|^{m-(2p+3)} |\eta_{x}(x,z)| + |z|^{m-2p}\right)\\
& \leq C( \delta_{\infty}(\gamma), \|\gamma\|_{\dot{H}^{3}(\T)} )|z|^{m-2p}\left( 1 + (\gamma^{\ppprime})^*(x)  \right).
\end{align*}
Now recall that $m > 4p-2$ and that $m$ is a non-negative integer by hypothesis. Thus $m-2p > -1,$ and so the claimed bound then follows by arguing as in the $0$-degenerate case.
\end{proof}

\begin{corollary}{ {\rm (Corollary \ref{cor:continuation}) } }
\label{cor:continuationA}
Suppose $K(u,v)$ satisfies the hypotheses of lemma \ref{lem:boundF}, that $\tang_0 \in H^{1}(\T)$ has zero mean, unit speed and induces a bi-Lipschitz embedding. Let $T^{*}$ denote the maximal interval of existence of a mild solution to
\begin{equation*}
\tang_{t} =  \tang_{xx} + \mathbf{F}_{\gamma_{\tang}} + \lambda_{\tang} + \mu_{\tang} \tang, \qquad \tang(0) = \tang_{0} \in H^{1}(\T).
\end{equation*}
If $T^{*} < \infty$ then one of
$$
\limsup_{t \nearrow T^{*}} \, \|\tang(t)\|_{H^{1}(\T)} = + \infty \qquad \text{or} \qquad \limsup_{ t \nearrow \infty } \, \delta_{\infty}\left( \gamma_{ \tang(t)} \right) = +\infty
$$
must hold.
\end{corollary}
\begin{proof}
Let
$$
\mathcal{C}(\tang_0) := \left\{ T > 0 : \exists \tang \in C([0,T];H^{1}(\T)), \,\forall t \in [0,T] \;\; \tang(t) = S(t)[\tang_0] + \int^{t}_{0} S(t-s)G(\tang(t),\tang_{x}(s))\, \rd s \right\}.
$$
denote the existence times of mild solutions and $T^{*}$ its positive supremum. If $0 < T_{*} < \infty$ take a monotone sequence $T_{n} \in \mathcal{C}(\tang_0)$ of times realizing this supremum, in the sense that
$$
T_1 < T_2 < T_3 < \cdots \qquad T_{n} \nearrow T^{*}.
$$
Let $\tang^{n} \in C([0;T_n];H^{1}(\T))$ denote the corresponding mild solution. Note first that $\tang^{n}(t) = \tang^{n+1}(t)$ on $[0,T_{n}]$ by uniqueness of mild solutions, and so on $[0,T^{*})$ we may consider $\tang(t)$ as a single, well-defined and $H^{1}(\T)$-valued mapping. Assume that we have a-priori $H^{1}(\T)$ and distortion bounds for the solution, by which we mean that
$$
H_{*} := \limsup_{ n \to \infty } \, \| \tang(T_n) \|_{H^{1}(\T)} \qquad \text{and} \qquad \delta_{*} := \limsup_{ n \to \infty } \, \delta_{\infty}\left( \gamma_{ \tang(T_n)} \right)
$$
are finite. By standard compactness results we may extract a subsequence $\tang^{n_k}$ of the $\tang(T_n),$ and infer that there exists $\tang^{*} \in H^{1}(\T)$ so that
\begin{align*}
&\tang^{n_k} \to \tang^{*} \quad \text{in} \quad C^{\alpha}(\T)\\
&\tang^{n_k}_x \to \tang^{*}_x \quad \text{weakly in} \quad L^{2}(\T) \\
&\|\tang^{*}\|_{H^{1}(\T)} \leq \liminf_{k \to \infty} \, \|\tang^{n_k} \|_{H^{1}(\T)}
\end{align*}
hold for $0 < \alpha < \frac1{2}$ an arbitrary exponent. Thus $\tang^{*}$ has zero mean and unit speed. We therefore infer that
$$
\|A^{-1}_{\tang^{*}}\|_{ {\rm op}} < +\infty, \qquad A_{*} := \limsup_{k \to \infty} \, \|A^{-1}_{\tang^{n_k}}\|_{ {\rm op}} < +\infty.
$$
 By theorem \ref{thm:regularity1} we may then infer that
$$
\limsup_{ k \to \infty}  \| \tang^{n_k} \|_{H^{2}(\T)}
$$
is finite as well. Appealing to standard compactness results once again then shows that, by passing to a further subsequence if necessary, we may assume
\begin{equation}\label{eq:strong}
\tang^{n_k} \to \tang^{*} \quad \text{strongly in} \quad H^{1}(\T)
\end{equation}
without loss of generality. The induced embeddings $\gamma_{\tang^{n_k}}$ therefore converge in $C^{1,\alpha}(\T),$ and so in particular
$$
\delta_{\infty}\left( \gamma_{\tang^{*}} \right) = \lim_{k \to \infty} \, \delta_{\infty}\left( \gamma_{\tang^{n_k}} \right) \leq  \limsup_{ n \to \infty } \, \delta_{\infty}\left( \gamma_{ \tang(T_n)} \right)
$$
is finite by the distortion a-priori bound. Thus $\tang_{*} \in H^{1}(\T)$ induces a bi-Lipschitz embedding, and so we may infer that
\begin{equation}\label{eq:Fcont}
G(\tang^{n_k},\tang^{n_k}_{x}) \to G(\tang^{*},\tang^{*}_{x}) \qquad \text{in} \qquad L^{1}(\T).
\end{equation}
Now apply theorem \ref{thm:local} with $\tang_*$ as initial data to find some $T_0>0$ and a $\sang \in C([0,T_0];H^{1}(\T))$ defining a mild solution. Define $\tang_{ {\rm E}} \in C([0,T_0+T^{*}];H^{1}(\T))$ piecewise by
$$
\tang_{ {\rm E}}(t) = \
\begin{cases}
\tang(t) & \text{if} \quad t \in [0,T^{*}) \\
\sigma(t-T^*) & \text{if} \quad t \in [T^*,T^{*}+T_0],
\end{cases}
$$
and note that \eqref{eq:strong} shows that $\tang_{ {\rm E}} \in C([0,T_0];H^{1}(\T))$ is, in fact, continuous. By using \eqref{eq:Fcont} it is then easy to show that $\tang_{ {\rm E}}$ defines a mild solution with $\tang_0$ as initial data. Thus $T^{*} + T_0 \in \mathcal{C}(\tang_0),$ a contradiction. 
\end{proof}

\end{document}